\documentclass{article}
\usepackage{amsfonts}
\usepackage{amsmath}

\newtheorem{theorem}{Theorem}

\newtheorem{corollary}[theorem]{Corollary}

\newtheorem{lemma}[theorem]{Lemma}

\newtheorem{remark}[theorem]{Remark}

\newenvironment{proof}[1][Proof]{\noindent\textbf{#1.} }{\ \rule{0.5em}{0.5em}}

\begin{document}

\title{Explicit formulas for $p$-adic integrals: approach to $p$-adic distributions and some families of special numbers and polynomials}
\author{Yilmaz Simsek \\
Department of Mathematics, Faculty of Science\\
University of Akdeniz TR-07058 \\
Antalya-Turkey\\
\textbf{E-mail.} ysimsek@akdeniz.edu.tr}
\maketitle

\begin{abstract}
The main objective of this article is to give and classify new formulas of $%
p $-adic integrals and blend these formulas with previously well known
formulas. Therefore, this article gives briefly the formulas of $p$-adic
integrals which were found previously, as well as applying the integral
equations to the generating functions and other special functions, giving
proofs of the new interesting and novel formulas. The $p$-adic integral
formulas provided in this article contain several important well-known
families of special numbers and special polynomials such as the Bernoulli
numbers and polynomials, the Euler numbers and polynomials, the Stirling
numbers, the Lah numbers, the Peters numbers and polynomials, the central
factorial numbers, the Daehee numbers and polynomials, the Changhee numbers
and polynomials, the Harmonic numbers, the Fubini numbers, combinatorial
numbers and sums. In addition, we defined two new sequences containing the
Bernoulli numbers and Euler numbers. These two sequences include central
factorial numbers, Bernoulli numbers and Euler numbers. Some computation
formulas and identities for these sequences are given. Finally we give
further remarks, observations and comments related to content of this paper. 
\newline
\textbf{Keywords:} $p$-adic $q$-integrals, Volkenborn integral, Generating
function, Special functions, Bernoulli numbers and polynomials, Euler
numbers and polynomials, Stirling numbers, Lah numbers, Peters numbers and
polynomials, Central factorial numbers, Daehee numbers and polynomials,
Changhee numbers and polynomials, Harmonic numbers, Fubini numbers,
Combinatorial numbers and sum. \newline
\textbf{MSC(2010):} 11S80, 11B68, 05A15, 05A19, 11M35, 30C15, 26C05, 12D10,
33C45.
\end{abstract}

\section{\textbf{Introduction}}

Just before the turn of the 20th century, around the end of the 19th
century, Kurt Hensel (1861--1941) constructed a new special number family,
now called $p$-adic numbers. Although $p$-adic numbers have been known for
nearly a hundred years, it is well-known that these numbers have recently
been applied in the fields of physics, mathematics and other engineering,
for example, coding theory and Diophantine equations. The mystery and
development of $p$-adic numbers in science is still growing rapidly. In this
way, many scientists continue to be the source of inspiration. Consequently,
these numbers are used in several mathematical fields such as Number Theory,
Algebraic Geometry, Algebraic Topology, Mathematical Analysis, Mathematical
Physics, String Theory, Field Theory, Stochastic Differential Equations on
real Banach Spaces and Manifolds, and other parts of the natural sciences.
In addition to this rapid development of $p$-adic numbers, very important
theories have been built in $p$-adic analysis and their applications. These
include $p$-adic distributions and $p$-adic measure, $p$-adic integrals, $p$%
-adic $L$-function, and other generalized functions. In order to solve
mathematical and physical problems, $p$-adic numbers are used. A connection
between $p$-adic Analysis and Quantum Groups with Noncommutative Geometry, $%
q $-deformation of ordinary analysis has recently given (\textit{cf}. \cite%
{Amice}, \cite{DSkim2}-\cite{Lim}, \cite{MSKIM}-\cite{MSKIMjnt2009}, \cite%
{Park0}, \cite{Robert}, \cite{RyooCHARbernoul}-\cite{Volkenborn}). These
major advances, especially the $p$-adic integral and applications of $p$%
-adic analysis, were greatly influenced. The $p$-adic integrals and their
applications are of great importance to find solutions to special
(differential) equations, solutions to real world problems in both
mathematics, physics and engineering. In recent years, many books,
scientific articles, theses and reports have been published on $p$-adic
integrals and their applications. The $p$-adic integral and generating
functions have been used in mathematics, in mathematical physics and in
others sciences. From another perspective, the $p$-adic integral and $p$%
-adic numbers are also intensively used in the theory of ultrametric
calculus, the $p$-adic quantum mechanics and the $p$-adic mechanics (\textit{%
cf}. \cite{Amice}, \cite{DSkim2}-\cite{Lim}, \cite{MSKIM}-\cite{MSKIMjnt2009}%
, \cite{Park0}, \cite{Robert}, \cite{RyooCHARbernoul}-\cite{Volkenborn}).

This paper is purpose to give comprehensive study of formulas and identities
on theory of not only the Volkenborn integral, but also the fermionic $p$%
-adic integral and also ($p$-adic) distributions. Therefore, in addition to
the well-known formulas including the Volkenborn integral and the fermionic $%
p$-adic integral, which we have found so far in various sources, we will
give new and comprehensive formulas that we have found in this study. We
hope and believe that these formulas have the enough quality and depth to be
used in the fields of $p$-adic numbers and $p$-adic analysis we have
mentioned above. The content of this paper is including a brief summary of
generating functions for special numbers and polynomials, $p$-adic
distribution, the Volkenborn integral and the fermionic $p$-adic integral.
Our new and novel $p$-adic integral formulas involving the Bernoulli numbers
and polynomials, the Euler numbers and polynomials, the Stirling numbers,
the Lah numbers, the Peters numbers and polynomials, the central factorial
numbers, the Daehee numbers and polynomials, the Changhee numbers and
polynomials, the Harmonic numbers, the Fubini numbers, combinatorial numbers
and sums.

The following definitions relations and notations are used all sections of
this article:

Let $\mathbb{N}$, $\mathbb{Z}$, $\mathbb{Q}$, $\mathbb{R}$ and $\mathbb{C}$
denote the set of natural numbers, the set of integers, the set of rational
numbers, the set of real numbers and the set of complex numbers,
respectively. Additionally, let $\mathbb{N}_{0}=\mathbb{N\cup }\left\{
0\right\}$.

Let $n,k\in \mathbb{N}_{0}$, the binomial coefficient $\binom{n}{k}$ is
given by%
\begin{equation*}
\binom{n}{k}=\frac{n!}{k!(n-k)!}.
\end{equation*}%
If $k>n$ or $k<0$, then we assume that%
\begin{equation*}
\binom{n}{k}=0
\end{equation*}%
(\textit{cf}. \cite{Aigner}-\cite{Wang}).

\subsection{\textbf{Rising and Falling Factorials}}

Let $x\in \mathbb{R}$. The rising factorial and the falling factorial are
defined as follows, respectively:%
\begin{equation}
x^{(n)}=\left\{ 
\begin{array}{cc}
x(x+1)(x+2)\ldots (x+n-1) & n\in \mathbb{N} \\ 
1 & n=0%
\end{array}
\right. ,  \label{aii0}
\end{equation}%
and%
\begin{equation}
x_{(n)}=\left\{ 
\begin{array}{cc}
x(x-1)(x-2)\ldots (x-n+1) & n\in \mathbb{N} \\ 
1 & n=0%
\end{array}
\right.  \label{aii0a}
\end{equation}%
For $n\in \mathbb{N}_{0}$, we have%
\begin{equation}
(-1)^{n}\left( -x\right) _{(n)}=(x+n-1)_{(n)}=x^{(n)}  \label{an4}
\end{equation}%
(\textit{cf}. \cite{Aigner}-\cite{Wang}).

In order to define the central factorials of degree $n$, we need the
following another type of falling factorial:%
\begin{equation}
x^{\left[ n\right] }=x(x+\frac{n}{2}-1)(x+\frac{n}{2}-2)\ldots (x-\frac{n}{2}
+1)  \label{acnum1}
\end{equation}%
and%
\begin{equation*}
x^{\left[ 0\right] }=1
\end{equation*}%
where $n\in \mathbb{N}$, $x\in \mathbb{R}$ (\textit{cf}. \cite{Boyadzhiev}, 
\cite{Butzer}, \cite{Charamb}, \cite{Cigler}, \cite{Comtet} , \cite%
{SrivatavaChoi}).

Using (\ref{aii0a}), we have the following well-known identities and
relations: 
\begin{equation}
xx_{(n)}=x_{(n+1)}+nx_{(n)},  \label{Ro}
\end{equation}%
\begin{equation}
x_{(n+1)}=x\sum_{k=0}^{n}(-1)^{n-k}n_{(n-k)}x_{(k)},  \label{IDD-1}
\end{equation}%
(\textit{cf}. \cite[p. 58]{Roman}). Additionally, we have 
\begin{equation}
\left( x+1\right) _{(n+1)}=xx_{(n)}+x_{(n)},  \label{ab6}
\end{equation}
and 
\begin{equation}
x_{(m)}x_{(n)}=\sum\limits_{k=0}^{m}\binom{m}{k}\binom{n}{k}k!x_{(m+n-k)},
\label{LamdaFun-1c}
\end{equation}
(\textit{cf}. \cite{wikiPEDIAfalling}).

It is well-known that the coefficients of $x_{(n+n-k)}$ are called
connection coefficients and they have a combinatorial interpretation as the
number of ways to identify $k$ elements each from a set of size $m$ and a
set of size $n$ (\textit{cf}. \cite{wikiPEDIAfalling}).

The well-known Chu-Vandermonde identity is defined as follows: 
\begin{equation}
\sum\limits_{k=0}^{n}\binom{x}{k}\binom{y}{n-k}=\binom{x+y}{n}.  \label{cv}
\end{equation}%
By (\ref{cv}), we have%
\begin{equation}
\sum\limits_{k=0}^{n}\binom{x}{k}\binom{m}{n-k}=\binom{x+m}{n},
\label{ai0a2}
\end{equation}%
\begin{equation*}
\Delta \binom{x}{n}=\binom{x}{n-1}
\end{equation*}%
and 
\begin{equation*}
\Delta \binom{x}{n}=\binom{x+1}{n}-\binom{x}{n}.
\end{equation*}%
Therefore 
\begin{equation}
\binom{x+1}{n}=\binom{x}{n}+\binom{x}{n-1}  \label{v1-A}
\end{equation}%
(\textit{cf}. \cite[p. 69, Eq-(7)]{Jordan}). By%
\begin{equation}
\Delta x_{(n)}=\left( x+1\right) _{(n)}-x_{(n)},  \label{FF-11a}
\end{equation}%
we have 
\begin{equation}
\left( x+1\right) _{(n)}=x_{(n)}+nx_{(n-1)}  \label{ab6a}
\end{equation}%
(\textit{cf}. \cite[p. 58]{Roman}).

Thanks to the works \cite{GouldVol3} and \cite{GouldV7} of Gould, it is also
known that the following formulas hold true: 
\begin{equation}
x\binom{x-2}{n-1}=\sum\limits_{k=1}^{n}(-1)^{k-n}\binom{x}{k}k  \label{Gg1}
\end{equation}%
(\textit{cf}. \cite[Vol. 3, Eq-(4.20)]{GouldVol3}), 
\begin{equation}
\binom{n-x}{n}=\sum\limits_{k=0}^{n}(-1)^{k-n}\binom{x}{k}  \label{Gg2}
\end{equation}%
(\textit{cf}. \cite[Vol. 3, Eq-(4.19)]{GouldVol3}), 
\begin{equation}
\binom{mx}{n}=\sum_{k=0}^{n}\binom{x}{k}\sum_{j=0}^{k}\left( -1\right) ^{j} 
\binom{k}{j}\binom{mk-mj}{n}  \label{Id-7}
\end{equation}%
(\textit{cf}. \cite[Eq-(2.65)]{GouldV7}), 
\begin{equation}
\binom{x}{n}^{r}=\sum_{k=0}^{nr}\binom{x}{k}\sum_{j=0}^{k}\left( -1\right)
^{j}\binom{k}{j}\binom{k-j}{n}^{r}  \label{Id-5}
\end{equation}%
(\textit{cf}. \cite[Eq-(2.66)]{GouldV7}). 
\begin{equation}
x\binom{x-2}{n-1}+x\left( x-1\right) \binom{n-3}{n-2}=\sum_{k=0}^{n}\left(
-1\right) ^{k}\binom{x}{k}k^{2},  \label{Id-6}
\end{equation}%
where $n\in \mathbb{N}$ with $n>1$ (\textit{cf}. \cite[Eq-(2.15)]{GouldV7}), 
\begin{equation}
\binom{x+n}{n}=\sum_{k=0}^{n}\binom{x}{k}\sum_{j=0}^{k}\left( -1\right) ^{j} 
\binom{k}{j}\binom{k-j+n}{n}  \label{Id-1a}
\end{equation}%
and 
\begin{equation}
\binom{x+n}{n}=\sum_{k=0}^{n}x^{k}\sum_{j=0}^{n}\binom{n}{j}\frac{S_{1}(j,k) 
}{j!}  \label{Id-2b}
\end{equation}%
(\textit{cf}. \cite[Eq-(2.64) and Eq-(6.17)]{GouldV7}),%
\begin{equation}
\binom{x+n+\frac{1}{2}}{n}=\left( 2n+1\right) \binom{2n}{n}
\sum\limits_{k=0}^{n}\binom{n}{k}\binom{x}{k}\frac{2^{2k-2n}}{\left(
2k+1\right) \binom{2k}{k}}  \label{1BIaa}
\end{equation}%
(cf. \cite[Vol. 3, Eq-(6.26)]{GouldVol3}), 
\begin{equation}
x\binom{x-2}{n-1}=\sum\limits_{k=1}^{n}(-1)^{k-n}\binom{x}{k}k  \label{1BIb3}
\end{equation}%
(\textit{cf}. \cite[Vol. 3, Eq-(4.20)]{GouldVol3}) and 
\begin{equation}
(-1)^{n}\binom{n-x}{n}=\sum\limits_{k=1}^{n}(-1)^{k}\binom{x}{k}
\label{1BIb4}
\end{equation}%
(\textit{cf}. \cite[Vol. 3, Eq-(4.19)]{GouldVol3}).

\subsection{\textbf{Generating Functions for Special Numbers and Polynomials}%
}

Here, we give some well-known generating functions which are for special
numbers and polynomials.

The Apostol-Bernoulli polynomials $\mathcal{B}_{n}(x;\lambda )$ are defined
by means of the following generating function:%
\begin{equation}
F_{A}(t,x;\lambda )=\frac{t}{\lambda e^{t}-1}e^{tx}=\sum_{n=0}^{\infty } 
\mathcal{B}_{n}(x;\lambda )\frac{t^{n}}{n!},  \label{Ap.B}
\end{equation}%
(\textit{cf}. \cite{apostol}).

Substituting $x=0$ into (\ref{Ap.B}), we have%
\begin{equation*}
\lambda \mathcal{B}_{1}(1;\lambda )=1+\mathcal{B}_{1}(\lambda )
\end{equation*}%
and for $n\geq 2$,%
\begin{equation*}
\lambda \mathcal{B}_{n}(1;\lambda )=\mathcal{B}_{n}(\lambda ),
\end{equation*}%
(\textit{cf}. \cite{apostol}).

By using (\ref{Ap.B}), we have%
\begin{equation}
\mathcal{B}_{n}\left( x;\lambda \right) =\sum\limits_{j=0}^{n}\binom{n}{j}
x^{n-j}\mathcal{B}_{j}\left( \lambda \right)  \label{ABP1}
\end{equation}%
(\textit{cf}. \cite{apostol}). By using (\ref{ABP1}), first few values of
the Apostol-Bernoulli polynomials are given as follows:%
\begin{eqnarray*}
\mathcal{B}_{0}\left( x;\lambda \right) &=&0, \\
\mathcal{B}_{1}\left( x;\lambda \right) &=&\frac{1}{\lambda -1}, \\
\mathcal{B}_{2}\left( x;\lambda \right) &=&\frac{1}{\lambda -1}x-\frac{
2\lambda }{\left( \lambda -1\right) ^{2}}, \\
\mathcal{B}_{3}\left( x;\lambda \right) &=&\frac{3}{\lambda -1}x^{2}-\frac{
6\lambda }{\left( \lambda -1\right) ^{2}}x+\frac{3\lambda \left( \lambda
+1\right) }{\left( \lambda -1\right) ^{3}}, \\
\mathcal{B}_{4}\left( x;\lambda \right) &=&\frac{4}{\lambda -1}x^{3}-\frac{
12\lambda }{\left( \lambda -1\right) ^{2}}x^{2}+\frac{12\lambda \left(
\lambda +1\right) }{\left( \lambda -1\right) ^{3}}x-\frac{4\lambda \left(
\lambda ^{2}+4\lambda +1\right) }{\left( \lambda -1\right) ^{4}}, \\
\mathcal{B}_{5}\left( x;\lambda \right) &=&\frac{5}{\lambda -1}x^{4}-\frac{
20\lambda }{\left( \lambda -1\right) ^{2}}x^{3}+\frac{30\lambda \left(
\lambda +1\right) }{\left( \lambda -1\right) ^{3}}x^{2}-\frac{20t\left(
\lambda ^{2}+4\lambda +1\right) }{\left( \lambda -1\right) ^{4}}x \\
&&+\frac{5\lambda \left( \lambda ^{3}+11\lambda ^{2}+11\lambda +1\right) }{
\left( \lambda -1\right) ^{5}}.
\end{eqnarray*}

Substituting $x=1$ into (\ref{Ap.B}), we have the following
Apostol-Bernoulli numbers:%
\begin{equation*}
\mathcal{B}_{n}(1,\lambda )=\sum\limits_{j=0}^{n}\binom{n}{j}\mathcal{B}
_{n}(\lambda ),
\end{equation*}%
where%
\begin{equation*}
\mathcal{B}_{n}(\lambda )=\mathcal{B}_{n}(0,\lambda )
\end{equation*}%
and 
\begin{equation*}
\mathcal{B}_{0}(\lambda )=0.
\end{equation*}%
Since 
\begin{eqnarray}
\mathcal{B}_{1}\left( \lambda \right) &=&\frac{1}{\lambda -1},  \notag \\
\mathcal{B}_{m}\left( \lambda \right) &=&\frac{\lambda }{1-\lambda }
\sum\limits_{j=0}^{m-1}\binom{m}{j}\mathcal{B}_{j}(\lambda ),  \label{ABP2a}
\end{eqnarray}%
we have the following few values of the Apostol-Bernoulli numbers: 
\begin{eqnarray*}
\mathcal{B}_{2}\left( \lambda \right) &=&\frac{-2\lambda }{\left( \lambda
-1\right) ^{2}}, \\
\mathcal{B}_{3}\left( \lambda \right) &=&\frac{3\lambda \left( \lambda
+1\right) }{\left( \lambda -1\right) ^{3}}, \\
\mathcal{B}_{4}\left( \lambda \right) &=&\frac{-4\lambda \left( \lambda
^{2}+4\lambda +1\right) }{\left( \lambda -1\right) ^{4}}, \\
\mathcal{B}_{5}\left( \lambda \right) &=&\frac{5\lambda \left( \lambda
^{3}+11\lambda ^{2}+11\lambda +1\right) }{\left( t-1\right) ^{5}}, \\
\mathcal{B}_{6}\left( \lambda \right) &=&\frac{-6\lambda \left( \lambda
^{4}+26\lambda ^{3}+66\lambda ^{2}+26\lambda +1\right) }{\left( t-1\right)
^{6}}, \\
\mathcal{B}_{7}\left( \lambda \right) &=&\frac{7\lambda \left( \lambda
^{5}+57\lambda ^{4}+302\lambda ^{3}+302\lambda ^{2}+57\lambda +1\right) }{
\left( \lambda -1\right) ^{7}},\ldots
\end{eqnarray*}%
(\textit{cf}. \cite{apostol}, for detail, see also \cite{Kim2006TMIC}, \cite%
{KIMjang}, \cite{Luo}, \cite{Srivastava2011}, \cite{srivas18}). When $%
\lambda =1$ in (\ref{Ap.B}), we have the Bernoulli polynomials of the first
kind%
\begin{equation*}
B_{n}(x)=\mathcal{B}_{n}(x;1).
\end{equation*}%
Hence, few values of the Bernoulli polynomials are given as follows:%
\begin{eqnarray*}
B_{0}\left( x\right) &=&1, \\
B_{1}\left( x\right) &=&x-\frac{1}{2}, \\
B_{2}\left( x\right) &=&x^{2}-x+\frac{1}{6}, \\
B_{3}\left( x\right) &=&x^{3}-\frac{3}{2}x^{2}+\frac{1}{2}x, \\
B_{4}\left( x\right) &=&x^{4}-2x^{3}+x^{2}-\frac{1}{30}, \\
B_{5}\left( x\right) &=&x^{5}-\frac{5}{2}x^{4}+\frac{5}{3}x^{3}-\frac{1}{6}x,
\\
B_{6}\left( x\right) &=&x^{6}-3x^{5}+\frac{5}{2}x^{4}-\frac{1}{2}x^{2}+\frac{
1}{42},
\end{eqnarray*}%
Since $B_{n}=B_{n}(0)$ denotes the Bernoulli numbers of the first kind, few
of these numbers are given as follows:%
\begin{eqnarray*}
B_{0} &=&1,B_{1}=-\frac{1}{2},B_{2}=\frac{1}{6},B_{3}=0,B_{4}=-\frac{1}{30},
\\
B_{6} &=&\frac{1}{42},B_{8}=-\frac{1}{30},B_{10}=\frac{5}{66},B_{12}=-\frac{
691}{2730},B_{14}=\frac{7}{6}, \\
B_{16} &=&-\frac{3617}{510},B_{18}=\frac{43867}{798},B_{20}=-\frac{174611}{
330},\ldots
\end{eqnarray*}%
with $B_{2n+1}=0$ for $n\geq 2$ (\textit{cf}. OEIS A000367, OEIS A002445;
and also see \cite{Bayad}-\cite{Wang}; see also the references cited in each
of these works).

The $\lambda $-Bernoulli polynomials (Apostol-type Bernoulli) polynomials $%
\mathfrak{B}_{n}(x;\lambda )$ are defined by means of the following
generating function: 
\begin{equation}
F_{B}(t,x;\lambda )=\frac{\log \lambda +t}{\lambda e^{t}-1}
e^{tx}=\sum_{n=0}^{\infty }\mathfrak{B}_{n}(x;\lambda )\frac{t^{n}}{n!},
\label{laBN}
\end{equation}%
(\textit{cf}. \cite{TkimJKMS}; see also \cite{jandY1}, \cite{srivas18}, \cite%
{simsek2017ascm}). For $n>1$, combining (\ref{laBN}) with (\ref{Ap.B}), we
have the following well-known identity:%
\begin{equation*}
\mathfrak{B}_{n-1}(x;\lambda )=\frac{\log \lambda }{n}\mathcal{B}
_{n}(x;\lambda )+\mathcal{B}_{n-1}(x;\lambda ).
\end{equation*}

The Apostol-Euler polynomials of the first kind $\mathcal{E}_{n}(x,\lambda )$
are defined by means of the following generating function: 
\begin{equation}
F_{P1}(t,x;k,\lambda )=\frac{2}{\lambda e^{t}+1}e^{tx}=\sum_{n=0}^{\infty } 
\mathcal{E}_{n}(x,\lambda )\frac{t^{n}}{n!},  \label{Cad3}
\end{equation}%
and by using (\ref{Cad3}), we have%
\begin{equation}
\mathcal{E}_{n}(x;\lambda )=\sum\limits_{j=0}^{n}\binom{n}{j}x^{n-j}\mathcal{%
\ E}_{j}(\lambda )  \label{AAEN}
\end{equation}%
(\textit{cf}. \cite{Bayad}-\cite{SrivastavaLiu}).

By combining (\ref{Cad3}) with (\ref{Ap.B}), we have the following
well-known relation:%
\begin{equation}
\mathcal{E}_{n}\left( x;\lambda \right) =-\frac{2}{n+1}\mathcal{B}
_{n+1}\left( x;-\lambda \right)  \label{RelationApostolEnBn}
\end{equation}%
(\textit{cf}. \cite{SrivatavaChoi}).

Substituting $\lambda =1$ into (\ref{Cad3}), we have the Euler polynomials
of the first kind; that is%
\begin{equation*}
E_{n}\left( x\right) =\mathcal{E}_{n}\left( x;1\right).
\end{equation*}%
In the light of this thought and also with the help of the equation (\ref%
{AAEN}), few values of the Euler polynomials of the first kind are given as
follows: 
\begin{eqnarray*}
E_{0}\left( x\right) &=&1, \\
E_{1}\left( x\right) &=&x-\frac{1}{2}, \\
E_{2}\left( x\right) &=&x^{2}-x, \\
E_{3}\left( x\right) &=&x^{3}-\frac{3}{2}x^{2}+\frac{1}{4}, \\
E_{4}\left( x\right) &=&x^{4}-2x^{3}+x, \\
E_{5}\left( x\right) &=&x^{5}-\frac{5}{2}x^{4}+\frac{5}{2}x^{2}-\frac{1}{2},
\end{eqnarray*}

Substituting $x=0$ into (\ref{Cad3}), we have the Apostol-Euler numbers of
the first kind:%
\begin{equation*}
\mathcal{E}_{n}(\lambda )=\mathcal{E}_{n}(0,\lambda ).
\end{equation*}%
Hence%
\begin{eqnarray}
\mathcal{E}_{0}\left( \lambda \right) &=&\frac{2}{\lambda +1},  \notag \\
\mathcal{E}_{m}\left( \lambda \right) &=&-\frac{\lambda }{1+\lambda }
\sum\limits_{j=0}^{m-1}\left( 
\begin{array}{c}
m \\ 
j%
\end{array}
\right) \mathcal{E}_{j}(\lambda ).  \label{AEP2a}
\end{eqnarray}%
Using (\ref{AEP2a}), we have%
\begin{eqnarray*}
\mathcal{E}_{1}\left( \lambda \right) &=&-\frac{2\lambda }{\left( \lambda
+1\right) ^{2}}, \\
\mathcal{E}_{2}\left( \lambda \right) &=&\frac{2\lambda \left( \lambda
-1\right) }{\left( \lambda +1\right) ^{3}}, \\
\mathcal{E}_{3}\left( \lambda \right) &=&-\frac{2\lambda \left( \lambda
^{2}-4\lambda +1\right) }{\left( \lambda +1\right) ^{4}}\ldots .
\end{eqnarray*}%
Setting $\lambda =1$ into (\ref{Cad3}), we have the Euler numbers of the
first kind:%
\begin{equation*}
E_{n}=\mathcal{E}_{n}^{(1)}(1)=E_{n}(0).
\end{equation*}%
Hence few of values of the Euler numbers of the first kind are given as
follows:%
\begin{eqnarray*}
E_{0} &=&1,E_{1}=-\frac{1}{2},E_{2}=0,E_{3}=\frac{1}{4}, \\
E_{5} &=&-\frac{1}{2},E_{7}=\frac{17}{8},E_{9}=-\frac{31}{2}\ldots
\end{eqnarray*}%
with $E_{2n}=0$ for $n\geq 1$ (\textit{cf}. \cite{Bayad}-\cite{SrivastavaLiu}%
; see also the references cited in each of these earlier works).

Let $u$ be a complex numbers with $u\neq 1$. The Frobenius-Euler numbers $%
H_{n}(u)$ are defined by means of the following generating function:%
\begin{equation}
F_{f}(t,u)=\frac{1-u}{e^{t}-u}=\sum_{n=0}^{\infty }H_{n}(u)\frac{t^{n}}{n!}.
\label{aii2}
\end{equation}%
By using (\ref{aii2}), we have%
\begin{equation*}
H_{n}(u)=\left\{ 
\begin{array}{cc}
1 & \text{for }n=0 \\ 
\frac{1}{u}\sum\limits_{j=0}^{n}\left( 
\begin{array}{c}
n \\ 
j%
\end{array}
\right) H_{j}(u) & \text{for }n>0.%
\end{array}
\right.
\end{equation*}%
By using the above formula, some values of the numbers $H_{n}(u)$ are given
as follows: 
\begin{eqnarray*}
H_{1}(u) &=&\frac{1}{u-1}, \\
H_{2}(u) &=&\frac{u+1}{\left( u-1\right) ^{2}}, \\
H_{3}(u) &=&\frac{u^{2}+4u+1}{\left( u-1\right) ^{3}}, \\
H_{4}(u) &=&\frac{u^{3}+11u^{2}+11u+1}{\left( u-1\right) ^{4}},\ldots
\end{eqnarray*}%
Substituting $u=-1$ into (\ref{aii2}), we have%
\begin{equation*}
E_{n}=H_{n}(-1)
\end{equation*}%
(\textit{cf}. \cite{DSKIMfrob}, \cite[Theorem 1, p. 439]{TkimJKMS}, \cite%
{srivas18}; see also the references cited in each of these earlier works).

The Euler numbers of the second kind $E_{n}^{\ast }$ are given by means of
the following generating function:%
\begin{equation*}
\frac{2}{e^{t}+e^{-t}}=\sum_{n=0}^{\infty }E_{n}^{\ast }\frac{t^{n}}{n!}.
\end{equation*}%
Since%
\begin{equation*}
E_{n}^{\ast }=2^{n}E_{n}\left( \frac{1}{2}\right) ,
\end{equation*}%
and using the definition of the Euler polynomials the first kind $%
E_{n}\left( x\right) $, it is easy to give few values of the Euler numbers
of the second kind $E_{n}^{\ast }$ given as follows:%
\begin{eqnarray*}
E_{0}^{\ast } &=&1,E_{2}^{\ast }=-1,E_{4}^{\ast }=5,E_{6}^{\ast
}=-61,E_{8}^{\ast }=1385,E_{10}^{\ast }=-50521, \\
E_{12}^{\ast } &=&2702765,E_{14}^{\ast }=-199360981,E_{16}^{\ast
}=19391512145,\ldots
\end{eqnarray*}
with $E_{2n+1}^{\ast }=0$ for $n\geq 0$ (\textit{cf}. \cite{Bayad}-\cite%
{SrivastavaLiu}; see also the references cited in each of these earlier
works).

The Bernstein basis functions are defined by means of the following
generating functions:%
\begin{equation*}
\frac{(tx)^{k}}{k!}e^{(1-x)t}=\sum\limits_{n=0}^{\infty }B_{k}^{n}(x)
\end{equation*}%
where%
\begin{equation}
B_{k}^{n}(x)=\binom{n}{k}x^{k}(1-x)^{n-k},  \label{A.Berns.}
\end{equation}%
$n,k\in \mathbb{N}_{0}$ with $0\leq k\leq n$ (\textit{cf}. \cite{mehmet}, 
\cite{KimRyooBERNSTEIN}, \cite{Lorenz}, \cite{AAM.Ackgoz}, \cite%
{simsekFPTABERN.}).

Note that there is one generating function for each value of $k$.

The Stirling numbers of the first kind $S_{1}(n,k)$ are defined by means of
the following generating function: 
\begin{equation}
F_{S1}(t,k)=\frac{\left( \log (1+t)\right) ^{k}}{k!}=\sum_{n=0}^{\infty
}S_{1}(n,k)\frac{t^{n}}{n!}.  \label{S1}
\end{equation}%
Some basic properties of these numbers are given as follows: 
\begin{equation*}
S_{1}(0,0)=1
\end{equation*}
and 
\begin{equation*}
S_{1}(0,k)=0
\end{equation*}%
if $k>0$. Also 
\begin{equation*}
S_{1}(n,0)=0
\end{equation*}%
if $n>0$ and 
\begin{equation*}
S_{1}(n,k)=0
\end{equation*}%
if $k>n$. A recurrence relation for these numbers is given by%
\begin{equation}
S_{1}(n+1,k)=-nS_{1}(n,k)+S_{1}(n,k-1)  \label{S2-1c}
\end{equation}%
(\textit{cf}. \cite{Charamb}, \cite{Bayad}, \cite{Chan}, \cite{Roman}, \cite%
{SimsekFPTA}, \cite{AM2014}; and see also the references cited in each of
these earlier works).

By using (\ref{S2-1c}), few values of the Stirling numbers of the first kind 
$S_{1}(n,k)$ are given by the following table: 
\begin{equation*}
\begin{array}{ccccccc}
n\backslash k & 0 & 1 & 2 & 3 & 4 & 5 \\ 
0 & 1 & 0 & 0 & 0 & 0 & 0 \\ 
1 & 0 & 1 & 0 & 0 & 0 & 0 \\ 
2 & 0 & -1 & 1 & 0 & 0 & 0 \\ 
3 & 0 & 2 & -3 & 1 & 0 & 0 \\ 
4 & 0 & -6 & 11 & -6 & 1 & 0 \\ 
5 & 0 & 24 & -50 & 35 & -10 & 1%
\end{array}%
\end{equation*}

Another generating function for the Stirling numbers of the first kind is
falling factorial function which is given as follows:%
\begin{equation}
x_{(n)}=\sum_{k=0}^{n}S_{1}(n,k)x^{k}  \label{S1a}
\end{equation}%
(\textit{cf}. \cite{Charamb}, \cite{Cigler}, \cite{Comtet}, \cite%
{SrivatavaChoi}).

Some well-known identities for the equation (\ref{S1a}) are given as follows:

Multiplying both sides of the equation (\ref{S1a}) by $x^{m}$, we have%
\begin{equation}
x^{m}x_{(n)}=\sum\limits_{k=0}^{n}S_{1}(n,k)x^{m+k}.  \label{1BIb2}
\end{equation}

By combining (\ref{LamdaFun-1c}) with (\ref{S1a}), we have%
\begin{equation}
x_{(m)}x_{(n)}=\sum\limits_{k=0}^{m}\binom{m}{k}\binom{n}{k}
k!\sum\limits_{l=0}^{m+n-k}S_{1}(m+n-k,l)x^{l}.  \label{1BIb}
\end{equation}%
By using (\ref{S1a}), we have%
\begin{equation}
x_{(m)}x_{(n)}=\sum_{j=0}^{n}\sum_{l=0}^{m}S_{1}(n,k)S_{1}(m,l)x^{j+l}.
\label{1BIb1}
\end{equation}
By combining (\ref{Ro}) with (\ref{S1a}), we have%
\begin{equation}
xx_{(n)}=\sum_{k=0}^{n}\left( S_{1}(n+1,k)+nS_{1}(n,k)\right) x^{k}+x^{n+1}.
\label{LamdaFun-A}
\end{equation}%
By combining the above equation with (\ref{S2-1c}), and using $S_{1}(n,k)=0$
if $k<0$, we have%
\begin{equation}
xx_{(n)}=\sum_{k=1}^{n}S_{1}(n,k-1)x^{k}+x^{n+1}.  \label{ai0a1}
\end{equation}

The unsigned Stirling numbers of the first kind are defined by%
\begin{equation*}
C(n,k)=\left\vert S_{1}(n,k)\right\vert=\left[ 
\begin{array}{c}
k \\ 
n%
\end{array}
\right]
\end{equation*}%
(\textit{cf}. \cite{Charamb}, \cite{Cigler}, \cite{Comtet}, \cite%
{SrivatavaChoi}). The numbers $C(n,k)$ are also defined as follows:%
\begin{equation}
x^{(n)}=\sum_{k=0}^{n}C(n,k)x^{k}  \label{AY-2}
\end{equation}%
(\textit{cf}. \cite{C. A. CharalambidesDISCRETE}).

The Bernoulli polynomials of the second kind $b_{n}(x)$ are defined by means
of the following generating function: 
\begin{equation}
F_{b2}(t,x)=\frac{t}{\log (1+t)}(1+t)^{x}=\sum_{n=0}^{\infty }b_{n}(x)\frac{
t^{n}}{n!}  \label{Br2}
\end{equation}
(\textit{cf}. \cite[pp. 113-117]{Roman}; see also the references cited in
each of these earlier works).

The Bernoulli numbers of the second kind $b_{n}(0)$ are defined by means of
the following generating function: 
\begin{equation}
F_{b2}(t)=\frac{t}{\log (1+t)}=\sum_{n=0}^{\infty }b_{n}(0)\frac{t^{n}}{n!}.
\label{Be-1t}
\end{equation}%
The Bernoulli polynomials of the second kind are defined by 
\begin{equation*}
b_{n}(x)=\int_{x}^{x+1}u_{(n)}du.
\end{equation*}%
Substituting $x=0$ into the above equation, one has 
\begin{equation}
b_{n}(0)=\int_{0}^{1}u_{(n)}du.  \label{LamdaFun-1p}
\end{equation}%
The numbers $b_{n}(0)$ are also so-called Cauchy numbers (i.e. Bernoulli
numbers of the second kind) (\textit{cf}. \cite[p. 116]{Roman}, \cite%
{TKimTAKAO}, \cite{Merlini}, \cite{ysimsek Ascm}; see also the references
cited in each of these earlier works).

The $\lambda $-array polynomials $S_{k}^{n}(x;\lambda )$ are defined by
means of the following generating function: 
\begin{equation}
F_{A}(t,x,k;\lambda )=\frac{\left( \lambda e^{t}-1\right) ^{k}}{k!}
e^{tx}=\sum_{n=0}^{\infty }S_{k}^{n}(x;\lambda )\frac{t^{n}}{n!},
\label{ARY-1}
\end{equation}%
where $k\in \mathbb{N}_{0}$ and $\lambda \in \mathbb{C}$ (\textit{cf}. \cite%
{SimsekFPTA}, for detail, see also \cite{Bayad}, \cite{Chan}, \cite{AM2014}).

By (\ref{ARY-1}), we have%
\begin{equation}
S_{k}^{n}(x;\lambda )=\frac{1}{k!}\sum_{j=0}^{k}(-1)^{k-j}\binom{k}{j}
\lambda ^{j}(j+x)^{n}  \label{AaRY.1}
\end{equation}%
(\textit{cf}. \cite{SimsekFPTA}). Substituting $x=0$ into (\ref{ARY-1}), we
have the $\lambda $-Stirling numbers $S_{2}(n,k;\lambda )$, which are
defined by the following generating function: 
\begin{equation}
F_{S}(t,k;\lambda )=\frac{\left( \lambda e^{t}-1\right) ^{k}}{k!}
=\sum_{n=0}^{\infty }S_{2}(n,k;\lambda )\frac{t^{n}}{n!},  \label{SN-1}
\end{equation}%
where $k\in \mathbb{N}_{0}$ and $\lambda \in \mathbb{C}$ (\textit{cf}. \cite%
{Luo}, \cite{Srivastava2011}, see also \cite{SimsekFPTA}).

Substituting $\lambda =1$ into (\ref{SN-1}), then we get the Stirling
numbers of the second kind, which is the number of partitions of a set of $n$
elements into $k$ nonempty subsets, as follows: 
\begin{equation}
S_{2}(n,k)=S_{2}(n,k;1)=\frac{1}{k!}\sum_{j=0}^{k}(-1)^{k-j}\binom{k}{j}
\lambda ^{j}j^{n}.  \label{AaS2}
\end{equation}%
The Stirling numbers of the second kind are also given by the following
generating function including falling factorial: 
\begin{equation}
x^{n}=\sum\limits_{k=0}^{n}S_{2}(n,k)x_{(k)},  \label{S2-1a}
\end{equation}%
(\textit{cf}. \cite{apostol}-\cite{SrivastavaLiu}; see also the references
cited in each of these earlier works).

By using (\ref{AaS2}), few values of the Stirling numbers of the second kind 
$S_{2}(n,k)$ are given by the following table: 
\begin{equation*}
\begin{array}{ccccccc}
n\backslash k & 0 & 1 & 2 & 3 & 4 & 5 \\ 
0 & 1 & 0 & 0 & 0 & 0 & 0 \\ 
1 & 0 & 1 & 0 & 0 & 0 & 0 \\ 
2 & 0 & 1 & 1 & 0 & 0 & 0 \\ 
3 & 0 & 1 & 3 & 1 & 0 & 0 \\ 
4 & 0 & 1 & 7 & 6 & 1 & 0 \\ 
5 & 0 & 1 & 15 & 25 & 10 & 1%
\end{array}%
\end{equation*}

The Schlomilch formula, associated with the Stirling numbers of the first
and the second kind, is given by 
\begin{equation*}
S_{1}(n,k)=\sum_{j=0}^{n-k}(-1)^{j}\binom{n+j-1}{k-1}\binom{2n-k}{n-k-j}
S_{2}(n-k+j,j)
\end{equation*}%
(\textit{cf}. \cite[p. 115]{C. A. CharalambidesDISCRETE}, \cite[p. 290,
Eq-(8.21)]{Charamb}).

In \cite{Osgood}, Osgood and Wu gave the following identity: 
\begin{equation}
(xy)_{(k)}=\sum\limits_{l,m=1}^{k}C_{l,m}^{(k)}x_{(l)}x_{(m)}
\label{LamdaFun-1v}
\end{equation}%
where%
\begin{equation*}
C_{l,m}^{(k)}=\sum\limits_{j=1}^{k}(-1)^{k-j}S_{1}(k,j)S_{2}(j,l)S_{2}(j,m)
\end{equation*}%
$C_{l,m}^{(k)}=C_{m,l}^{(k)}$, $C_{1,1}^{(1)}=1$, $C_{1,1}^{(2)}=0$, $%
C_{1,2}^{(3)}=0=C_{2,1}^{(3)}$.

By using (\ref{S2-1a}), we have%
\begin{equation}
(xy)_{(k)}=\sum\limits_{m=0}^{k}S_{1}(k,m)x^{m}y^{m}.  \label{LamdaFun-1w}
\end{equation}

The Lah numbers are defined by means of the following generating function: 
\begin{equation}
F_{L}(t,k)=\frac{1}{k!}\left( \frac{t}{1-t}\right) ^{k}=\sum_{n=k}^{\infty
}L(n,k)\frac{t^{n}}{n!}  \label{La}
\end{equation}%
(\textit{cf}. \cite{C. A. CharalambidesDISCRETE}, \cite[p. 44]{riARDON}, 
\cite{Belbechir}, \cite{QiLAH}, \cite{Charamb}, \cite{Comtet}, \cite{QiLAH}, 
\cite{WikipeLAH}, and the references cited therein).

Using (\ref{La}), we have%
\begin{equation}
L(n,k)=(-1)^{n}\frac{n!}{k!}\binom{n-1}{k-1}.  \label{LAH-1a}
\end{equation}%
The unsigned Lah numbers are defined by%
\begin{equation}
\left\vert L(n,k)\right\vert =\frac{n!}{k!}\binom{n-1}{k-1},  \label{alah}
\end{equation}%
where $n,k\in \mathbb{N}$ with $1\leq k\leq n$.

By using (\ref{alah}), few values of the unsigned Lah numbers $\left\vert
L(n,k)\right\vert $ are given by the following table: 
\begin{equation*}
\begin{array}{ccccccc}
n\backslash k & 0 & 1 & 2 & 3 & 4 & 5 \\ 
0 & 1 & 0 & 0 & 0 & 0 & 0 \\ 
1 & 0 & 1 & 0 & 0 & 0 & 0 \\ 
2 & 0 & 2 & 1 & 0 & 0 & 0 \\ 
3 & 0 & 6 & 6 & 1 & 0 & 0 \\ 
4 & 0 & 24 & 36 & 12 & 1 & 0 \\ 
5 & 0 & 120 & 240 & 120 & 20 & 1%
\end{array}%
\end{equation*}

By the help of the following the initial conditions 
\begin{equation*}
L(n,0)=\delta _{n,0}
\end{equation*}%
and 
\begin{equation*}
L(0,k)=\delta _{0,k},
\end{equation*}%
for all $k,n\in \mathbb{N}$, we have recurrence relations for the Lah
numbers given as follows: 
\begin{equation*}
L(n+1,k)=-(n+k)L(n,k)-L(n,k-1)
\end{equation*}
and 
\begin{equation*}
L(n,k)=\sum_{j=0}^{n}(-1)^{j}S_{1}(n,j)S_{2}(j,k)
\end{equation*}%
(\textit{cf}. \cite{Garsia}, \cite[p. 44]{riARDON}, \cite{QiLAH}).

Using (\ref{an4}), we have another definition of the Lah numbers including
the falling factorial and the rising factorial:%
\begin{equation}
\left( -x\right) _{(n)}=\sum_{k=1}^{n}L(n,k)x_{(k)}  \label{Lah}
\end{equation}%
so that 
\begin{equation*}
x_{(n)}=\sum_{k=1}^{n}L(n,k)\left( -x\right) _{(k)}
\end{equation*}%
and 
\begin{equation}
x^{(n)}=\sum_{k=1}^{n}\left\vert L(n,k)\right\vert x_{(k)}.  \label{LahLAH}
\end{equation}%
where $n\in \mathbb{N}_{0}$ (\textit{cf}. \cite{C. A. CharalambidesDISCRETE}%
, \cite{Charamb}, \cite{Comtet}, \cite{DSKimTkimASCM2019}, \cite{Garsia}, 
\cite{QiLAH}, \cite{riARDON}, \cite{WikipeLAH}).

The equation (\ref{acnum1}) classification enables us to give the following
central factorials of degree $n$, $t(n,k)$ and $T(n,k)$ of the first and the
second kind, respectively:%
\begin{equation}
x^{\left[ n\right] }=\sum_{k=0}^{n}t(n,k)x^{k}  \label{acnum1t}
\end{equation}%
and%
\begin{equation}
x^{n}=\sum_{k=0}^{n}T(n,k)x^{\left[ k\right] }  \label{acnum1T}
\end{equation}%
with%
\begin{equation*}
t(n,0)=T(n,0)=\delta _{0n}
\end{equation*}%
where $\delta _{mn}$ denotes the Kronecker delta and $n\in \mathbb{N}_{0}$ (%
\textit{cf}. \cite{Butzer}).

Observe that 
\begin{equation*}
D^{j}\left\{ x^{\left[ n\right] }\right\} =j!\sum_{k=j}^{n}\binom{k}{j}
t(n,k)x^{k-j},
\end{equation*}
\begin{equation*}
\delta ^{j}\left\{ x^{n}\right\} =j!\sum_{k=j}^{n}\binom{k}{j}T(n,k)x^{\left[
k-j\right] },
\end{equation*}%
\begin{equation*}
\delta ^{j}\left\{ x^{n}\right\} \left\vert _{x=0}\right. =j!T(n,j),
\end{equation*}%
and%
\begin{equation*}
D^{j}\left\{ x^{\left[ n\right] }\right\} \left\vert _{x=0}\right. =j!t(n,j)
\end{equation*}%
where%
\begin{equation*}
\delta \left\{ f(x)\right\} =f\left( x+\frac{1}{2}\right) -f\left( x-\frac{1 
}{2}\right)
\end{equation*}%
and%
\begin{equation*}
D=\frac{d}{dx}
\end{equation*}%
(\textit{cf}. \cite[Eq. (2.7), Eq. (2.9)]{Butzer}, \cite{Kim2018PJMS}, \cite%
{AM2014}).

Applying the Cauchy's integral theorem to the function $\left( \sinh \left( 
\frac{z}{2}\right) \right) ^{m}$, we have the well-known integral
representations for the numbers $T(n,k)$ and $t(n,k)$, respectively, given
as follows:%
\begin{equation*}
T(n,k)=\frac{k!}{m!2\pi i}\int\limits_{\left\vert w\right\vert =r}\left(
2\sinh \left( \frac{z}{2}\right) \right) ^{m}\frac{dz}{z^{k+1}}
\end{equation*}%
and%
\begin{equation*}
t(n,k)=\frac{k!}{m!2\pi i}\int\limits_{\left\vert w\right\vert =r}\left( 2%
\text{area} \sinh \left( \frac{z}{2}\right) \right) ^{m}\frac{dz}{z^{k+1}}
\end{equation*}%
(\textit{cf}. \cite[Preposition 4.2.2]{Butzer}).

The following tables give us the upper part of the matrices of central
factorial numbers of the first and second kind, respectively (\textit{cf}. 
\cite[p. 13]{Cigler}; see also OEIS A036969):%
\begin{equation}
\left( T(i,j)\right) _{i,j=0}^{6}=\left[ 
\begin{array}{ccccccc}
1 & 0 & 0 & 0 & 0 & 0 & 0 \\ 
0 & 1 & 0 & 0 & 0 & 0 & 0 \\ 
0 & 1 & 1 & 0 & 0 & 0 & 0 \\ 
0 & 1 & 5 & 1 & 0 & 0 & 0 \\ 
0 & 1 & 21 & 14 & 1 & 0 & 0 \\ 
0 & 1 & 85 & 147 & 30 & 1 & 0 \\ 
0 & 1 & 341 & 1408 & 627 & 55 & 1%
\end{array}
\right]  \label{AmatCT1}
\end{equation}%
and%
\begin{equation}
\left( t(i,j)\right) _{i,j=0}^{6}=\left[ 
\begin{array}{ccccccc}
1 & 0 & 0 & 0 & 0 & 0 & 0 \\ 
0 & 1 & 0 & 0 & 0 & 0 & 0 \\ 
0 & -1 & 1 & 0 & 0 & 0 & 0 \\ 
0 & 4 & -5 & 1 & 0 & 0 & 0 \\ 
0 & -36 & 49 & -14 & 1 & 0 & 0 \\ 
0 & 576 & -870 & 273 & -30 & 1 & 0 \\ 
0 & -14400 & 21076 & -7645 & 1023 & -55 & 1%
\end{array}
\right].  \label{AmatCt}
\end{equation}

The Daehee numbers of the first kind and the second kind are defined by
means of the following generating functions, respectively: 
\begin{equation}
\frac{\log (1+t)}{t}=\sum_{n=0}^{\infty }D_{n}\frac{t^{n}}{n!}  \label{aii4}
\end{equation}%
and 
\begin{equation*}
\frac{(1+t)\log (1+t)}{t}=\sum_{n=0}^{\infty }\widehat{D}_{n}\frac{t^{n}}{n!}
,
\end{equation*}%
(\textit{cf}. \cite[p. 45]{riARDON}, \cite{ElDosky}, \cite{DSkimDaehee}).
Using (\ref{aii4}), we have%
\begin{equation*}
D_{n}=(-1)^{n}\frac{n!}{n+1}
\end{equation*}%
(\textit{cf}. \cite{DSkimDaehee}, see also \cite{DN1}, \cite{DN2}, \cite{DN3}%
).

The Changhee numbers of the first kind and the second kind are defined by
means of the following generating functions, respectively: 
\begin{equation}
\frac{2}{t+1}=\sum_{n=0}^{\infty }Ch_{n}\frac{t^{n}}{n!}  \label{aii5}
\end{equation}%
and 
\begin{equation*}
\frac{2(1+t)}{t+2}=\sum_{n=0}^{\infty }\widehat{Ch}_{n}\frac{t^{n}}{n!},
\end{equation*}%
(\textit{cf}. \cite{DSkim2}). Using (\ref{aii5}), we have%
\begin{equation*}
Ch_{n}=(-1)^{n}\frac{n!}{2^{n}}
\end{equation*}%
and%
\begin{equation}
Ch_{n}=\sum_{k=0}^{n}S_{1}(n,k)E_{k}  \label{ChEuler}
\end{equation}%
(\textit{cf}. \cite{DSkim2}, see also \cite{CN1}, \cite{CN2}).

The Peters polynomials $s_{k}(x;\lambda ,\mu )$ are defined by means of the
following generating function: 
\begin{equation}
\frac{1}{\left( 1+\left( 1+t\right) ^{\lambda }\right) ^{\mu }}
(1+t)^{x}=\sum_{n=0}^{\infty }s_{n}(x;\lambda ,\mu )\frac{t^{n}}{n!}
\label{PP}
\end{equation}%
which, for $x=0$, reduces to the Peters numbers $s_{n}(0;\lambda ,\mu
)=s_{n}(\lambda ,\mu )$ (\textit{cf.} \cite{Roman}, \cite{Jordan}), \cite%
{simsekMMas2019}).

Substituting $\mu =1$ into (\ref{PP}), we have the Boole polynomials. If we
substitute $\lambda =1$ and $\mu =1$ into (\ref{PP}), then we have the
Changhee polynomials (\textit{cf.} \cite{DSkim2}, \cite{DSKIMBoole}) We
observe that recently, the Peters polynomials, the Boole polynomials, the
Changhee polynomials, Daehee polynomials and combinatorial numbers and
polynomials have been studied by many authors see, for details (\textit{cf.} 
\cite{Do}-\cite{ElDosky2}, \cite{DSKimTkimASCM2019}-\cite{Lim}; see also
many of the recent works cited in this paper).

By using (\ref{PP}), we have%
\begin{equation*}
\sum_{n=0}^{\infty }s_{n}(x;\lambda ,\mu )\frac{t^{n}}{n!}
=\sum_{n=0}^{\infty }s_{n}(\lambda ,\mu )\frac{t^{n}}{n!}\sum_{n=0}^{\infty
}x_{(n)}\frac{t^{n}}{n!}.
\end{equation*}%
Therefore, we have%
\begin{equation}
s_{n}(x;\lambda ,\mu )=\sum\limits_{v=0}^{n}\binom{n}{v}x_{(n-v)}s_{v}(
\lambda ,\mu )  \label{ay3}
\end{equation}%
(\textit{cf.} \cite{Roman}, \cite{Jordan}, \cite{DSKIMBoole}, \cite%
{simsekMMas2019}).

In \cite{simsekAADM2018}, we defined the generating function for the
combinatorial numbers $y_{1}(n,k;\lambda )$ as follows: 
\begin{equation}
F_{y_{1}}(t,k;\lambda )=\frac{1}{k!}\left( \lambda e^{t}+1\right)
^{k}=\sum_{n=0}^{\infty }y_{1}(n,k;\lambda )\frac{t^{n}}{n!},  \label{ay1}
\end{equation}%
where $k\in \mathbb{N}_{0}$, $\lambda \in \mathbb{C}$ and 
\begin{equation*}
y_{1}(n,k;\lambda)=\frac{1}{k!}\sum_{j=0}^{k}\binom{k}{j}j^n \lambda^j.
\end{equation*}
By (\ref{ay1}), we have%
\begin{equation*}
B(n,k)=k!y_{1}(n,k;1)
\end{equation*}
(\textit{cf}. \cite{Glomberg}, \cite{simsekAADM2018}).

\begin{theorem}[\textit{cf}. \protect\cite{simsekAADM}]
Let $\mu \in \mathbb{N}$. Then we have 
\begin{equation}
x_{(n)}=\sum\limits_{v=0}^{n}\sum\limits_{j=0}^{\mu }\binom{\mu }{j}\binom{n 
}{v}\left( \lambda j\right) _{(v)}s_{n-v}\left( x;\lambda ,\mu \right) .
\label{ay1B}
\end{equation}
\end{theorem}

\begin{theorem}[\textit{cf}. \protect\cite{simsekAADM}]
Let $\mu \in \mathbb{Z}^{+}$. Then we have 
\begin{equation}
x_{(n)}=\sum\limits_{v=0}^{n}\sum\limits_{k=0}^{v}\binom{n}{v} \lambda
^{k}B\left( k,\mu \right) S_1\left( v,k\right) s_{n-v}\left( x;\lambda ,\mu
\right) .  \label{ay1C}
\end{equation}
\end{theorem}

In \cite{simsekAADM}, we constructed the following generating function for
combinatorial polynomials $Y_{n,2}\left( x,\lambda \right) $, which are
member of the family of the Peters polynomials, as follows:%
\begin{equation}
F_{Y_{2}}\left( t,x;\lambda \right) =\frac{2\left( 1+\lambda t\right) ^{x}}{
\lambda ^{2}t+2\left( \lambda -1\right) }=\sum\limits_{n=0}^{\infty
}Y_{n,2}\left( x;\lambda \right) \frac{t^{n}}{n!}  \label{aY2}
\end{equation}%
in which if we set $x=0$, then we have combinatorial numbers $Y_{n,2}\left(
\lambda \right) =Y_{n,2}\left( 0;\lambda \right) $. By using (\ref{aY2}), we
have 
\begin{equation*}
Y_{n,2}\left( \lambda \right) =\frac{1}{2^{n+1}}Y_{n}\left( \lambda \right)
\end{equation*}%
(\textit{cf}. \cite{simsekTJM}, \cite{simsekAADM}) so that the numbers $%
Y_{n}\left( \lambda \right)$ are defined by the author in (\textit{cf}. \cite%
{simsekTJM}).

Substituting $x=0$, $\lambda =\mu =1$ and $t=\frac{\theta ^{2}u}{\theta -1}$
into (\ref{PP}), we have%
\begin{equation*}
s_{n}\left( 0;1,1\right) =\frac{\left( \theta -1\right) ^{n+1}}{2\theta
^{2n} }Y_{n,2}\left( \theta \right)
\end{equation*}%
(\textit{cf}. \cite{simsekAADM}).

\begin{theorem}[\textit{cf}. \protect\cite{simsekAADM}]
Let $n\in \mathbb{N}$. Then we have 
\begin{eqnarray*}
s_{n}\left( x;\lambda ,\mu \right) &=&\frac{n}{2}\sum\limits_{j=0}^{n-1} 
\binom{n-1}{j}\theta ^{j+2-n}s_{j}\left( \lambda ,\mu \right)
Y_{n-1-j,2}\left( x,\theta \right) \\
&&+\left( \theta -1\right) \sum\limits_{j=0}^{n}\binom{n}{j}\theta
^{j-n}s_{j}\left( \lambda ,\mu \right) Y_{n-j,2}\left( x,\theta \right) .
\end{eqnarray*}
\end{theorem}

\begin{theorem}[\textit{cf}. \protect\cite{simsekAADM}]
\begin{equation}
Y_{n,2}\left( x;\lambda \right) =\sum\limits_{j=0}^{n}\left( 
\begin{array}{c}
n \\ 
j%
\end{array}
\right) \lambda ^{n-j}Y_{j,2}\left( \lambda \right) x_{(n-j)}.  \label{A1}
\end{equation}
\end{theorem}

\begin{lemma}[\textit{cf}. \protect\cite{simsekAADM}]
Let $n\in \mathbb{N}_{0}$. Then we have 
\begin{equation}
Y_{n,2}\left( \lambda \right) =2\left( -1\right) ^{n}n!\frac{\lambda ^{2n}}{
\left( 2\lambda -2\right) ^{n+1}} .  \label{A2}
\end{equation}
\end{lemma}

Substituting (\ref{A2}) into (\ref{A1}), we get a explicit formula for the
polynomials $Y_{n,2}\left( x;\lambda \right) $ by the following theorem:

\begin{theorem}[\textit{cf}. \protect\cite{simsekAADM}]
Let $n\in \mathbb{N}_{0}$. Then we have 
\begin{equation}
Y_{n,2}\left( x;\lambda \right) =2\sum\limits_{j=0}^{n}\left( -1\right)
^{j}j!\left( 
\begin{array}{c}
n \\ 
j%
\end{array}
\right) \frac{\lambda ^{n+j}}{\left( 2\lambda -2\right) ^{j+1}}x_{(n-j)}.
\label{A3}
\end{equation}
\end{theorem}

This paper has exactly 9 main sections including introduction. We summarize
as follows:

In Section 2, we give some definitions, notations and formulas for
distributions and $p$-adic ($q$-) integrals on $\mathbb{Z}_{p}$.

In Section 3, we give some applications and formulas for the Volkenborn
integral.

In Section 4, we give some applications and formulas for the fermionic $p$%
-adic integral.

In Section 5, we give many new formulas for the Volkenborn integral
including the falling factorials, the raising factorials, the Bernoulli
numbers and polynomials, the Euler numbers and polynomials, the Stirling
numbers, the Lah numbers, the Peters numbers and polynomials, the central
factorial numbers, the Daehee numbers and polynomials, the Changhee numbers
and polynomials, the Harmonic numbers, the Fubini numbers, combinatorial
numbers and sums.

In Section 6, we give many new formulas for the fermionic $p$-adic integral
including the falling factorials, the raising factorials, the Bernoulli
numbers and polynomials, the Euler numbers and polynomials, the Stirling
numbers, the Lah numbers, the Peters numbers and polynomials, the central
factorial numbers, the Daehee numbers and polynomials, the Changhee numbers
and polynomials, the Harmonic numbers, the Fubini numbers, combinatorial
numbers and sums.

In Section 7, by using formulas in Section 6 and in Section 7, we give
various identities including the Bernoulli numbers and polynomials, the
Euler numbers and polynomials, the Stirling numbers, the Lah numbers, the
Peters numbers and polynomials, the central factorial numbers, the Daehee
numbers and polynomials, the Changhee numbers and polynomials, the Harmonic
numbers, the Fubini numbers, combinatorial numbers and sums.

In Section 8, we give some questions and new sequences with their
definitions and properties.

In Section 9, we give conclusion and observations on our results.

\section{\textbf{Distributions and $p$-adic $q$-integrals on $\mathbb{Z}_{p}$%
}}

In this section, we give brief introduction for $p$-adic distributions and $%
p $-adic ($q$-) integrals. For the fundamental properties of $p$-adic
integrals and $p$-adic distributions, which are given briefly below, we may
refer the references \cite{Amice, AraciEtAl1, AraciEtAl2, Khrennikov, T.
Kim, Kim2006TMIC, OzdenThesis, Schikof, Vladimirov, Volkenborn}; and the
references cited therein.

Some notations and definitions for $p$-adic integrals are given as follows:

Let $p$ be an odd prime number. Let $m\in \mathbb{N}$. Let $ord_{p}(m)$
denote the greatest integer $k$ ($k\in \mathbb{N}_{0})$ such that $p^{k}$
divides $m$ in $\mathbb{Z}$. If $m=0$, then $ord_{p}(m)=\infty $. Let $x\in 
\mathbb{Q}$, the set of rational numbers, with $x=\frac{a}{b}$ for $a,b\in 
\mathbb{Z}$ with $n\neq 0$. Therefore,%
\begin{equation*}
ord_{p}(x)=ord_{p}(\frac{a}{b})=ord_{p}(a)-ord_{p}(b).
\end{equation*}%
Let $\left\vert .\right\vert _{p}$ is a map on $\mathbb{Q}$. This map, which
is a norm over $\mathbb{Q}$, is defined by%
\begin{equation*}
\left\vert x\right\vert _{p}=\left\{ 
\begin{array}{cc}
p^{-ord_{p}(x)} & \text{if }x\neq 0, \\ 
0 & \text{if }x=0.%
\end{array}
\right.
\end{equation*}%
For instance, $x\in \mathbb{Q}$ with $x=p^{y}\frac{x_{1}}{x_{2}}$ where $%
y,x_{1},x_{2}\in \mathbb{Z}$ and $x_{1}$ and $x_{2}$ are not divisible by $p$%
. Hence, $ord_{p}(x)=y$ and $\left\vert x\right\vert _{p}=p^{-y}$. The set $%
\mathbb{Q}_{p}$ equipped with this norm $\left\vert x\right\vert _{p}$ is a
topological completion of of set $\mathbb{Q}$. Let $\mathbb{C}_{p}$ be the
field of $p$-adic completion of algebraic closure of $\mathbb{Q}_{p}$. Let $%
\mathbb{Z}_{p}$ be topological closure of $\mathbb{Z}$. Let $\mathbb{Z}_{p}$
be a set of $p$-adic integers, which is related to the norm $\left\vert
x\right\vert _{p}$, given as follows: 
\begin{equation*}
\mathbb{Z}_{p}=\left\{ x\in \mathbb{Q}_{p}:\left\vert x\right\vert _{p}\leq
1\right\}
\end{equation*}

In order to define $p$-adic integral, we need the following definitions and
formulas.

Let $f:\mathbb{Z}_{p}\rightarrow \mathbb{C}_{p}$. $f$ is called a uniformly
differential function at a point$\ a\in \mathbb{Z}_{p}$ if $f$ satisfies the
following conditions:

If the difference quotients $\Phi _{f}:\mathbb{Z}_{p}\times \mathbb{Z}%
_{p}\rightarrow \mathbb{C}_{p}$ such that%
\begin{equation*}
\Phi _{f}(x,y)=\frac{f(x)-f(y)}{x-y}
\end{equation*}%
have a limit $f^{\prime }(z)$ as $(x,y)\rightarrow (0,0)$ ($x$ and $y$
remaining distinct). A set of uniformly differential functions is briefly
indicated by $f\in UD(\mathbb{Z}_{p})$ or $f\in C^{1}(\mathbb{Z}%
_{p}\rightarrow \mathbb{C}_{p}\mathbb{)}$. The additive cosets of $\mathbb{Z}%
_{p}$ are given as follows:%
\begin{equation*}
p\mathbb{Z}_{p}=\left\{ x\in \mathbb{Z}_{p}:\left\vert x\right\vert
_{p}<1\right\} ,1+p\mathbb{Z}_{p},\ldots ,p-1+p\mathbb{Z}_{p},
\end{equation*}%
where $p\mathbb{Z}_{p}$ is a maximal ideal of $\mathbb{Z}_{p}$ and for each $%
j\in \left\{ 0,1,\ldots ,p^{n}-1\right\} $ we set%
\begin{equation*}
j+p^{n}\mathbb{Z}_{p}=\left\{ x\in \mathbb{Z}_{p}:\left\vert x-j\right\vert
_{p}<p^{1-n}\right\} .
\end{equation*}%
Thus, we have%
\begin{equation*}
\mathbb{Z}_{p}=\cup _{j=0}^{p-1}\left( j+p\mathbb{Z}_{p}\right).
\end{equation*}%
By using the above coset, we give the following distributions on $\mathbb{Z}%
_{p}$ for $p$-adic integrals:

Every map $\mu $ from the set of intervals contained in $X$ to $\mathbb{Q}%
_{p}$ for which%
\begin{equation*}
\mu \left( x+p^{n}\mathbb{Z}_{p}\right) =\sum\limits_{j=0}^{p-1}\mu \left(
x+jp^{n}+p^{n+1}\mathbb{Z}_{p}\right)
\end{equation*}%
whenever $x+p^{n}\mathbb{Z}_{p}\subset X$, exists uniquely to a $p$-adic
distribution on $X$ (\textit{cf}. \cite{Amice}, \cite{Koblitz}, \cite%
{Schikof}, \cite{Simsek SrivastavaP.ADIC.DiST}, \cite{Vladimirov}, \cite%
{Volkenborn}).

Some well-known examples for the $p$-adic distribution are given as follows:

The Haar distribution is defined by%
\begin{equation}
\mu _{Haar}\left( x+p^{N}\mathbb{Z}_{p}\right) =\mu _{1}\left( x+p^{N} 
\mathbb{Z}_{p}\right) =\frac{1}{p^{N}},  \label{HDis}
\end{equation}%
which denotes by $\mu _{1}\left( x+p^{N}\mathbb{Z}_{p}\right) =\mu
_{1}\left( x\right) $;

The Dirac distribution is defined by%
\begin{equation*}
\mu _{Dirac}\left( x+p^{N}\mathbb{Z}_{p}\right) =\mu _{\alpha }\left(
X\right) =\left\{ 
\begin{array}{cc}
1 & \text{if }x\in X, \\ 
0 & \text{otherwise};%
\end{array}
\right.
\end{equation*}

The Mazur distribution is defined by%
\begin{equation*}
\mu _{Mazur}\left( x+p^{N}\mathbb{Z}_{p}\right) =\frac{a}{p^{N}}-\frac{1}{2},
\end{equation*}%
where $a\in \mathbb{Q}$ with $0\leq a\leq p-1$;

The Bernoulli distribution is defined by%
\begin{equation*}
\mu _{B,k}\left( x+p^{N}\mathbb{Z}_{p}\right) =p^{N(k-1)}B_{k}\left( \frac{a 
}{p^{N}}\right) ,
\end{equation*}%
(\textit{cf}. \cite{Amice}, \cite{KimJNT}, \cite{mskimKOREAN}, \cite{Koblitz}%
, \cite{Schikof}, \cite{Simsek SrivastavaP.ADIC.DiST}, \cite{Vladimirov}, 
\cite{Volkenborn}).

Observe that special values of the Bernoulli distribution are related to
Haar distribution and the Mazur distribution; that is%
\begin{equation*}
\mu _{B,0}\left( x+p^{N}\mathbb{Z}_{p}\right) =\mu _{1}\left( x+p^{N}\mathbb{%
\ Z}_{p}\right)
\end{equation*}%
and%
\begin{equation*}
\mu _{B,1}\left( x+p^{N}\mathbb{Z}_{p}\right) =\mu _{Mazur}\left( x+p^{N} 
\mathbb{Z}_{p}\right)
\end{equation*}%
(\textit{cf}. \cite[p.35]{Koblitz}).

On the other hand, we have the following distribution $\mu _{-1}\left(
x+p^{N}\mathbb{Z}_{p}\right) $ on $\mathbb{Z}_{p}$:%
\begin{equation}
\mu _{-1}\left( x+p^{N}\mathbb{Z}_{p}\right) =(-1)^{x}  \label{KDis}
\end{equation}%
which is denoted by $\mu _{-1}\left( x+p^{N}\mathbb{Z}_{p}\right) =\mu
_{-1}\left( x\right) $ (\textit{cf}. \cite{MSKIM}, \cite{MSKIMjnt2009}, \cite%
{T. Kim}, \cite{Kim2006TMIC}, \cite{RimKim}).

The Euler distribution is defined by%
\begin{equation*}
\mu _{\mathcal{E},k,q}\left( x+fp^{N}\mathbb{Z}_{p}\right) =(-1)^{a}\left(
fp^{N)}\right) ^{k}\mathcal{E}_{k}\left( \frac{a}{fp^{N}};q^{fp^{N}}\right) ,
\end{equation*}%
where $N,k,f\in \mathbb{N}$ and $f$ is odd (\textit{cf}. \cite{RimKim}, \cite%
{ozdenSmsekCangul}, \cite{OzdenAMC2014}).

Therefore 
\begin{eqnarray}
\lim_{q\rightarrow 1}\mu _{\mathcal{E},k,q}\left( x+fp^{N}\mathbb{Z}%
_{p}\right)&=&\mu _{E,k}\left( x+fp^{N}\mathbb{Z}_{p}\right)  \label{AEMus}
\\
&=&(-1)^{x}\left( fp^{N)}\right) ^{k}E_{k}\left( \frac{x}{fp^{N}}\right), 
\notag
\end{eqnarray}
(\textit{cf}. \cite{RimKim}).

Since%
\begin{equation*}
\left\vert \mu _{E,k}\left( x+fp^{N}\mathbb{Z}_{p}\right) \right\vert
_{p}\leq 1,
\end{equation*}%
$\mu _{E,k}\left( x+fp^{N}\mathbb{Z}_{p}\right) $ is a measure on $\mathbb{X}
$ where%
\begin{equation*}
\mathbb{X}=\mathbb{X}_{f}=\lim_{\overleftarrow{N}}\mathbb{Z}/fp^{N}\mathbb{Z}
\text{ and }\mathbb{X}_{1}=\mathbb{Z}_{p}
\end{equation*}%
(\textit{cf}. \cite{RimKim}; see also \cite{mskimKOREAN}, \cite%
{ozdenSmsekCangul}, \cite{OzdenAMC2014}).

Substituting $f=1$ and $k=0$ into (\ref{AEMus}), we have%
\begin{equation*}
\mu _{E,0}\left( x+fp^{N}\mathbb{Z}_{p}\right) =\mu _{-1}\left( x\right) .
\end{equation*}

In order to define invariant $p$-adic $z$-integrals, Kim \cite{KimITSFdahee}
gave the following distribution on $\mathbb{Z}_{p}$:

Let%
\begin{equation*}
\mu _{z}\left( a+p^{N}\mathbb{Z}_{p}\right) =\frac{z^{a}}{\left[ p^{N}:z %
\right] },
\end{equation*}%
where%
\begin{equation*}
\left[ x:z\right] =\frac{1-z^{x}}{1-z}.
\end{equation*}%
It well-known that $\mu _{z}\left( x+p^{N}\mathbb{Z}_{p}\right) =\mu
_{z}\left( x\right) $ is extended distribution on $\mathbb{Z}_{p}$ (\textit{%
\ cf}. \cite[p. 244]{srivas18}).

For a compact-open subset $\mathbb{X}$ of $\mathbb{Q}_{p}$, a $p$-adic
distribution $\mu $ on $\mathbb{X}$ is a $\mathbb{Q}_{p}$-linear vector
space homomorphism from the $\mathbb{Q}_{p}$-vector space of locally
constant functions on $\mathbb{X}$ to $\mathbb{Q}_{p}$ (\textit{cf}. \cite%
{Schikof}).

Let $\mathbb{K}$ be a field with a complete valuation and $C^{1}(\mathbb{Z}%
_{p}\rightarrow \mathbb{K)}$ be a set of functions which have continuous
derivative (see, for detail, \cite{Schikof}).

Kim \cite{T. Kim} defined the $p$-adic $q$-integral as follows:

Let $f\in C^{1}(\mathbb{Z}_{p}\rightarrow \mathbb{K)}$ and $q\in \mathbb{C}%
_{p}$ with $\left\vert 1-q\right\vert _{p}<1$. Then we have%
\begin{equation}
I_{q}(f(x))=\int_{\mathbb{Z}_{p}}f(x)d\mu _{q}(x)=\lim_{N\rightarrow \infty
} \frac{1}{[p^{N}]_{q}}\sum_{x=0}^{p^{N}-1}f(x)q^{x},  \label{q-BI}
\end{equation}%
where%
\begin{equation*}
\left[ x\right] =\left[ x:q\right] =\left\{ 
\begin{array}{c}
\frac{1-q^{x}}{1-q},q\neq 1 \\ 
x,q=1%
\end{array}
\right.
\end{equation*}%
and%
\begin{equation*}
\mu _{q}(x)=\mu _{q}\left( x+p^{N}\mathbb{Z}_{p}\right)
\end{equation*}%
which denotes $q$-distribution on $\mathbb{Z}_{p}$ and it is defined by%
\begin{equation*}
\mu _{q}\left( x+p^{N}\mathbb{Z}_{p}\right) =\frac{q^{x}}{\left[ p^{N}\right]
_{q}},
\end{equation*}%
(\textit{cf}. \cite{T. Kim}).

Observe that%
\begin{equation*}
\lim_{q\rightarrow 1}\mu _{q}\left( x+p^{N}\mathbb{Z}_{p}\right) =\mu
_{Haar}\left( x+p^{N}\mathbb{Z}_{p}\right) =\mu _{1}\left( x\right)
\end{equation*}%
and%
\begin{equation*}
\lim_{q\rightarrow -1}\mu _{q}\left( x+p^{N}\mathbb{Z}_{p}\right) =\mu
_{-1}\left( x\right) .
\end{equation*}

Observe that if $q\rightarrow 1$, then (\ref{q-BI}) reduces to the following
well-known Volkenborn integral (bosonic integral), which is denoted by $%
I_{1}(f(x))$:%
\begin{equation}
\lim_{q\rightarrow 1}I_{q}(f(x))=I_{1}(f(x))=\int\limits_{\mathbb{Z}
_{p}}f\left( x\right) d\mu _{1}\left( x\right) =\underset{N\rightarrow
\infty }{\lim }\frac{1}{p^{N}}\sum_{x=0}^{p^{N}-1}f\left( x\right) ,
\label{M}
\end{equation}%
where $\mu _{1}\left( x\right) $ is given by the equation (\ref{HDis}), that
is%
\begin{equation*}
\mu _{1}\left( x\right) =\frac{1}{p^{N}}
\end{equation*}%
(\textit{cf}. \cite{Amice}, \cite{Khrennikov}, \cite{Schikof}, \cite%
{Vladimirov}, \cite{Volkenborn}); see also the references cited in each of
these earlier works).

The above integral has many applications not only in mathematics, but also
in mathematical physics. By using this integral and its integral equations,
various family of generating functions associated with Bernoulli-type
numbers and polynomials have been constructed (\textit{cf}. \cite{apostol}-%
\cite{Wang}).

Over and above, if $q\rightarrow -1$, then (\ref{q-BI}) reduces to the
following well-known fermionic $p$-adic integral, which is denoted by $%
I_{-1}(f(x))$: 
\begin{eqnarray}
\lim_{q\rightarrow -1}I_{q}(f(x))=I_{-1}(f(x))&=&\int\limits_{\mathbb{Z}
_{p}}f\left( x\right) d\mu _{-1}\left( x\right)  \label{Mmm} \\
&=&\underset{N\rightarrow \infty }{\lim }\sum_{x=0}^{p^{N}-1}\left(
-1\right) ^{x}f\left( x\right),  \notag
\end{eqnarray}
where $\mu _{-1}\left( x\right) $ is given by the equation (\ref{KDis}),
that is%
\begin{equation*}
\mu _{-1}\left( x\right) =\left( -1\right) ^{x}
\end{equation*}%
(\textit{cf}. \cite{Kim2006TMIC}, see also \cite{MSKIMjnt2009}).

By using $p$-adic fermionic integral and its integral equations, various
different generating functions, including Euler-type numbers and polynomials
and Genocchi-type numbers and polynomials, have been constructed (\textit{cf}%
. \cite{apostol}-\cite{Wang}).

We also note that $p$-adic $q$-integrals are related to the theory of the
generating functions, ultrametric calculus, the quantum groups, cohomology
groups, $q$-deformed oscillator and $p$-adic models (\textit{cf}. \cite%
{Khrennikov, Vladimirov}).

\subsection{\textbf{Some Properties of the Volkenborn Integral}}

Here, we give some well-known properties of the Volkenborn integral (bosonic 
$p$-adic integral).

The Volkenborn integral is given in terms of the Mahler coefficients $\binom{
x}{n}$ as follows: 
\begin{equation*}
\int\limits_{\mathbb{Z}_{p}}f\left( x\right) d\mu _{1}\left( x\right)
=\sum\limits_{n=0}^{\infty }\frac{(-1)^{n}}{n+1}a_{n},
\end{equation*}%
where%
\begin{equation*}
f\left( x\right) =\sum\limits_{n=0}^{\infty }a_{n}\binom{x}{n}\in C^{1}( 
\mathbb{Z}_{p}\rightarrow \mathbb{K)},
\end{equation*}%
where%
\begin{equation*}
\binom{x}{n}=\frac{x_{(n)}}{n!},
\end{equation*}%
$n\in \mathbb{N}_{0}$ (\textit{cf}. \cite[p. 168-Proposition 55.3]{Schikof}).

In \cite{Schikof}, Schikhof gave the following integral formula: 
\begin{equation}
\int\limits_{\mathbb{Z}_{p}}f(x+n)d\mu _{1}\left( x\right) =\int\limits_{ 
\mathbb{Z}_{p}}f(x)d\mu _{1}\left( x\right)
+\sum\limits_{k=0}^{n-1}f^{\prime }(k),  \label{v1}
\end{equation}%
where 
\begin{equation*}
f^{\prime }(x)=\frac{d}{dx}\left\{ f(x)\right\} .
\end{equation*}%
By substituting 
\begin{equation*}
f(x)=\binom{x}{n}
\end{equation*}%
into (\ref{v1}), we get 
\begin{equation*}
\int\limits_{\mathbb{Z}_{p}}\binom{x+1}{n}d\mu _{1}\left( x\right)
=\int\limits_{\mathbb{Z}_{p}}\binom{x}{n}d\mu _{1}\left( x\right) +\frac{d}{
dx}\left\{ \binom{x}{n}\right\} \left\vert _{x=0}\right. ,
\end{equation*}%
where 
\begin{eqnarray*}
\frac{d}{dx}\left\{ \binom{x}{n}\right\} \left\vert _{x=0}\right. &=&\left\{ 
\frac{1}{n!}x_{(n)}\sum\limits_{k=0}^{n-1}\frac{1}{x-k}\right\} \left\vert
_{x=0}\right. \\
&=&(-1)^{n-1}\frac{1}{n}.
\end{eqnarray*}%
Therefore, we have 
\begin{equation}
\int\limits_{\mathbb{Z}_{p}}\binom{x+1}{n}d\mu _{1}\left( x\right) =\frac{
(-1)^{n}}{n+1}+(-1)^{n-1}\frac{1}{n}.  \label{ai0a4}
\end{equation}

Let $f:\mathbb{Z}_{p}\rightarrow \mathbb{K}$ be an analytic function and%
\begin{equation*}
f\left( x\right) =\sum\limits_{n=0}^{\infty }a_{n}x^{n}
\end{equation*}%
with $x\in \mathbb{Z}_{p}$.

The Volkenborn integral of this analytic function is given by 
\begin{equation*}
\int\limits_{\mathbb{Z}_{p}}\left( \sum\limits_{n=0}^{\infty
}a_{n}x^{n}\right) d\mu _{1}\left( x\right) =\sum\limits_{n=0}^{\infty
}a_{n}\int\limits_{\mathbb{Z}_{p}}x^{n}d\mu _{1}\left( x\right)
\end{equation*}%
and%
\begin{equation*}
\int\limits_{\mathbb{Z}_{p}}\left( \sum\limits_{n=0}^{\infty
}a_{n}x^{n}\right) d\mu _{1}\left( x\right) =\sum\limits_{n=0}^{\infty
}a_{n}B_{n},
\end{equation*}%
(\textit{cf}. \cite{T. Kim}, \cite{Kim2006TMIC}, \cite{Schikof}; see also
the references cited in each of these earlier works).

Integral equation for the Volkenborn integral is given as follows: 
\begin{equation}
\int\limits_{\mathbb{Z}_{p}}E^{m}\left[ f(x)\right] d\mu _{1}\left( x\right)
=\int\limits_{\mathbb{Z}_{p}}f(x)d\mu _{1}\left( x\right)
+\sum\limits_{j=0}^{m-1}\frac{d}{dx}\left\{ f(x)\right\} \left\vert
_{x=j}\right.  \label{vi-1}
\end{equation}%
where%
\begin{equation*}
E^{m}\left[ f(x)\right] =f(x+m)
\end{equation*}%
and%
\begin{equation*}
\frac{d}{dx}\left\{ f(x)\right\} \left\vert _{x=j}\right. =f^{^{\prime }}(j)
\end{equation*}%
(\textit{cf}. \cite{T. Kim}, \cite{Kim2006TMIC}, \cite{Schikof}, \cite%
{wikipe}; see also the references cited in each of these earlier works).

Using (\ref{q-BI}), the following integral equation was given by Kim \cite%
{KIMaml2008}:%
\begin{equation}
q\int_{\mathbb{Z}_{p}}E\left[ f(x)\right] d\mu _{q}(x)=\int_{\mathbb{Z}
_{p}}f(x)d\mu _{q}(x)+\frac{q-1}{\log q}f^{^{\prime }}(0)+(q-1)f(0)
\label{q-BI.1}
\end{equation}%
(\textit{cf}. see also \cite{DSkim2}-\cite{KimMSKimARXIVE}).

As usual, exponential function is defined as follows:%
\begin{equation*}
e^{t}=\sum_{n=0}^{\infty }\frac{t^{n}}{n!}.
\end{equation*}%
The above series convergences in region $\mathbb{E}$ which is a subset of
field $\mathbb{K}$ with $char(\mathbb{K})=0$ (\textit{cf}. \cite[p. 70]%
{Schikof}). Let $k$ be residue class field of $\mathbb{K}$. If $char(k)=p$,
then%
\begin{equation*}
\mathbb{E}=\left\{ x\in \mathbb{K}:\left\vert x\right\vert <p^{\frac{1}{1-p}
}\right\},
\end{equation*}%
and if $char(k)=0$, then%
\begin{equation*}
\mathbb{E}=\left\{ x\in \mathbb{K}:\left\vert x\right\vert <1\right\} .
\end{equation*}

Let $f\in C^{1}(\mathbb{Z}_{p}\rightarrow \mathbb{K)}$. Kim \cite[Theorem 1]%
{KIMaml2008} gave the following integral equation: 
\begin{eqnarray}
&&q^{n}\int\limits_{\mathbb{Z}_{p}}E^{n}\left[ f\left( x\right) \right] d\mu
_{q}\left( x\right) -\int\limits_{\mathbb{Z}_{p}}f\left( x\right) d\mu
_{q}\left( x\right)  \label{TKint.q} \\
&=& \frac{q-1}{\log q}\left( \sum_{j=0}^{n-1}q^{j}f^{^{\prime }}(j)+\log
q\sum_{j=0}^{n-1}q^{j}f(j)\right),  \notag
\end{eqnarray}
where $n$ is a positive integer.

Observe that substituting $n=1$ into (\ref{TKint.q}), we arrive at (\ref%
{q-BI.1}).

\begin{theorem}
\label{TheoremShcrikof} Let $n\in 
\mathbb{N}
_{0}$. Then we have 
\begin{equation}
\int\limits_{\mathbb{Z}_{p}}\binom{x}{n}d\mu _{1}\left( x\right) =\frac{
(-1)^{n}}{n+1}.  \label{C7}
\end{equation}
\end{theorem}

Note that Theorem \ref{TheoremShcrikof} was proved by Schikhof \cite{Schikof}%
.

Substituting $m=1$ and $f(x)=(1+a)^{x}$ into (\ref{vi-1}), we have%
\begin{equation*}
\int\limits_{\mathbb{Z}_{p}}(1+t)^{x}d\mu _{1}\left( x\right) =\frac{1}{t}
\log (1+t).
\end{equation*}%
Therefore, 
\begin{equation*}
\sum_{n=0}^{\infty }t^{n}\int\limits_{\mathbb{Z}_{p}}\binom{x}{n}d\mu
_{1}\left( x\right) =\frac{1}{t}\log (1+t).
\end{equation*}%
Combining the above equation with (\ref{C7}), we have the following
well-known relation:%
\begin{equation*}
\log (1+t)=\sum_{n=0}^{\infty }\frac{(-1)^{n}t^{n+1}}{n+1}
\end{equation*}%
(\textit{cf}. \cite{Schikof}, \cite{wikipe}). We observe that%
\begin{equation*}
\int\limits_{\mathbb{Z}_{p}}a^{x}d\mu _{1}\left( x\right) =\frac{1}{a-1}\log
_{p}(a),
\end{equation*}%
where $a\in \mathbb{C}_{p}^{+}$ with $a\neq 1$ (\textit{cf}. \cite[p. 170]%
{Schikof}).

Let $f\in C^{1}(\mathbb{Z}_{p}\rightarrow \mathbb{K)}$. Then we have%
\begin{equation*}
\int\limits_{\mathbb{Z}_{p}}f(-x)d\mu _{1}\left( x\right) =\int\limits_{ 
\mathbb{Z}_{p}}f(1+x)d\mu _{1}\left( x\right)
\end{equation*}%
and if $f(-x)=-f(x)$, which is geometrically symmetric about the origin, we
have%
\begin{equation}
\int\limits_{\mathbb{Z}_{p}}f(x)d\mu _{1}\left( x\right) =-\frac{1}{2}
f^{\prime }(0)  \label{ATTek}
\end{equation}%
(\textit{cf}. \cite[p. 169]{Schikof}). By using the above formulas, we give
the following well-known examples:%
\begin{eqnarray}
\int\limits_{\mathbb{Z}_{p}}e^{ax}d\mu _{1}\left( x\right) &=&\frac{a}{
e^{a}-1}  \label{aaABEr} \\
&=&\sum\limits_{n=0}^{\infty }B_{n}\frac{a^{n}}{n!},  \notag
\end{eqnarray}%
where $a\in \mathbb{E}$ with $a\neq 0$ (\textit{cf}. \cite[p. 172]{Schikof}%
). Using Taylor series for $e^{ax}$ in the left-hand side of the equation (%
\ref{aaABEr}), we have the following well-known the Witt's formula for the
Bernoulli numbers, $B_{n}$:%
\begin{equation}
B_{n}=\int\limits_{\mathbb{Z}_{p}}x^{n}d\mu _{1}\left( x\right) ,
\label{ABw}
\end{equation}%
where $n\in \mathbb{N}_{0}$ (\textit{cf}. \cite{Schikof}; see also \cite{T.
Kim}, \cite{Kim2006TMIC} and the references cited in each of these earlier
works). It is well-known that the denominator of the Bernoulli numbers $%
B_{n} $ is free of square. Consequently, for $n\in \mathbb{N}_{0}$, $%
x^{n}\in C^{1}(\mathbb{Z}_{p}\rightarrow \mathbb{Q)}$. Then we have%
\begin{equation*}
\left\vert B_{n}\right\vert _{p}=\left\vert \int\limits_{\mathbb{Z}
_{p}}x^{n}d\mu _{1}\left( x\right) \right\vert _{p}\leq p
\end{equation*}%
(\textit{cf}. \cite[p. 172]{Schikof}). Similarly, we have $p$-adic
representation for the Bernoulli polynomials as follows:%
\begin{equation}
\int\limits_{\mathbb{Z}_{p}}\left( z+x\right) ^{n}d\mu _{1}\left( x\right)
=B_{n}(z)  \label{wb}
\end{equation}%
where $n\in \mathbb{N}_{0}$ (\textit{cf}. \cite{T. Kim}, \cite{Kim2006TMIC}, 
\cite{Schikof}; see also the references cited in each of these earlier
works).

Let%
\begin{equation*}
P_{n}(x)=\sum\limits_{j=0}^{n}a_{j}x^{j}
\end{equation*}%
be a polynomial of degree $n$ ($n\in \mathbb{N}_{0}$). Substituting $%
P_{n}(x) $ into (\ref{M}), we have%
\begin{equation*}
\int\limits_{\mathbb{Z}_{p}}P_{n}(x)d\mu _{1}\left( x\right)
=\sum\limits_{j=0}^{n}a_{j}\int\limits_{\mathbb{Z}_{p}}x^{j}d\mu _{1}\left(
x\right) .
\end{equation*}%
Since $B_{2n+1}=0$ for $n\in \mathbb{N}$, by combining the above equation
with (\ref{ABw}), we thus have%
\begin{equation*}
\int\limits_{\mathbb{Z}_{p}}P_{n}(x)d\mu _{1}\left( x\right) =a_{0}-\frac{1}{
2}a_{1}+\sum\limits_{j=1}^{\left[ \frac{n}{2}\right] }a_{2j}B_{2j}.
\end{equation*}%
Similarly, substituting%
\begin{equation*}
f(x,t;\lambda )=\lambda ^{x}e^{tx}
\end{equation*}%
into (\ref{M}), we have%
\begin{equation}
\int\limits_{\mathbb{Z}_{p}}\lambda ^{x}e^{t\left( x+y\right) }d\mu
_{1}\left( x\right) =\frac{\log \lambda +t}{\lambda e^{t}-1}e^{ty},
\label{ALNu}
\end{equation}%
where $\lambda \in \mathbb{Z}_{p}$ (\textit{cf}. \cite{TkimJKMS}; see also 
\cite{jandY1}, \cite{srivas18}, \cite{simsek2017ascm}). Combining (\ref{ALNu}%
) with (\ref{laBN}), we have%
\begin{equation*}
\int\limits_{\mathbb{Z}_{p}}\lambda ^{x}\left( x+y\right) ^{n}d\mu
_{1}\left( x\right) =\mathfrak{B}_{n}(y;\lambda )
\end{equation*}

According to \cite{Shiratani}, \cite{Vudock}, \cite{KimTwisted} and \cite%
{MSKIM}, for each integer $N\geq 0$; $C_{p^{N}}$ denotes the multiplicative
group of the primitive $p^{N}$th roots of unity in $\mathbb{C}_{p}^{\ast }=%
\mathbb{C}_{p}\backslash \left\{ 0\right\} $.

Let%
\begin{equation*}
\mathbb{T}_{p}=\left\{ \xi \in \mathbb{C}_{p}:\xi ^{p^{N}}=1\text{, for }
N\geq 0\right\} =\cup _{N\geq 0}C_{p^{N}}
\end{equation*}%
In the sense of the $p$-adic Pontrjagin duality, the dual of $\mathbb{Z}_{p}$
is $\mathbb{T}_{p}=C_{p^{\infty }}$, the direct limit of cyclic groups $%
C_{p^{N}}$ of order $p^{N}$ with $N\geq 0$, with the discrete topology. $%
\mathbb{T}_{p}$ accept a natural $\mathbb{Z}_{p}$-module structure which can
be written briefly as $\xi ^{x}$ for $\xi \in \mathbb{T}_{p}$ and $x\in 
\mathbb{Z}_{p}$. $\mathbb{T}_{p}$ are embedded discretely in $\mathbb{C}_{p}$
as the multiplicative $p$-torsion subgroup. If $\xi \in \mathbb{T}_{p}$,
then $\vartheta _{\xi }:(\mathbb{Z}_{p},+)\rightarrow (\mathbb{C}_{p},.)$ is
the locally constant character, $x\rightarrow \xi ^{x}$, which is a locally
analytic character if $\xi \in \left\{ \xi \in \mathbb{C}_{p}:ord_{p}(\xi
-1)>0\right\} $. Consequently, it is well-known that $\vartheta _{\xi }$ has
a continuation to a continuous group homomorphism from $(\mathbb{Z}_{p},+)$
to $(\mathbb{C}_{p},.)$ (\textit{cf}. \cite{Shiratani}, \cite{Vudock}, \cite%
{KimTwisted}, \cite{MSKIM}, \cite{simsekCAM}; see also the references cited
in each of these earlier works).

We assume that $\lambda \in \mathbb{T}_{p}$. Then we have%
\begin{equation*}
\int\limits_{\mathbb{Z}_{p}}\lambda ^{x}x^{n}d\mu _{1}\left( x\right) = 
\mathcal{B}_{n}(\lambda )=\frac{nH_{n-1}(\lambda ^{-1})}{\lambda -1}
\end{equation*}%
(\textit{cf}. \cite{TkimJKMS}).

The Volkenborn integral of some trigonometric functions are given as follows:%
\begin{eqnarray*}
\int\limits_{\mathbb{Z}_{p}}\cos (ax)d\mu _{1}\left( x\right) &=&\frac{a\sin
(a)}{2(1-\cos (a))} \\
&=&\frac{a}{2}\cot \left( \frac{a}{2}\right) \\
&=&\sum\limits_{n=0}^{\infty }(-1)^{n}B_{2n}\frac{a^{2n}}{(2n)!},
\end{eqnarray*}%
where $a\in \mathbb{E}$ with $a\neq 0$, $p\neq 2$ (\textit{cf}. \cite[p. 172]%
{Schikof}, \cite{KIMjmaa2017}); 
\begin{equation*}
\int\limits_{\mathbb{Z}_{p}}\sin (ax)d\mu _{1}\left( x\right) =-\frac{a}{2},
\end{equation*}%
where $a\in \mathbb{E}$ (\textit{cf}. \cite[p. 170]{Schikof}, \cite%
{KIMjmaa2017}); and also 
\begin{equation*}
\int\limits_{\mathbb{Z}_{p}}\tan (ax)d\mu _{1}\left( x\right) =-\frac{a}{2}.
\end{equation*}

Note that 
\begin{equation*}
\int\limits_{\mathbb{Z}_{p}}\sinh (ax)d\mu _{1}\left( x\right) =\frac{1}{2}
\int\limits_{\mathbb{Z}_{p}}e^{ax}d\mu _{1}\left( x\right) -\frac{1}{2}
\int\limits_{\mathbb{Z}_{p}}e^{-ax}d\mu _{1}\left( x\right).
\end{equation*}%
Combining the above equation with (\ref{aaABEr}), we have%
\begin{equation*}
\int\limits_{\mathbb{Z}_{p}}\sinh (ax)d\mu _{1}\left( x\right) =\frac{1}{2} 
\frac{a}{e^{a}-1}+\frac{1}{2}\frac{a}{e^{-a}-1}=-\frac{a}{2}.
\end{equation*}

\subsection{$p$-adic integral over subsets:}

Let $V$ be a compact open subset of $\mathbb{Z}_{p}$. Let $f\in C^{1}(%
\mathbb{Z}_{p}\rightarrow \mathbb{K})$. Then we have%
\begin{equation*}
\int\limits_{V}f(x)d\mu _{1}\left( x\right) =\int\limits_{\mathbb{Z}
_{p}}g(x)d\mu _{1}\left( x\right) ,
\end{equation*}%
where%
\begin{equation*}
g(x)=\left\{ 
\begin{array}{cc}
f(x) & \text{if }x\in V \\ 
0 & \text{if }x\in \mathbb{Z}_{p}\setminus V%
\end{array}
\right.
\end{equation*}
(\textit{cf}. \cite[p. 174]{Schikof}).

\subsection{$p$-adic integral over \protect\boldmath{$j+p^{n}\mathbb{Z}_{p}$}%
:}

Let $f\in C^{1}(\mathbb{Z}_{p}\rightarrow \mathbb{K})$. Then we have%
\begin{equation}
\int\limits_{j+p^{n}\mathbb{Z}_{p}}f(x)d\mu _{1}\left( x\right)
=\int\limits_{p^{n}\mathbb{Z}_{p}}f(j+x)d\mu _{1}\left( x\right) =\frac{1}{
p^{n}}\int\limits_{\mathbb{Z}_{p}}f(j+p^{n}x)d\mu _{1}\left( x\right)
\label{aCoset}
\end{equation}%
(\textit{cf}. \cite[p. 175]{Schikof}). Substituting $f(x)=x^{m}$ with $m\in 
\mathbb{N}$ into (\ref{aCoset}), we have%
\begin{equation*}
\int\limits_{j+p^{n}\mathbb{Z}_{p}}x^{m}d\mu _{1}\left( x\right)
=p^{n(m-1)}B_{m}\left( \frac{j}{p^{n}}\right)
\end{equation*}%
(\textit{cf}. \cite[p. 175]{Schikof}).

We now give some examples for the above formula:

Let%
\begin{equation*}
\mathbf{T}_{p}=\mathbb{Z}_{p}\setminus p\mathbb{Z}_{p}
\end{equation*}%
and $f:\mathbf{T}_{p}\rightarrow \mathbb{Q}_{p}$ and a $C^{1}$-function and
also $f(-x)=-f(x)$ with $x\in \boldsymbol{T}_{p}$. Thus we have%
\begin{equation*}
\int\limits_{\boldsymbol{T}_{p}}f(x)d\mu _{1}\left( x\right) =0.
\end{equation*}%
Therefore%
\begin{equation*}
\int\limits_{\boldsymbol{T}_{p}}\frac{1}{x}d\mu _{1}\left( x\right)
=\int\limits_{\boldsymbol{T}_{p}}\frac{1}{x^{3}}d\mu _{1}\left( x\right)
=\int\limits_{\boldsymbol{T}_{p}}\frac{1}{x^{5}}d\mu _{1}\left( x\right)
=\cdots =\int\limits_{\boldsymbol{T}_{p}}\frac{1}{x^{2n+1}}d\mu _{1}\left(
x\right) =0,
\end{equation*}%
where $n\in \mathbb{N}$ (\textit{cf}. \cite[p. 175]{Schikof}) and%
\begin{equation}
\int\limits_{\boldsymbol{T}_{p}}x^{j}\left( x^{p-1}\right) ^{s}d\mu
_{1}\left( x\right) =\left( j+(p-1)s\right) \zeta _{p,j}(s),  \label{aZeta}
\end{equation}%
where $\zeta _{p,j}(s)$ denotes the $p$-adic zeta function, $\left\vert
s\right\vert _{p}<p^{\frac{p-2}{p-1}}$, $s\neq -\frac{j}{p-1}$ and $j\in
\left\{ 0,1,\ldots ,p-2\right\} $, $p\neq 2$ (\textit{cf}. \cite[p. 187]%
{Schikof}, \cite{srivas18}).

Substituting $s=n$ ($n\in \mathbb{N}_{0}$) into (\ref{aZeta}), we have
following values of the $p$-adic zeta function:%
\begin{equation*}
\int\limits_{\boldsymbol{T}_{p}}\left( x^{p-1}\right) ^{n}d\mu _{1}\left(
x\right) =\left( 1-p^{n\left( p-1\right) -1}\right) \frac{B_{n\left(
p-1\right) }}{n\left( p-1\right) },
\end{equation*}%
\begin{equation*}
\int\limits_{\boldsymbol{T}_{p}}x^{j}\left( x^{p-1}\right) ^{n}d\mu
_{1}\left( x\right) =\left( 1-p^{j-1+n\left( p-1\right) }\right) \frac{%
B_{j+n\left( p-1\right) }}{j+n\left( p-1\right) }
\end{equation*}%
whereas for $n\in \left\{ 2,4,6,8\ldots \right\} $, $j=0$ and $p=2$; and
consequently we also have%
\begin{equation*}
\int\limits_{\boldsymbol{T}_{p}}x^{n}d\mu _{1}\left( x\right) =\left(
1-2^{n-1}\right) \frac{B_{n}}{n}
\end{equation*}%
(\textit{cf}. \cite[p. 187]{Schikof}, \cite{srivas18}).

\subsection{\textbf{$p$-adic Integral of the Falling Factorial}}

Kim et al. \cite{DSkimDaehee} defined Witt-type identities for the Daehee
numbers of the first kind by the following $p$-adic integral representation
as follows:%
\begin{equation}
D_{n}=\int\limits_{\mathbb{Z}_{p}}x_{(n)}d\mu _{1}\left( x\right),
\label{Y1}
\end{equation}
or equivalently 
\begin{equation}
\int\limits_{\mathbb{Z}_{p}}x_{(n)}d\mu _{1}\left( x\right)
=\sum_{l=0}^{n}S_{1}(n,l)B_{l},  \label{aii3}
\end{equation}%
where $n\in \mathbb{N}_{0}$ (\textit{cf}. \cite{DSkimDaehee}).

Kim et al. \cite{DSkimDaehee} defined the Daehee numbers of the second kind
as follows:%
\begin{equation}
\widehat{D_{n}}=\int\limits_{\mathbb{Z}_{p}}t^{(n)}d\mu _{1}\left( t\right) .
\label{Y2}
\end{equation}

Kim et al. \cite{DSkimDaehee} defined the Daehee polynomials of the first
and second kind, respectively, as follows:%
\begin{equation}
D_{n}(x)=\int\limits_{\mathbb{Z}_{p}}\left( x+t\right) _{(n)}d\mu _{1}\left(
t\right)  \label{Y1a}
\end{equation}%
and%
\begin{equation}
\widehat{D_{n}}(x)=\int\limits_{\mathbb{Z}_{p}}\left( x+t\right) ^{(n)}d\mu
_{1}\left( t\right) ,  \label{Y2a}
\end{equation}%
where $n\in \mathbb{N}_{0}$.

Combining the following relation%
\begin{equation*}
x_{\left( n\right) }=n!\binom{x}{n},
\end{equation*}%
and (\ref{C7}) with (\ref{Y1}), we also have%
\begin{equation}
\int\limits_{\mathbb{Z}_{p}}x_{\left( n\right) }d\mu _{1}\left( x\right) = 
\frac{(-1)^{n}n!}{n+1},  \label{ak1}
\end{equation}%
where $n\in \mathbb{N}_{0}$ (\textit{cf}. \cite{DSkimDaehee}).

In \cite{DSkimDaehee}, Kim et al. gave the following formula:

\begin{equation}
D_{n}=\frac{(-1)^{n}n!}{n+1}.  \label{ak1d}
\end{equation}

By using (\ref{C7}), we have%
\begin{eqnarray}
\int\limits_{\mathbb{Z}_{p}}\binom{x+n-1}{n}d\mu _{1}\left( x\right)
&=&\sum_{m=0}^{n}\binom{n-1}{n-m}\int\limits_{\mathbb{Z}_{p}}\binom{x}{m}
d\mu _{1}\left( x\right)  \notag \\
&=&\sum_{m=1}^{n}(-1)^{m}\binom{n-1}{m-1}\frac{1}{m+1}  \label{C0} \\
&=&\sum_{m=0}^{n}(-1)^{m}\binom{n-1}{n-m}\frac{1}{m+1}  \notag
\end{eqnarray}%
(\textit{cf}. \cite{DSkim2}, \cite{AM2014}, \cite{DSkimDaehee}). By using (%
\ref{C0}), we obtain%
\begin{equation}
\int\limits_{\mathbb{Z}_{p}}\left( x+n-1\right) _{(n)}d\mu _{1}\left(
x\right) =n!\sum_{m=0}^{n}(-1)^{m}\binom{n-1}{n-m}\frac{1}{m+1}.  \label{1BI}
\end{equation}

\section{\textbf{Integral Formulas for the Volkenborn Integral}}

In \cite{simsekRJMP.mtjpam}, we gave the following interesting and new
integral formulas for the Volkenborn integral including the falling
factorial and the rising factorial with their identities and relations, the
combinatorial sums, the special numbers such as the Bernoulli numbers, the
Stirling numbers and the Lah numbers.

Using (\ref{Ro}) and (\ref{C7}), we \cite{simsekRJMP.mtjpam} gave the
following formula:

\begin{theorem}[\textit{cf}. \protect\cite{simsekRJMP.mtjpam}]
Let $n\in 
\mathbb{N}
_{0}$. Then we have 
\begin{equation}
\int\limits_{\mathbb{Z}_{p}}xx_{(n)}d\mu _{1}\left( x\right) =(-1)^{n+1} 
\frac{n!}{n^{2}+3n+2}.  \label{L1}
\end{equation}
\end{theorem}

By using (\ref{LahLAH}), we have the following formulas:

\begin{theorem}[\textit{cf}. \protect\cite{simsekRJMP.mtjpam}]
Let $n\in \mathbb{N}_{0}$. Then we have 
\begin{equation}
\int\limits_{\mathbb{Z}_{p}}xx_{(n)}d\mu _{1}\left( x\right)
=\sum_{k=1}^{n}S_{1}(n,k-1)B_{k}+B_{n+1}.  \label{L1-A}
\end{equation}
\end{theorem}

\begin{theorem}[\textit{cf}. \protect\cite{simsekRJMP.mtjpam}]
Let $n\in 
\mathbb{N}
$. Then we have 
\begin{equation}
\int\limits_{\mathbb{Z}_{p}}xx^{_{^{(n)}}}d\mu _{1}\left( x\right)
=\sum_{k=1}^{n}(-1)^{k+1}\binom{n-1}{k-1}\frac{n!}{k^{2}+3k+2}.
\label{LL-1a}
\end{equation}
\end{theorem}

By using 
\begin{equation}
xx^{(n)}=\sum_{k=1}^{n}C(n,k)x^{k+1}  \label{LL-1c}
\end{equation}%
and (\ref{ABw}), we \cite{simsekRJMP.mtjpam} gave the following formula:

\begin{theorem}[\textit{cf}. \protect\cite{simsekRJMP.mtjpam}]
Let $n\in 
\mathbb{N}
_{0}$. Then we have 
\begin{equation}
\int\limits_{\mathbb{Z}_{p}}xx^{_{^{(n)}}}d\mu _{1}\left( x\right)
=\sum_{k=1}^{n}C(n,k)B_{k+1}.  \label{LL-1b}
\end{equation}
\end{theorem}

\begin{theorem}[\textit{cf}. \protect\cite{simsekRJMP.mtjpam}]
Let $n\in \mathbb{N}_{0}$. Then we have 
\begin{equation*}
\int\limits_{\mathbb{Z}_{p}}\frac{x_{(n+1)}}{x}d\mu _{1}\left( x\right)
=\sum_{k=0}^{n}(-1)^{n}n_{(n-k)}\frac{k!}{k+1}.
\end{equation*}
\end{theorem}

By applying the Volkenborn integral to (\ref{IDD-1}), we have 
\begin{equation*}
\int\limits_{\mathbb{Z}_{p}}x_{(n+1)}d\mu _{1}\left( x\right)
=\sum\limits_{k=0}^{n}(-1)^{n-k}n_{(n-k)}\int\limits_{\mathbb{Z}
_{p}}xx_{(k)}d\mu _{1}\left( x\right) .
\end{equation*}%
By combining (\ref{ak1}), (\ref{v1a}) and (\ref{L1}) with the above
equation, we \cite{simsekRJMP.mtjpam} have the following formula:

\begin{theorem}[\textit{cf}. \protect\cite{simsekRJMP.mtjpam}]
Let $n\in 
\mathbb{N}
_{0}$. Then we have 
\begin{equation*}
\sum\limits_{k=0}^{n}n_{(n-k)}\frac{k!}{k^{2}+3k+2}=\frac{(n+1)!}{n+2}.
\end{equation*}
\end{theorem}

Applying the Volkenborn integral to (\ref{ab6}), we have 
\begin{equation*}
\int\limits_{\mathbb{Z}_{p}}\left( x+1\right) _{(n+1)}d\mu _{1}\left(
x\right) =\int\limits_{\mathbb{Z}_{p}}xx_{(n)}d\mu _{1}\left( x\right)
+\int\limits_{\mathbb{Z}_{p}}x_{(n)}d\mu _{1}\left( x\right) .
\end{equation*}%
Combining the above equation with (\ref{L1}) and (\ref{ak1}), after some
elementary calculations, we arrive at the following result:

\begin{corollary}[\textit{cf}. \protect\cite{simsekRJMP.mtjpam}]
Let $n\in 
\mathbb{N}
_{0}$. Then we have 
\begin{equation*}
\int\limits_{\mathbb{Z}_{p}}\left( x+1\right) _{(n+1)}d\mu _{1}\left(
x\right) =\frac{(-1)^{n}}{n+2}n!.
\end{equation*}
\end{corollary}

By applying the Volkenborn integral to (\ref{ai0a2}), we get the following
formula:%
\begin{equation}
\int\limits_{\mathbb{Z}_{p}}\binom{x+m}{n}d\mu _{1}\left( x\right)
=\sum_{m=0}^{n}\left( -1\right) ^{k}\binom{m}{n-k}\frac{1}{k+1}
\label{ai0a3}
\end{equation}%
(\textit{cf}. \cite{simsekRJMP.mtjpam}).

By applying the Volkenborn integral with respect to $x$ and $y$ to (\ref{cv}%
), we have 
\begin{eqnarray}
&&\int\limits_{\mathbb{Z}_{p}}\int\limits_{\mathbb{Z}_{p}}\sum
\limits_{k=0}^{n}\binom{x}{k}\binom{y}{n-k}d\mu _{1}\left( y\right) d\mu
_{1}\left( y\right)  \label{IR-11} \\
&=&\int\limits_{\mathbb{Z}_{p}}\int\limits_{\mathbb{Z}_{p}}\binom{x+y}{n}
d\mu _{1}\left( y\right) d\mu _{1}\left( y\right) .  \notag
\end{eqnarray}%
By combining the following identity with the above equation: 
\begin{equation}
\binom{x+y}{n}=\frac{1}{n!}\left( x+y\right) _{(n)}=\frac{1}{n!}
\sum\limits_{k=0}^{n}\binom{n}{k}x_{(k)}y_{(n-k)},  \label{LamdaFun-1d}
\end{equation}%
and using (\ref{Y1}) and (\ref{ak1}), we also get the following lemma:

\begin{lemma}[\textit{cf}. \protect\cite{simsekRJMP.mtjpam}]
Let $n\in 
\mathbb{N}
_{0}$. Then we have 
\begin{equation}
\int\limits_{\mathbb{Z}_{p}}\int\limits_{\mathbb{Z}_{p}}\binom{x+y}{n}d\mu
_{1}\left( y\right) d\mu _{1}\left( y\right) =\sum\limits_{k=0}^{n}(-1)^{n} 
\frac{1}{(k+1)(n-k+1)}.  \label{LamdaFun-1a}
\end{equation}
\end{lemma}

By combining (\ref{IR-11}) with the following identity: 
\begin{equation*}
\binom{x+y}{n}=\frac{1}{n!}\left( x+y\right) _{(n)}=\frac{1}{n!}
\sum\limits_{k=0}^{n}S_{1}(n,k)\left( x+y\right) ^{k},
\end{equation*}%
and using (\ref{Y1}) and (\ref{ak1}), we also get the following lemma:

\begin{lemma}[\textit{cf}. \protect\cite{simsekRJMP.mtjpam}]
Let $n\in 
\mathbb{N}
_{0}$. Then we have 
\begin{equation}
\int\limits_{\mathbb{Z}_{p}}\int\limits_{\mathbb{Z}_{p}}\binom{x+y}{n}d\mu
_{1}\left( y\right) d\mu _{1}\left( y\right) =\frac{1}{n!}
\sum\limits_{k=0}^{n}\sum\limits_{j=0}^{k}\binom{k}{j}S_{1}(n,k)B_{j}B_{k-j}.
\label{LamdaFun-1b}
\end{equation}
\end{lemma}

By using (\ref{ai0a4}), we have the following formula:

\begin{lemma}[\textit{cf}. \protect\cite{simsekRJMP.mtjpam}]
Let $n\in \mathbb{N}_{0}$. Then we have 
\begin{equation}
\int\limits_{\mathbb{Z}_{p}}\binom{x+1}{n}d\mu _{1}\left( x\right) =\frac{
(-1)^{n+1}}{n^{2}+n}.  \label{v1a}
\end{equation}
\end{lemma}

\begin{theorem}[\textit{cf}. \protect\cite{simsekRJMP.mtjpam}]
Let $n\in 
\mathbb{N}
_{0}$. Then we have 
\begin{equation*}
\int\limits_{\mathbb{Z}_{p}}\left( x+1\right) _{(n)}d\mu _{1}\left( x\right)
=(-1)^{n+1}\frac{n!}{n^{2}+n}.
\end{equation*}
\end{theorem}

\begin{corollary}[\textit{cf}. \protect\cite{simsekRJMP.mtjpam}]
Let $n\in \mathbb{N}$. Then we have 
\begin{equation*}
\int\limits_{\mathbb{Z}_{p}}\Delta x_{(n)}d\mu _{1}\left( x\right)
=(-1)^{n+1}\left( n-1\right) !.
\end{equation*}
\end{corollary}

By applying the Volkenborn integral to the equation (\ref{Lah}), we obtain%
\begin{equation*}
\int\limits_{\mathbb{Z}_{p}}\left( -x\right) _{(n)}d\mu _{1}\left( x\right)
=\sum_{k=0}^{n}L(n,k)\int\limits_{\mathbb{Z}_{p}}x_{(k)}d\mu _{1}\left(
x\right) ,
\end{equation*}%
where $n\in \mathbb{N}_{0}$.

By using (\ref{an4}), we get%
\begin{equation*}
\int\limits_{\mathbb{Z}_{p}}\left( -x\right) _{(n)}d\mu _{1}\left( x\right)
=\sum_{k=0}^{n}(-1)^{k}\frac{k!L(n,k)}{k+1},
\end{equation*}%
where $n\in \mathbb{N}_{0}$. Substituting (\ref{LAH-1a}) into the above
equation, we arrive at the following theorem:

\begin{theorem}[\textit{cf}. \protect\cite{simsekRJMP.mtjpam}]
Let $n\in \mathbb{N}$. Then we have 
\begin{equation*}
\int\limits_{\mathbb{Z}_{p}}\left( -x\right) _{(n)}d\mu _{1}\left( x\right)
=\sum_{k=1}^{n}(-1)^{k+n}\binom{n-1}{k-1}\frac{n!}{k+1}.
\end{equation*}
\end{theorem}

\begin{corollary}[\textit{cf}. \protect\cite{simsekRJMP.mtjpam}]
Let $n\in 
\mathbb{N}
_{0}$. Then we have 
\begin{equation}
\int\limits_{\mathbb{Z}_{p}}\binom{x+1}{n+1}d\mu _{1}\left( x\right) =\frac{
(-1)^{n}}{n^{2}+3n+2}.  \label{v1b}
\end{equation}
\end{corollary}

By applying the Volkenborn integral to the equation (\ref{LamdaFun-1v}) and (%
\ref{LamdaFun-1w}), respectively, we have the following results:

\begin{lemma}[\textit{cf}. \protect\cite{simsekRJMP.mtjpam}]
Let $k\in 
\mathbb{N}
_{0}$. Then we have 
\begin{equation*}
\int\limits_{\mathbb{Z}_{p}}\int\limits_{\mathbb{Z}_{p}}(xy)_{(k)}d\mu
_{1}\left( x\right) d\mu _{1}\left( y\right)
=\sum\limits_{l,m=1}^{k}D_{l}D_{m}C_{l,m}^{(k)}
\end{equation*}
and 
\begin{equation}
\int\limits_{\mathbb{Z}_{p}}\int\limits_{\mathbb{Z}_{p}}(xy)_{(k)}d\mu
_{1}\left( x\right) d\mu _{1}\left( y\right)
=\sum\limits_{l,m=1}^{k}(-1)^{l+m}\frac{l!m!}{(l+1)(m+1)}C_{l,m}^{(k)}.
\label{LamdaFun-1s}
\end{equation}
\end{lemma}

\begin{lemma}[\textit{cf}. \protect\cite{simsekRJMP.mtjpam}]
Let $k\in 
\mathbb{N}
_{0}$. Then we have 
\begin{equation}
\int\limits_{\mathbb{Z}_{p}}\int\limits_{\mathbb{Z}_{p}}(xy)_{(k)}d\mu
_{1}\left( x\right) d\mu _{1}\left( y\right)
=\sum\limits_{m=0}^{k}S_{1}(k,m)\left( B_{m}\right) ^{2}.
\label{LamdaFun-1u}
\end{equation}
\end{lemma}

By applying the Volkenborn integral to (\ref{Gg1}), and using (\ref{C7}), we
arrive the following result:

\begin{theorem}[\textit{cf}. \protect\cite{simsekRJMP.mtjpam}]
Let $n\in 
\mathbb{N}
_{0}$. Then we have 
\begin{equation*}
\int\limits_{\mathbb{Z}_{p}}x\binom{x-2}{n-1}d\mu _{1}\left( x\right)
=(-1)^{n}\sum\limits_{k=1}^{n}\frac{k}{k+1}.
\end{equation*}
\end{theorem}

By applying the Volkenborn integral to (\ref{Gg2}) and using (\ref{C7}), we
have the following result:

\begin{theorem}[\textit{cf}. \protect\cite{simsekRJMP.mtjpam}]
Let $n\in 
\mathbb{N}
_{0}$. Then we have 
\begin{equation*}
\int\limits_{\mathbb{Z}_{p}}\binom{n-x}{n}d\mu _{1}\left( x\right)
=(-1)^{n}H_{n},
\end{equation*}
where $H_{n}$ denotes the harmonic numbers given by 
\begin{equation}
H_{n}=\sum\limits_{k=0}^{n}\frac{1}{k+1}.  \label{AHn}
\end{equation}
\end{theorem}

By applying the Volkenborn integral to (\ref{Id-7}), and using (\ref{C7}),
we arrive at the following result:

\begin{theorem}[\textit{cf}. \protect\cite{simsekRJMP.mtjpam}]
Let $m\in 
\mathbb{N}
$ and $n\in 
\mathbb{N}
_{0}$. Then we have 
\begin{equation*}
\int\limits_{\mathbb{Z}_{p}}\binom{mx}{n}d\mu _{1}\left( x\right)
=\sum_{k=0}^{n}\frac{\left( -1\right) ^{k}}{k+1}\sum_{j=0}^{k}\left(
-1\right) ^{j}\binom{k}{j}\binom{mk-mj}{n}.
\end{equation*}
\end{theorem}

By applying the Volkenborn integral to the above equation (\ref{Id-5}), and
using (\ref{C7}), we arrive at the following theorem:

\begin{theorem}[\textit{cf}. \protect\cite{simsekRJMP.mtjpam}]
Let $n,r\in 
\mathbb{N}
_{0}$. Then we have 
\begin{equation}
\int\limits_{\mathbb{Z}_{p}}\binom{x}{n}^{r}d\mu _{1}\left( x\right)
=\sum_{k=0}^{nr}\frac{\left( -1\right) ^{k}}{k+1}\sum_{j=0}^{k}\left(
-1\right) ^{j}\binom{k}{j}\binom{k-j}{n}^{r}.  \label{IR-2}
\end{equation}
\end{theorem}

\begin{remark}
Substituting $r=1$ into (\ref{IR-2}), since $\binom{k-j}{n}=0$ if $k-j<n$,
we arrive at the equation (\ref{C7}).
\end{remark}

By applying the Volkenborn integral to (\ref{Id-6}), and using (\ref{C7}),
we arrive at the following theorem:

\begin{theorem}[\textit{cf}. \protect\cite{simsekRJMP.mtjpam}]
Let $n\in 
\mathbb{N}
$ with $n>1$. Then we have 
\begin{equation*}
\int\limits_{\mathbb{Z}_{p}}\left\{ x\binom{x-2}{n-1}+x\left( x-1\right) 
\binom{n-3}{n-2}\right\} d\mu _{1}\left( x\right) =\left( -1\right)
^{n}\sum_{k=0}^{n}\frac{k^{2}}{k+1}.
\end{equation*}
\end{theorem}

By applying the Volkenborn integral to the above equations (\ref{Id-1a}) and
(\ref{Id-2b}), using (\ref{C7}) and (\ref{ABw}), respectively, we arrive at
the following theorem:

\begin{theorem}[\textit{cf}. \protect\cite{simsekRJMP.mtjpam}]
Let $n\in 
\mathbb{N}
_{0}$. Then we have 
\begin{equation}
\int\limits_{\mathbb{Z}_{p}}\binom{x+n}{n}d\mu _{1}\left( x\right)
=\sum_{k=0}^{n}\frac{\left( -1\right) ^{k}}{k+1}\sum_{j=0}^{k}\left(
-1\right) ^{j}\binom{k}{j}\binom{k-j+n}{n}  \label{Id-1}
\end{equation}
and 
\begin{equation}
\int\limits_{\mathbb{Z}_{p}}\binom{x+n}{n}d\mu _{1}\left( x\right)
=\sum_{k=0}^{n}B_{k}\sum_{j=0}^{n}\binom{n}{j}\frac{S_{1}(j,k)}{j!}.
\label{Id-2}
\end{equation}
\end{theorem}

By applying the Volkenborn integral to (\ref{1BIaa}), and using (\ref{C7}),
we arrive at the following theorem:

\begin{theorem}[\textit{cf}. \protect\cite{simsekRJMP.mtjpam}]
Let $n\in 
\mathbb{N}
_{0}$. Then we have 
\begin{equation*}
\int\limits_{\mathbb{Z}_{p}}\binom{x+n+\frac{1}{2}}{n}d\mu _{1}\left(
x\right) =\binom{2n}{n}\sum\limits_{k=0}^{n}(-1)^{k}\binom{n}{k}\frac{
2^{2k-2n}\left( 2n+1\right) }{\left( k+1\right) \left( 2k+1\right) \binom{2k 
}{k}}.
\end{equation*}
\end{theorem}

By applying the Volkenborn integral to (\ref{S1a}), and using (\ref{ABw}),
we get the following result:

\begin{theorem}[\textit{cf}. \protect\cite{simsekRJMP.mtjpam}]
Let $m,n\in 
\mathbb{N}
_{0}$. Then we have 
\begin{equation}
\int\limits_{\mathbb{Z}_{p}}x^{m}x_{(n)}d\mu _{1}\left( x\right)
=\sum\limits_{k=0}^{n}S_{1}(n,k)B_{k+m}.  \label{aS1B}
\end{equation}
\end{theorem}

By applying the Volkenborn integral to (\ref{1BIb1}), and using (\ref{ABw}),
we get the following lemma:

\begin{lemma}[\textit{cf}. \protect\cite{simsekRJMP.mtjpam}]
Let $m,n\in 
\mathbb{N}
_{0}$. Then we have 
\begin{equation}
\int\limits_{\mathbb{Z}_{p}}x_{(n)}x_{(m)}d\mu _{1}\left( x\right)
=\sum_{j=0}^{n}\sum_{l=0}^{m}S_{1}(n,k)S_{1}(m,l)B_{j+l}.
\label{LamdaFun-1h}
\end{equation}
\end{lemma}

\begin{theorem}[\textit{cf}. \protect\cite{simsekRJMP.mtjpam}]
Let $m,n\in 
\mathbb{N}
_{0}$. Then we have 
\begin{equation}
\int\limits_{\mathbb{Z}_{p}}x_{(n)}x_{(m)}d\mu _{1}\left( x\right)
=\sum\limits_{k=0}^{m}(-1)^{m+n-k}\binom{m}{k}\binom{n}{k}\frac{k!(m+n-k)!}{
m+n-k+1}.  \label{1LaHv}
\end{equation}
\end{theorem}

By applying the Volkenborn integral to (\ref{1BIb}), we get the following
lemma:

\begin{lemma}[\textit{cf}. \protect\cite{simsekRJMP.mtjpam}]
Let $m,n\in 
\mathbb{N}
_{0}$. Then we have 
\begin{equation}
\int\limits_{\mathbb{Z}_{p}}x_{(n)}x_{(m)}d\mu _{1}\left( x\right)
=\sum\limits_{k=0}^{m}\binom{m}{k}\binom{n}{k}k!\sum
\limits_{l=0}^{m+n-k}S_{1}(m+n-k,l)B_{l}.  \label{LamdaFun-1i}
\end{equation}
\end{lemma}

Let $z\in \mathbb{C}_{p}$, we have%
\begin{equation}
D_{m}\left( z:q\right) =\int_{\mathbb{Z}_{p}}\left[ x\right] ^{m}d\mu
_{z}\left( x\right)  \label{1.7}
\end{equation}%
(\textit{cf}. \cite{KimITSFdahee}). If we take $z=q$ in (\ref{1.7}), then we
see that $D_{m}\left( q:q\right) =\beta _{m}\left( q\right) $, Carlitz's $q$%
-Bernoulli numbers. In the case when $z=u$ in (\ref{1.7}), $q$-Daehee
numbers and $q$-Daehee polynomials are defined, respectively, as follows:%
\begin{equation}
D_{m}\left( u:q\right) =\int_{\mathbb{Z}_{p}}\left[ x\right] ^{m}d\mu
_{u}\left( x\right) =H_{m}\left( u^{-1}:q\right)  \label{1.8}
\end{equation}%
and%
\begin{equation*}
D_{m}\left( z,x:q\right) =\int_{\mathbb{Z}_{p}}\left[ x+t\right] ^{m}d\mu
_{z}\left( t\right) ,
\end{equation*}%
where $m\in \mathbb{N}_{0}$ and $z\in \mathbb{C}_{p}$ (\textit{cf}. \cite%
{KimITSFdahee}, \cite[Eq. (1.10)]{srivas18}).

\begin{theorem}
\textup{(\textit{cf}. \cite[Theorem 1]{SimsekqDahee})} Assume that $a,b$ are
integers with $\left( a,b\right) =\left( p,b\right) =1$. Let 
\begin{equation*}
S_{q}\left( a,b:n;q^{l}\right) =\sum_{M=1}^{k-1}\frac{\left[ M\right] }{ %
\left[ b\right] }\int_{\mathbb{Z}_{p}}q^{-lx}\left[ x+\left\{ \frac{aM}{b}
\right\} :q^{l}\right] ^{n}d\mu _{q^{l}}\left( x\right).
\end{equation*}
\end{theorem}

Observe that%
\begin{equation*}
D_{m}\left( u:q\right) =H_{m}\left( u^{-1}:q\right)
\end{equation*}%
(\textit{cf}. \cite[Eq. (5)]{KimITSFdahee}).

\begin{corollary}
\textup{(\textit{cf}. \cite[Corollary 1]{SimsekqDahee})} If $z=\lambda $,
then we have 
\begin{equation}
B_{q}\left( d,c:0,\lambda :m\right) =\frac{m}{\left[ c^{m}\right] }
\sum_{\lambda }\frac{1}{\left[ \lambda -1\right] \left[ \lambda ^{-d}-1 %
\right] }\int_{\mathbb{Z}_{p}}\left[ t\right] ^{m}d\mu _{\lambda }\left(
t\right) .  \label{2.4}
\end{equation}
\end{corollary}

\begin{remark}
If $q\rightarrow 1$, then (\ref{2.4}) is reduced to (\ref{2.1}). That is, 
\begin{equation*}
\lim_{q\rightarrow 1}B_{q}\left( d,c:0,\lambda :m\right) =S\left(
d,c:m\right)
\end{equation*}
where 
\begin{equation}
S\left( d,c:m\right) =\frac{m}{c^{m}}\sum_{\lambda }\frac{H_{m-1}\left(
\lambda ^{-1}\right) }{\left( \lambda -1\right) \left( \lambda
^{-d}-1\right) },  \label{2.1}
\end{equation}
where $\lambda$ runs through the $c$th roots of unity distinct from $1$ and $%
H_{m}\left( \lambda \right) $ is the Frobenius-Euler numbers (\textit{cf}. 
\cite[Remark 1]{SimsekqDahee}).

Now, it is time to raise the following question:

Is it possible to give any reciprocity law for the $q$-Dedekind type sum $%
B_{q}\left( d,c:0,\lambda :m\right)$. That is, how can we calculate the
following relation: 
\begin{equation*}
B_{q}\left( d,c:0,\lambda :m\right) +B_{q}\left( c,d:0,\lambda :m\right) =?
\end{equation*}
\end{remark}

\subsection{\textbf{Some Properties of the Fermionic $p$-adic Integral}}

Here, we give some well-known properties of the fermionic $p$-adic integral.

Let $f\in C^{1}(\mathbb{Z}_{p}\rightarrow \mathbb{K})$. Kim \cite%
{KIMjmaa2017} gave the following integral equation for the fermionic $p$%
-adic integral on $\mathbb{Z}_{p}$: 
\begin{eqnarray}
&&\int\limits_{\mathbb{Z}_{p}}E^{n}\left[ f\left( x\right) \right] d\mu
_{-1}\left( x\right) +(-1)^{n+1}\int\limits_{\mathbb{Z}_{p}}f\left( x\right)
d\mu _{-1}\left( x\right)  \label{TKint} \\
&&=2\sum_{j=0}^{n-1}(-1)^{n-1-j}f(j),  \notag
\end{eqnarray}
where $n\in \mathbb{N}$.

Substituting $n=1$ into (\ref{TKint}), we have very useful integral
equation, which is used to construct generating functions associated with
Euler-type numbers and polynomials, given as follows: 
\begin{equation}
\int\limits_{\mathbb{Z}_{p}}f\left( x+1\right) d\mu _{-1}\left( x\right)
+\int\limits_{\mathbb{Z}_{p}}f\left( x\right) d\mu _{-1}\left( x\right)
=2f(0)  \label{MmmA}
\end{equation}
(\textit{cf}. \cite{KIMjmaa2017}).

By using (\ref{Mmm}) and (\ref{MmmA}), the well-known Witt's type formulas
for the Euler numbers and polynomials are given as follows, respectively: 
\begin{equation}
E_{n}=\int\limits_{\mathbb{Z}_{p}}x^{n}d\mu _{-1}\left( x\right)  \label{Mm1}
\end{equation}%
and%
\begin{equation}
E_{n}(z)=\int\limits_{\mathbb{Z}_{p}}\left( z+x\right) ^{n}d\mu _{-1}\left(
x\right) ,  \label{we}
\end{equation}%
where $n\in \mathbb{N}_{0}$ (\textit{cf}. \cite{Kim2006TMIC}, \cite{KIMjang}%
; see also the references cited in each of these earlier works).

\begin{theorem}
\label{ThoremKIM} Let $n\in \mathbb{N}_{0}$. Then we have 
\begin{equation}
\int\limits_{\mathbb{Z}_{p}}\binom{x}{n}d\mu _{-1}\left( x\right)
=(-1)^{n}2^{-n}.  \label{est-3}
\end{equation}
\end{theorem}

Theorem \ref{ThoremKIM} was proved by Kim et al. \cite[Theorem 2.3]{DSkim2}.

Substituting $x_{\left( n\right) }=n!\binom{x}{n}$ into (\ref{est-3}), we
have the following well-known identity:%
\begin{equation}
\int\limits_{\mathbb{Z}_{p}}x_{(n)}d\mu _{-1}\left( x\right)
=(-1)^{n}2^{-n}n!  \label{ak2}
\end{equation}%
where $n\in \mathbb{N}_{0}$ (\textit{cf}. \cite{DSkim2}).

Recently, by using the fermionic $p$-adic integral on $\mathbb{Z}_{p}$, Kim
et al. \cite{DSkim2} defined the Changhee numbers of the first and the
second kind, respectively, as follows: 
\begin{equation}
Ch_{n}=\int\limits_{\mathbb{Z}_{p}}x_{(n)}d\mu _{-1}\left( x\right)
\label{y1}
\end{equation}%
and%
\begin{equation}
\widehat{Ch}_{n}=\int\limits_{\mathbb{Z}_{p}}x^{(n)}d\mu _{-1}\left(
x\right) ,  \label{y2}
\end{equation}%
where $n\in \mathbb{N}_{0}$.

For $n\in \mathbb{N}_{0}$, Kim et al. \cite{DSkim2} gave the following
formula for the Changhee numbers of the first kind:%
\begin{equation}
Ch_{n}=(-1)^{n}2^{-n}n!.  \label{yy1}
\end{equation}

Kim et al. \cite{DSkim2} also defined the Changhee polynomials of the first
and the second, respectively, as follows:%
\begin{equation}
Ch_{n}(x)=\int\limits_{\mathbb{Z}_{p}}\left( x+t\right) _{(n)}d\mu
_{-1}\left( t\right)  \label{y1a}
\end{equation}%
and 
\begin{equation}
\widehat{Ch}_{n}(x)=\int\limits_{\mathbb{Z}_{p}}\left( x+t\right) ^{(n)}d\mu
_{-1}\left( t\right) .  \label{y2a}
\end{equation}%
Therefore, by using Theorem \ref{ThoremKIM}, we have%
\begin{eqnarray}
\int\limits_{\mathbb{Z}_{p}}\binom{x+n-1}{n}d\mu _{-1}\left( x\right)
&=&\sum_{m=0}^{n}\binom{n-1}{n-m}\int\limits_{\mathbb{Z}_{p}}\binom{x}{m}
d\mu _{-1}\left( x\right)  \notag \\
&=&\sum_{m=1}^{n}(-1)^{m}\binom{n-1}{m-1}2^{-m}  \label{Ca-1} \\
&=&\sum_{m=0}^{n}(-1)^{m}\binom{n-1}{n-m}2^{-m}  \notag
\end{eqnarray}%
(\textit{cf}. \cite{DSkim2}, \cite{AM2014}, \cite{DSkimDaehee}). By using (%
\ref{Ca-1}), we have%
\begin{equation}
\int\limits_{\mathbb{Z}_{p}}\left( x+n-1\right) _{(n)}d\mu _{-1}\left(
x\right) =n!\sum_{m=0}^{n}(-1)^{m}\binom{n-1}{n-m}2^{-m}.  \label{1FI}
\end{equation}

By using (\ref{TKint}), Kim \cite{KIMaml2008} modified (\ref{Mmm}). He gave
the following integral equation:%
\begin{equation}
q^{d}\int\limits_{\mathbb{Z}_{p}}E^{d}f\left( x\right) d\mu _{-q}\left(
x\right) +\int\limits_{\mathbb{Z}_{p}}f\left( x\right) d\mu _{-q}\left(
x\right) =\left[ 2\right] \sum_{j=0}^{d-1}(-1)^{j}q^{j}f(j),
\label{TkimFint}
\end{equation}%
where $d$ is an positive odd integer.

Some examples for the fermionic $p$-adic integral are given as follows:

The Volkenborn integral of some trigonometric functions are given as follows:%
\begin{equation*}
\int\limits_{\mathbb{Z}_{p}}\cos (ax)d\mu _{-1}\left( x\right) =1,
\end{equation*}%
where $a\in \mathbb{E}$ with $a\neq 0$, $p\neq 2$ (\textit{cf}. \cite%
{KIMjmaa2017}); 
\begin{equation*}
\int\limits_{\mathbb{Z}_{p}}\sin (a\left( x+1\right) )d\mu _{-1}\left(
x\right) =-\int\limits_{\mathbb{Z}_{p}}\sin (ax)d\mu _{-1}\left( x\right) ,
\end{equation*}%
where $a\in \mathbb{E}$ (\textit{cf}. \cite{KIMjmaa2017}); and also%
\begin{equation*}
\int\limits_{\mathbb{Z}_{p}}\sin (ax)d\mu _{-1}\left( x\right) =-\frac{\sin
(a)}{\cos (a)+1}.
\end{equation*}

Note that 
\begin{equation*}
\int\limits_{\mathbb{Z}_{p}}\sinh (ax)d\mu _{-1}\left( x\right) =\frac{1}{2}
\int\limits_{\mathbb{Z}_{p}}e^{ax}d\mu _{-1}\left( x\right) -\frac{1}{2}
\int\limits_{\mathbb{Z}_{p}}e^{-ax}d\mu _{-1}\left( x\right).
\end{equation*}%
Combining the above equation with (\ref{MmmA}), we have%
\begin{equation*}
\int\limits_{\mathbb{Z}_{p}}\sinh (ax)d\mu _{-1}\left( x\right) =\frac{1}{
e^{a}+1}+\frac{1}{e^{-a}+1}=1.
\end{equation*}

Let%
\begin{equation*}
P_{n}(x)=\sum\limits_{j=0}^{n}a_{j}x^{j}
\end{equation*}%
be a polynomial of degree $n$ ($n\in \mathbb{N}_{0}$). Substituting $%
P_{n}(x) $ into (\ref{Mmm}), we have%
\begin{equation*}
\int\limits_{\mathbb{Z}_{p}}P_{n}(x)d\mu _{-1}\left( x\right)
=\sum\limits_{j=0}^{n}a_{j}\int\limits_{\mathbb{Z}_{p}}x^{j}d\mu _{-1}\left(
x\right) .
\end{equation*}%
Since $E_{2n}=0$ for $n\in \mathbb{N}$, by combining the above equation with
(\ref{Mm1}), we thus have%
\begin{equation*}
\int\limits_{\mathbb{Z}_{p}}P_{n}(x)d\mu _{-1}\left( x\right)
=1+\sum\limits_{j=0}^{\left[ \frac{n+1}{2}\right] }a_{2j+1}E_{2j+1}.
\end{equation*}%
By using (\ref{AEMus}), we have%
\begin{equation}
\int\limits_{\mathbb{X}}d\mu _{\mathcal{E},k,\lambda }\left( x+fp^{N}\mathbb{%
\ Z}_{p}\right) =\mathcal{E}_{k}\left( \lambda \right) ,  \label{AEMus1}
\end{equation}
where $\lambda \in \mathbb{Z}_{p}$ (\textit{cf}. \cite{ozdenSmsekCangul}, 
\cite{OzdenAMC2014}).

Substituting%
\begin{equation*}
g(x,t;\lambda )=\lambda ^{x}e^{tx}
\end{equation*}%
into (\ref{MmmA}), we have the following well-known formula:%
\begin{equation*}
\int\limits_{\mathbb{Z}_{p}}g(x,t;\lambda )d\mu _{-1}\left( x\right) =\frac{%
2 }{\lambda e^{t}+1}.
\end{equation*}%
Combining the above equation with (\ref{Cad3}), we have%
\begin{equation}
\int\limits_{\mathbb{Z}_{p}}\lambda ^{x}x^{n}d\mu _{-1}\left( x\right) = 
\mathcal{E}_{n}\left( \lambda \right).  \label{AEMus2}
\end{equation}

By assuming that $\chi $ is the primitive Dirichlet's character with odd
conductor $f$, Rim and Kim \cite{RimKim} gave the following formula:%
\begin{equation*}
\int\limits_{\mathbb{X}}\chi (x)d\mu _{E,k}\left( x+fp^{N}\mathbb{Z}
_{p}\right) =E_{k,\chi }
\end{equation*}%
where%
\begin{equation*}
\frac{2}{e^{ft}+1}\sum_{j=0}^{f-1}(-1)^{j}\chi (j)e^{tj}=\sum_{n=0}^{\infty
}E_{n,\chi }\frac{t^{n}}{n!}.
\end{equation*}%
Combining (\ref{AEMus1}) with (\ref{AEMus2}), we have the following
well-known relation:%
\begin{equation}
d\mu _{\mathcal{E},k,\lambda }\left( x+fp^{N}\mathbb{Z}_{p}\right) =\lambda
^{x}x^{n}d\mu _{-1}\left( x+p^{N}\mathbb{Z}_{p}\right) .  \label{AEMus3}
\end{equation}%
Setting $\lambda =1$ in (\ref{AEMus3}), we have%
\begin{equation*}
d\mu _{E,k}\left( x+fp^{N}\mathbb{Z}_{p}\right) =x^{n}d\mu _{-1}\left(
x+p^{N}\mathbb{Z}_{p}\right)
\end{equation*}%
or equivalently 
\begin{equation*}
d\mu _{E,k}\left( x\right) =x^{n}d\mu _{-1}\left( x\right)
\end{equation*}%
(\textit{cf}. \cite{RimKim}).

Therefore, combining (\ref{RelationApostolEnBn}) with (\ref{AEMus3}), we have%
\begin{equation*}
d\mu _{\mathcal{E},k,\lambda }\left( x+fp^{N}\mathbb{Z}_{p}\right) =-\frac{2 
}{k+1}d\mu _{\mathcal{B},k,-\lambda }\left( x+fp^{N}\mathbb{Z}_{p}\right) .
\end{equation*}

\section{\textbf{Integral Formulas for the Fermionic $p$-adic Integral}}

In \cite{simsekRJMP.mtjpam}, we gave the following interesting and new
integral formulas for the fermionic $p$-adic integral including the falling
factorial and the rising factorial with their identities and relations, the
combinatorial sums, the special numbers such as the Euler numbers, the
Stirling numbers and the Lah numbers.

By applying the $p$-adic fermionic integral to the both sides of equation (%
\ref{Ro}) and using (\ref{ak2}), we have the following theorem:

\begin{theorem}[\textit{cf}. \protect\cite{simsekRJMP.mtjpam}]
Let $n\in 
\mathbb{N}
$. Then we have 
\begin{equation}
\int\limits_{\mathbb{Z}_{p}}xx_{(n)}d\mu _{-1}\left( x\right) =(-1)^{n}\frac{
(n-1)}{2^{n+1}}n!.  \label{ab7}
\end{equation}
\end{theorem}

By applying the $p$-adic fermionic integral to the both sides of equation (%
\ref{ak8}), and using (\ref{ab7}) and (\ref{alah}), we have the following
theorem:

\begin{theorem}[\textit{cf}. \protect\cite{simsekRJMP.mtjpam}]
Let $n\in 
\mathbb{N}
_{0}$. Then we have 
\begin{equation}
\int\limits_{\mathbb{Z}_{p}}xx^{_{^{(n)}}}d\mu _{-1}\left( x\right)
=\sum_{k=1}^{n}(-1)^{k}\binom{n-1}{k-1}\frac{(k-1)}{2^{k+1}}n!.  \label{ab7a}
\end{equation}
\end{theorem}

By applying the $p$-adic fermionic integral to (\ref{ab6a}) and using (\ref%
{ak2}), we have the following theorem:

\begin{theorem}[\textit{cf}. \protect\cite{simsekRJMP.mtjpam}]
Let $n\in 
\mathbb{N}
_{0}$. Then we have 
\begin{equation}
\int\limits_{\mathbb{Z}_{p}}\left( x+1\right) _{(n)}d\mu _{-1}\left(
x\right) =(-1)^{n+1}\frac{1}{2^{n}}n!.  \label{v1-B}
\end{equation}
\end{theorem}

By applying the $p$-adic fermionic integral to equation (\ref{IDD-1}), and
using (\ref{ak2}), we arrive at the following theorem:

\begin{theorem}[\textit{cf}. \protect\cite{simsekRJMP.mtjpam}]
Let $n\in 
\mathbb{N}
_{0}$. Then we have 
\begin{equation*}
\int\limits_{\mathbb{Z}_{p}}\frac{x_{(n+1)}}{x}d\mu _{-1}\left( x\right)
=\sum_{k=0}^{n}(-1)^{n}n_{(n-k)}\frac{k!}{2^{k}}.
\end{equation*}
\end{theorem}

\begin{lemma}[\textit{cf}. \protect\cite{simsekRJMP.mtjpam}]
Let $k\in 
\mathbb{N}
_{0}$. Then we have 
\begin{equation}
\int\limits_{\mathbb{Z}_{p}}\int\limits_{\mathbb{Z}_{p}}(xy)_{(k)}d\mu
_{-1}\left( x\right) d\mu _{-1}\left( y\right)
=\sum\limits_{l,m=1}^{k}(-1)^{l+m}2^{-m-l}l!m!C_{l,m}^{(k)}.
\label{LamdaFun-1y}
\end{equation}
\end{lemma}

\begin{lemma}[\textit{cf}. \protect\cite{simsekRJMP.mtjpam}]
Let $k\in 
\mathbb{N}
_{0}$. Then we have 
\begin{equation}
\int\limits_{\mathbb{Z}_{p}}\int\limits_{\mathbb{Z}_{p}}(xy)_{(k)}d\mu
_{-1}\left( x\right) d\mu _{-1}\left( y\right)
=\sum\limits_{m=0}^{k}S_{1}(k,m)\left( E_{m}\right) ^{2}.
\label{LamdaFun-1z}
\end{equation}
\end{lemma}

\begin{theorem}[\textit{cf}. \protect\cite{simsekRJMP.mtjpam}]
Let $n\in \mathbb{N}$ with $n>1$. Then we have 
\begin{equation*}
\int\limits_{\mathbb{Z}_{p}}\left\{ x\binom{x-2}{n-1}+x\left( x-1\right) 
\binom{n-3}{n-2}\right\} d\mu _{-1}\left( x\right) =\left( -1\right)
^{n}\sum_{k=0}^{n}\frac{k^{2}}{2^{k}}.
\end{equation*}
\end{theorem}

\begin{theorem}[\textit{cf}. \protect\cite{simsekRJMP.mtjpam}]
Let $n\in 
\mathbb{N}
_{0}$. Then we have 
\begin{equation}
\int\limits_{\mathbb{Z}_{p}}\binom{x+n}{n}d\mu _{-1}\left( x\right)
=\sum_{k=0}^{n}\frac{\left( -1\right) ^{k}}{2^{k}}\sum_{j=0}^{k}\left(
-1\right) ^{j}\binom{k}{j}\binom{k-j+n}{n}  \label{Id-3}
\end{equation}
and 
\begin{equation}
\int\limits_{\mathbb{Z}_{p}}\binom{x+n}{n}d\mu _{-1}\left( x\right)
=\sum_{k=0}^{n}E_{k}\sum_{j=0}^{n}\binom{n}{j}\frac{S_{1}(j,k)}{j!}.
\label{Id-4}
\end{equation}
\end{theorem}

\begin{theorem}[\textit{cf}. \protect\cite{simsekRJMP.mtjpam}]
Let $m\in 
\mathbb{N}
$ and $n\in 
\mathbb{N}
_{0}$. Then we have 
\begin{equation*}
\int\limits_{\mathbb{Z}_{p}}\binom{mx}{n}d\mu _{-1}\left( x\right)
=\sum_{k=0}^{n}\frac{\left( -1\right) ^{k}}{2^{k}}\sum_{j=0}^{k}\left(
-1\right) ^{j}\binom{k}{j}\binom{mk-mj}{n}.
\end{equation*}
\end{theorem}

\begin{theorem}[\textit{cf}. \protect\cite{simsekRJMP.mtjpam}]
Let $n,r\in 
\mathbb{N}
_{0}$. Then we have 
\begin{equation}
\int\limits_{\mathbb{Z}_{p}}\binom{x}{n}^{r}d\mu _{-1}\left( x\right)
=\sum_{k=0}^{nr}\frac{\left( -1\right) ^{k}}{2^{k}}\sum_{j=0}^{k}\left(
-1\right) ^{j}\binom{k}{j}\binom{k-j}{n}^{r}.  \label{IR-1}
\end{equation}
\end{theorem}

\begin{theorem}[\textit{cf}. \protect\cite{simsekRJMP.mtjpam}]
Let $n\in 
\mathbb{N}
_{0}$. Then we have 
\begin{equation*}
\int\limits_{\mathbb{Z}_{p}}\binom{n-x}{n}d\mu _{-1}\left( x\right)
=(-1)^{n}\sum\limits_{k=1}^{n}2^{-k}.
\end{equation*}
\end{theorem}

By applying the $p$-adic fermionic integral to (\ref{1BIaa}), and using (\ref%
{est-3}), we arrive at the following theorem:

\begin{theorem}[\textit{cf}. \protect\cite{simsekRJMP.mtjpam}]
Let $n\in 
\mathbb{N}
_{0}$. Then we have 
\begin{equation*}
\int\limits_{\mathbb{Z}_{p}}\binom{x+n+\frac{1}{2}}{n}d\mu _{-1}\left(
x\right) =\left( 2n+1\right) \binom{2n}{n}\sum\limits_{k=0}^{n}(-1)^{k} 
\binom{n}{k}\frac{2^{k-2n}}{\left( 2k+1\right) \binom{2k}{k}}.
\end{equation*}
\end{theorem}

By using (\ref{TKint}), Kim et al. \cite[Theorem 2.1]{KimRyooBERNSTEIN}
proved the following theorem:

\begin{theorem}[\textit{cf}. \protect\cite{KimRyooBERNSTEIN}]
Let $n\in 
\mathbb{N}
_{0}$. Then we have 
\begin{equation}
\int\limits_{\mathbb{Z}_{p}}(1-x)^{n}d\mu _{-1}\left( x\right)
=2+\int\limits_{\mathbb{Z}_{p}}x^{n}d\mu _{-1}\left( x\right) .
\label{A.Berns.1}
\end{equation}
\end{theorem}

By using (\ref{A.Berns.1}), Kim et al. \cite[Theorem 2.1]{KimRyooBERNSTEIN}
proved the following theorem:

\begin{theorem}[\textit{cf}. \protect\cite{KimRyooBERNSTEIN}]
Let $k,n\in 
\mathbb{N}
_{0}$ with $0\leq k\leq n$. If $k=0$, we have 
\begin{equation}
\int\limits_{\mathbb{Z}_{p}}B_{k}^{n}(x)d\mu _{-1}\left( x\right) =2+E_{n}
\label{A.Berns.3}
\end{equation}
and if $k>0$, we have 
\begin{equation}
\int\limits_{\mathbb{Z}_{p}}B_{k}^{n}(x)d\mu _{-1}\left( x\right) =\binom{n}{
k}\sum\limits_{j=0}^{n-k}(-1)^{n-k-j}\binom{n-k}{j}E_{n-j}.
\label{A.Berns.4}
\end{equation}
\end{theorem}

In \cite{Kim Symmetry}, Kim et al. gave the following formula:%
\begin{equation}
\int\limits_{\mathbb{Z}_{p}}\left( -y^{w}\left( e^{t}-1\right) ^{w}\right)
^{x}d\mu _{-1}\left( x\right) =\frac{2}{1-y^{w}\left( e^{t}-1\right) ^{w}},
\label{1aSs1d}
\end{equation}%
where $w\in \mathbb{N}$. By using the above formula, they defined so-called $%
w$-torsion Fubini polynomials. If $w=y=1$, right-hand side of the equation (%
\ref{1aSs1d}) reduces to generating function for the Fubini numbers (\textit{%
cf}. \cite{NeslihanJKMS}).

\section{\textbf{New Integral Formulas Involving Volkenborn Integral}}

In this section, we give some new integral formulas for the Volkenborn
integral. These new formulas are related to some special functions, special
numbers and polynomials such as rising factorial and the falling factorial,
the Bernoulli numbers and polynomials, the Euler numbers and polynomials,
the Stirling numbers, the Lah numbers, the Peters numbers and polynomials,
the central factorial numbers, the Daehee numbers and polynomials, the
Changhee numbers and polynomials, the Harmonic numbers, the Fubini numbers,
combinatorial numbers and sums.

\begin{theorem}
Let $m,n\in 
\mathbb{N}
$. Then we have 
\begin{equation}
\int\limits_{\mathbb{Z}_{p}}x_{(m)}\left( x-m\right) _{(n)}d\mu _{1}\left(
x\right) =(-1)^{m+n}\frac{\left( m+n\right) !}{m+n+1}.  \label{1BIab}
\end{equation}
\end{theorem}

\begin{proof}
By applying the Volkenborn integral to following well-known identity: 
\begin{equation}
x_{(m+n)}=x_{(m)}\left( x-m\right) _{(n)},  \label{1BIFe}
\end{equation}
we get 
\begin{equation*}
\int\limits_{\mathbb{Z}_{p}}x_{(m)}\left( x-m\right) _{(n)}d\mu _{1}\left(
x\right) =\int\limits_{\mathbb{Z}_{p}}x_{(m+n)}d\mu _{1}\left( x\right) .
\end{equation*}
Combining the above equation with (\ref{ak1}), we get the desired result.
\end{proof}

\begin{theorem}
Let $n\in \mathbb{N}_{0}$. Then we have 
\begin{equation}
\int\limits_{\mathbb{Z}_{p}}x^{\left[ n\right] }d\mu _{1}\left( x\right)
=\sum_{k=0}^{n}t(n,k)B_{k}.  \label{acnum1Tt}
\end{equation}
\end{theorem}

\begin{proof}
By applying the Volkenborn integral to the equation (\ref{acnum1t}), we get 
\begin{equation*}
\int\limits_{\mathbb{Z}_{p}}x^{\left[ n\right] }d\mu _{1}\left( x\right)
=\sum_{k=0}^{n}t(n,k)\int\limits_{\mathbb{Z}_{p}}x^{k}d\mu _{1}\left(
x\right) .
\end{equation*}
Combining the above equation with (\ref{ABw}), we arrive at the desired
result.
\end{proof}

\begin{theorem}
Let $n\in \mathbb{N}$ with $n\geq 2$. Then we have 
\begin{equation*}
\int\limits_{\mathbb{Z}_{p}}x^{2}x^{\left[ n-2\right] }d\mu _{1}\left(
x\right) =\sum_{k=0}^{n}t(n,k)B_{k}+\left( \frac{n-2}{2}\right)
^{2}\sum_{k=0}^{n-2}t(n-2,k)B_{k}.
\end{equation*}
\end{theorem}

\begin{proof}
By applying the Volkenborn integral to the following well-known equation 
\begin{equation}
x^{\left[ n\right] }=\left( x^{2}-\left( \frac{n-2}{2}\right) ^{2}\right)
x^{ \left[ n-2\right] }  \label{ABuTFAL.}
\end{equation}
(\textit{cf}. \cite[p. 11]{Butzer}), we get 
\begin{equation*}
\int\limits_{\mathbb{Z}_{p}}x^{2}x^{\left[ n-2\right] }d\mu _{1}\left(
x\right) =\int\limits_{\mathbb{Z}_{p}}x^{\left[ n\right] }d\mu _{1}\left(
x\right) +\left( \frac{n-2}{2}\right) ^{2}\int\limits_{\mathbb{Z}_{p}}x^{ %
\left[ n-2\right] }d\mu _{1}\left( x\right) .
\end{equation*}
Combining the above equation with (\ref{acnum1Tt}), we arrive at the desired
result.
\end{proof}

\begin{theorem}
Let $n\in \mathbb{N}$. Then we have 
\begin{equation*}
\int\limits_{\mathbb{Z}_{p}}x^{2}\prod\limits_{k=1}^{n-1}\left(
x^{2}-k^{2}\right) d\mu _{1}\left( x\right) =\sum_{k=0}^{2n}t(2n,k)B_{2k}.
\end{equation*}
\end{theorem}

\begin{proof}
By applying the Volkenborn integral to the following well-known equation 
\begin{equation}
x^{\left[ 2n\right] }=x^{2}\prod\limits_{k=1}^{n-1}\left( x^{2}-k^{2}\right)
,  \label{aBuT}
\end{equation}
which is an even function (\textit{cf}. \cite[Eq. (2.1)]{Butzer}), we get 
\begin{equation*}
\int\limits_{\mathbb{Z}_{p}}x^{2}\prod\limits_{k=1}^{n-1}\left(
x^{2}-k^{2}\right) d\mu _{1}\left( x\right) =\int\limits_{\mathbb{Z}_{p}}x^{ %
\left[ 2n\right] }d\mu _{1}\left( x\right) .
\end{equation*}
Combining right-hand side of the above equation with (\ref{acnum1Tt}), we
arrive at the desired result.
\end{proof}

\begin{theorem}
Let $n\in \mathbb{N}_{0}$. Then we have 
\begin{equation*}
\int\limits_{\mathbb{Z}_{p}}x\prod\limits_{k=1}^{n}\left( x^{2}-\frac{\left(
2k-1\right) ^{2}}{4}\right) d\mu _{1}\left( x\right) =-\frac{1}{2}\frac{d}{%
dx }\left\{ x^{\left[ 2n+1\right] }\right\} \left\vert _{x=0}\right.
\end{equation*}
\end{theorem}

\begin{proof}
By applying the Volkenborn integral to the following well-known equation 
\begin{equation}
x^{\left[ 2n+1\right] }=x\prod\limits_{k=1}^{n}\left( x^{2}-\frac{\left(
2k-1\right) ^{2}}{4}\right) ,  \label{aBuT1}
\end{equation}
which is an odd function (\textit{cf}. \cite[Eq. (2.2)]{Butzer}), we get 
\begin{equation}
\int\limits_{\mathbb{Z}_{p}}x^{2}\prod\limits_{k=1}^{n-1}\left(
x^{2}-k^{2}\right) d\mu _{1}\left( x\right) =\int\limits_{\mathbb{Z}_{p}}x^{ %
\left[ 2n+1\right] }d\mu _{1}\left( x\right) .  \label{aBuT2}
\end{equation}
Since the function $x^{\left[ 2n+1\right] }$ is an odd function, combining
right-hand side of the equation (\ref{aBuT2}) with (\ref{ATTek}), we arrive
at the desired result.
\end{proof}

\begin{remark}
By combining (\ref{1BIab}) with (\ref{Y1}) and (\ref{aii3}), we get the
following identities: 
\begin{equation*}
\int\limits_{\mathbb{Z}_{p}}x_{(m)}\left( x-m\right) _{(n)}d\mu _{1}\left(
x\right) =D_{m+n}
\end{equation*}
and 
\begin{equation*}
\int\limits_{\mathbb{Z}_{p}}x_{(m)}\left( x-m\right) _{(n)}d\mu _{1}\left(
x\right) =\sum\limits_{k=0}^{n+m}S_{1}(m+n,k)B_{k}.
\end{equation*}
\end{remark}

\begin{theorem}
Let $n\in 
\mathbb{N}
$. Then we have 
\begin{equation*}
\int\limits_{\mathbb{Z}_{p}}x\binom{x-2}{n-1}d\mu _{1}\left( x\right)
=(-1)^{-n}\sum\limits_{k=1}^{n}\frac{k}{k+1}.
\end{equation*}
\end{theorem}

\begin{proof}
By applying the Volkenborn integral to (\ref{1BIb3}), and using (\ref{C7}),
we get the desired result.
\end{proof}

\begin{theorem}
Let $n\in 
\mathbb{N}
_{0}$. Then we have 
\begin{eqnarray*}
\int\limits_{\mathbb{Z}_{p}}\binom{n-x}{n}d\mu _{-1}\left( x\right)
&=&(-1)^{n}\sum\limits_{k=1}^{n}\frac{1}{k+1} \\
&=&(-1)^{n}\left( H_{n}-H_{0}\right) .
\end{eqnarray*}
\end{theorem}

\begin{proof}
By applying the Volkenborn integral to the above integral, and using (\ref%
{C7}) and (\ref{AHn}), we get the desired result.
\end{proof}

\begin{theorem}
Let $n,r\in \mathbb{N}_{0}$. Then we have 
\begin{equation*}
\int\limits_{\mathbb{Z}_{p}}x^{v}\binom{x}{n}^{r}d\mu _{1}\left( x\right)
=\sum_{k=0}^{nr}\sum_{j=0}^{k}\left( -1\right) ^{j}\binom{k}{j}\binom{k-j}{n}
^{r}\sum_{l=0}^{k}\frac{S_{1}(k,l)B_{v+l}}{k!}.
\end{equation*}
\end{theorem}

\begin{proof}
By applying the Volkenborn integral to (\ref{Id-5}), we get 
\begin{eqnarray*}
\int\limits_{\mathbb{Z}_{p}}x^{v}\binom{x}{n}^{r}d\mu _{1}\left( x\right)
&=&\sum_{k=0}^{nr}\sum_{j=0}^{k}\left( -1\right) ^{j}\binom{k}{j}\binom{k-j}{
n}^{r}\int\limits_{\mathbb{Z}_{p}}x^{v}\binom{x}{k}d\mu _{1}\left( x\right)
\\
&=&\sum_{k=0}^{nr}\sum_{j=0}^{k}\left( -1\right) ^{j}\binom{k}{j}\binom{k-j}{
n}^{r}\int\limits_{\mathbb{Z}_{p}}\frac{x^{v}}{k!}x_{(k)}d\mu _{1}\left(
x\right) .
\end{eqnarray*}
Combining the above equation with (\ref{S1a}), we obtain {\small {\ 
\begin{eqnarray*}
\int\limits_{\mathbb{Z}_{p}}x^{v}\binom{x}{n}^{r}d\mu _{1}\left(
x\right)=\sum_{k=0}^{nr}\sum_{j=0}^{k}\left( -1\right) ^{j}\binom{k}{j}%
\binom{k-j}{ n}^{r}\sum_{l=0}^{k}\frac{S_{1}(k,l)}{k!}\int\limits_{\mathbb{Z}
_{p}}x^{v+l}d\mu _{1}\left( x\right) .
\end{eqnarray*}%
} }{\normalsize {Combining the above equation with (\ref{ABw}), we arrive at
the desired result.} }
\end{proof}

\begin{theorem}
Let $k,n\in 
\mathbb{N}
_{0}$ with $0\leq k\leq n$. Then we have 
\begin{equation*}
\sum\limits_{k=0}^{n}(-1)^{k-n}\int\limits_{\mathbb{Z}_{p}}B_{k}^{n}(x)d\mu
_{1}\left( x\right) =\sum\limits_{j=0}^{n}\binom{n}{j}(-2)^{n-j}B_{n-j}.
\end{equation*}
\end{theorem}

\begin{proof}
By applying the Volkenborn integral to following well-known identity: 
\begin{equation}
\sum\limits_{k=0}^{n}(-1)^{k}B_{k}^{n}(x)=(1-2x)^{n}  \label{A.Berns.2}
\end{equation}
(\textit{cf}. \cite[Theorem 3.4]{mmasFORMULAS}), we get {\small {\ 
\begin{eqnarray}
\sum\limits_{k=0}^{n}(-1)^{k}\int\limits_{\mathbb{Z}_{p}}B_{k}^{n}(x)d\mu
_{1}\left( x\right) &=&\int\limits_{\mathbb{Z}_{p}}(1-2x)^{n}d\mu _{1}\left(
x\right)  \label{A.Berns.8b} \\
&=&\sum\limits_{j=0}^{n}\binom{n}{j}(-2)^{n-j}\int\limits_{\mathbb{Z}%
_{p}}x^{n-j}d\mu _{1}\left( x\right) .  \notag
\end{eqnarray}
} }{\normalsize \ By combining (\ref{A.Berns.8b}) with (\ref{ABw}), we get
the desired result. }
\end{proof}

Combining (\ref{A.Berns.8b}) with(\ref{S2-1a}), we obtain%
\begin{eqnarray}
&&\sum\limits_{k=0}^{n}(-1)^{k}\int\limits_{\mathbb{Z}_{p}}B_{k}^{n}(x)d\mu
_{1}\left( x\right)  \label{A.Berns.8c} \\
&&=\sum\limits_{j=0}^{n}\sum\limits_{m=0}^{n-j}\binom{n}{j}%
(-2)^{n-j}S_{2}(n-j,m)\int \limits_{\mathbb{Z}_{p}}x_{(m)}d\mu _{1}\left(
x\right).  \notag
\end{eqnarray}%
{\normalsize Combining (\ref{A.Berns.8c}) with (\ref{ak1}), we arrive at the
following theorem: }

\begin{theorem}
{\normalsize Let $k,n\in 
\mathbb{N}
_{0}$ with $0\leq k\leq n$. Then we have 
\begin{eqnarray}
&&\sum\limits_{k=0}^{n}(-1)^{k}\int\limits_{\mathbb{Z}_{p}}B_{k}^{n}(x)d\mu
_{1}\left( x\right)  \label{A.Berns.9} \\
&&=\sum\limits_{j=0}^{n}\sum\limits_{m=0}^{n-j}\binom{n}{j}%
(-1)^{n+m-j}2^{n-j}S_{2}(n-j,m) \frac{m!}{m+1}.  \notag
\end{eqnarray}
}
\end{theorem}

{\normalsize Combining (\ref{A.Berns.8c}) with (\ref{Y1}), we arrive at the
following result: }

\begin{corollary}
{\normalsize Let $k,n\in 
\mathbb{N}
_{0}$ with $0\leq k\leq n$. Then we have 
\begin{equation}
\sum\limits_{k=0}^{n}(-1)^{k}\int\limits_{\mathbb{Z}_{p}}B_{k}^{n}(x)d\mu
_{1}\left( x\right) =\sum\limits_{j=0}^{n}\sum\limits_{m=0}^{n-j}\binom{n}{j}%
(-2)^{n-j}S_{2}(n-j,m)D_{m.}  \label{A.Berns.9a}
\end{equation}
}
\end{corollary}

{\normalsize By applying the Volkenborn integral to (\ref{ay3}) and (\ref%
{ay1B}), using (\ref{C7}), (\ref{Y1}) (\ref{ak1}) and (\ref{aii3}), we
arrive at the following results: }

\begin{theorem}
{\normalsize Let $n\in \mathbb{N}_{0}$. Then we have 
\begin{equation}
\int\limits_{\mathbb{Z}_{p}}s_{n}(x;\lambda ,\mu )d\mu _{1}\left( x\right)
=\sum\limits_{v=0}^{n}\binom{n}{v}s_{v}(\lambda ,\mu )D_{n-v}.  \label{1aS1}
\end{equation}
}
\end{theorem}

\begin{theorem}
{\normalsize 
\begin{equation}
\int\limits_{\mathbb{Z}_{p}}s_{n}(x;\lambda ,\mu )d\mu _{1}\left( x\right)
=\sum\limits_{v=0}^{n}(-1)^{n-v}\binom{n}{v}\frac{s_{v}(\lambda ,\mu
)(n-v+1)!}{n-v+1}.  \label{1aS2}
\end{equation}
}
\end{theorem}

\begin{theorem}
{\normalsize 
\begin{equation}
\int\limits_{\mathbb{Z}_{p}}s_{n}(x;\lambda ,\mu )d\mu _{1}\left( x\right)
=\sum\limits_{v=0}^{n}\binom{n}{v}s_{v}(\lambda ,\mu
)\sum_{l=0}^{n-v}S_{1}(n-v,l)B_{l}.  \label{1aS3}
\end{equation}
}
\end{theorem}

{\normalsize Combining (\ref{1aS3}) with (\ref{ak1}), we also have the
following theorem: }

\begin{theorem}
{\normalsize 
\begin{equation}
\sum\limits_{v=0}^{n}\sum\limits_{j=0}^{\mu }\binom{\mu }{j}\binom{n}{v}
\left( \lambda j\right) _{(v)}\int\limits_{\mathbb{Z}_{p}}s_{n-v}\left(
x;\lambda ,\mu \right) d\mu _{1}\left( x\right) =(-1)^{n}\frac{n!}{n+1}.
\label{1aS5}
\end{equation}
}
\end{theorem}

\begin{theorem}
{\normalsize Let $H_{k}\in \mathbb{H}$, the set of harmonic numbers. Let $%
1\leq n\leq k$. Then we have 
\begin{equation*}
\int\limits_{\mathbb{Z}_{p}}\prod\limits_{j=1}^{k}\left( 1+jx\right) d\mu
_{1}\left( x\right) =\sum\limits_{n=0}^{k}k!\binom{H_{k}}{k-n}_{\mathbb{H}
}B_{n},
\end{equation*}
where $\binom{H_{k}}{n}_{\mathbb{H}}$ denotes the harmonic binomial
coefficient. }
\end{theorem}

\begin{proof}
{\normalsize In \cite[Theorem 3.17]{BirghamHarmonic}, Brigham II defined the
following identity: 
\begin{equation}
\prod\limits_{j=1}^{k}\left( 1+jx\right) =\sum\limits_{n=0}^{k}k!\binom{%
H_{k} }{k-n}_{\mathbb{H}}x^{n},  \label{AF5s}
\end{equation}
where $\binom{H_{k}}{n}_{\mathbb{H}}$ denotes the harmonic binomial
coefficient, which given in \cite[Lemma 3.2]{BirghamHarmonic} as follows: 
\begin{equation*}
\binom{H_{k}}{n}_{\mathbb{H}}=\frac{1}{k!}\left[ 
\begin{array}{c}
k+1 \\ 
n+1%
\end{array}
\right]
\end{equation*}
in which $\left[ 
\begin{array}{c}
k \\ 
n%
\end{array}
\right] $ denotes the unsigned Stirling numbers of the first kind. By
applying the Volkenborn integral to (\ref{AF5s}), we get 
\begin{equation*}
\int\limits_{\mathbb{Z}_{p}}\prod\limits_{j=1}^{k}\left( 1+jx\right) d\mu
_{1}\left( x\right) =\sum\limits_{n=0}^{k}k!\binom{H_{k}}{k-n}_{\mathbb{H}
}\int\limits_{\mathbb{Z}_{p}}x^{n}d\mu _{1}\left( x\right) .
\end{equation*}
Combining the above equation with (\ref{ABw}), we arrive at the desired
result. }
\end{proof}

\section{{\protect\normalsize \textbf{New Integral Formulas Involving
Fermionic $p$-adic Integral}}}

{\normalsize In this section, we give some new integral formulas for the
fermionic $p$-adic integral. These new formulas are related to some special
functions, special numbers and polynomials such as rising factorial and the
falling factorial, the Bernoulli numbers and polynomials, the Euler numbers
and polynomials, the Stirling numbers, the Lah numbers, the Peters numbers
and polynomials, the central factorial numbers, the Daehee numbers and
polynomials, the Changhee numbers and polynomials, the Harmonic numbers, the
Fubini numbers, combinatorial numbers and sums. }

{\normalsize By applying the fermionic $p$-adic integral to (\ref{S1a}), and
using (\ref{Mm1}), we get the following theorem: }

\begin{theorem}
{\normalsize Let $m,n\in \mathbb{N}_{0}$. Then we have 
\begin{equation}
\int\limits_{\mathbb{Z}_{p}}x^{m}x_{(n)}d\mu _{-1}\left( x\right)
=\sum\limits_{k=0}^{n}S_{1}(n,k)E_{k+m}.  \label{aS11a}
\end{equation}
}
\end{theorem}

\begin{theorem}
{\normalsize Let $n\in \mathbb{N}_{0}$. Then we have 
\begin{equation}
\int\limits_{\mathbb{Z}_{p}}x^{\left[ n\right] }d\mu _{-1}\left( x\right)
=\sum_{k=0}^{n}t(n,k)E_{k}.  \label{acnum1Ttt}
\end{equation}
}
\end{theorem}

\begin{proof}
{\normalsize By applying the fermionic $p$-adic integral to the equation (%
\ref{acnum1t}), we get 
\begin{equation*}
\int\limits_{\mathbb{Z}_{p}}x^{\left[ n\right] }d\mu _{-1}\left( x\right)
=\sum_{k=0}^{n}t(n,k)\int\limits_{\mathbb{Z}_{p}}x^{k}d\mu _{-1}\left(
x\right) .
\end{equation*}
Combining the above equation with (\ref{Mm1}), we arrive at the desired
result. }
\end{proof}

\begin{theorem}
{\normalsize Let $n\in N$ with $n\geq 2$. Then we have 
\begin{equation*}
\int\limits_{\mathbb{Z}_{p}}x^{2}x^{\left[ n-2\right] }d\mu _{-1}\left(
x\right) =\sum_{k=0}^{n}t(n,k)E_{k}+\left( \frac{n-2}{2}\right)
^{2}\sum_{k=0}^{n-2}t(n-2,k)E_{k}.
\end{equation*}
}
\end{theorem}

\begin{proof}
{\normalsize By applying the fermionic $p$-adic integral to the equation (%
\ref{ABuTFAL.} ), we get 
\begin{equation*}
\int\limits_{\mathbb{Z}_{p}}x^{2}x^{\left[ n-2\right] }d\mu _{-1}\left(
x\right) =\int\limits_{\mathbb{Z}_{p}}x^{\left[ n\right] }d\mu _{-1}\left(
x\right) +\left( \frac{n-2}{2}\right) ^{2}\int\limits_{\mathbb{Z}_{p}}x^{ %
\left[ n-2\right] }d\mu _{-1}\left( x\right) .
\end{equation*}
Combining the above equation with (\ref{acnum1Ttt}), we arrive at the
desired result. }
\end{proof}

\begin{corollary}
{\normalsize Let $n\in \mathbb{N}$. Then we have 
\begin{equation*}
\int\limits_{\mathbb{Z}_{p}}x^{2}\prod\limits_{k=1}^{n-1}\left(
x^{2}-k^{2}\right) d\mu _{-1}\left( x\right) =\sum_{k=0}^{2n}t(2n,k)E_{2k}.
\end{equation*}
}
\end{corollary}

\begin{proof}
{\normalsize By applying the Volkenborn integral to the equation (\ref{aBuT}%
), we get 
\begin{equation*}
\int\limits_{\mathbb{Z}_{p}}x^{2}\prod\limits_{k=1}^{n-1}\left(
x^{2}-k^{2}\right) d\mu _{-1}\left( x\right) =\int\limits_{\mathbb{Z}%
_{p}}x^{ \left[ 2n\right] }d\mu _{-1}\left( x\right) .
\end{equation*}
Combining right-hand side of the above equation with (\ref{acnum1Ttt}), we
arrive at the desired result. }
\end{proof}

\begin{theorem}
{\normalsize Let $m,n\in \mathbb{N}_{0}$. Then we have 
\begin{equation}
\int\limits_{\mathbb{Z}_{p}}x_{(m)}\left( x-m\right) _{(n)}d\mu _{-1}\left(
x\right) =(-1)^{m+n}\frac{\left( m+n\right) !}{2^{m+n}}.  \label{1BIac}
\end{equation}
}
\end{theorem}

\begin{proof}
{\normalsize By applying the fermionic $p$-integral to (\ref{1BIFe}), we
obtain 
\begin{equation*}
\int\limits_{\mathbb{Z}_{p}}x_{(m)}\left( x-m\right) _{(n)}d\mu _{-1}\left(
x\right) =\int\limits_{\mathbb{Z}_{p}}x_{(m+n)}d\mu _{-1}\left( x\right) .
\end{equation*}
Combining the above equation with (\ref{ak2}), we get the desired result. }
\end{proof}

\begin{remark}
{\normalsize Combining (\ref{1BIac}) with (\ref{y1}) and (\ref{ChEuler}), we
arrive at the following identities: 
\begin{equation*}
\int\limits_{\mathbb{Z}_{p}}x_{(m)}\left( x-m\right) _{(n)}d\mu _{-1}\left(
x\right) =Ch_{m+n}
\end{equation*}
and 
\begin{equation*}
\int\limits_{\mathbb{Z}_{p}}x_{(m)}\left( x-m\right) _{(n)}d\mu _{-1}\left(
x\right) =\sum_{k=0}^{n+m}S_{1}(n+m,k)E_{k}.
\end{equation*}
}
\end{remark}

{\normalsize By applying the fermionic $p$-adic integral to (\ref%
{LamdaFun-1c}), and using (\ref{ak2}), we get the following theorem: }

\begin{theorem}
{\normalsize Let $m,n\in \mathbb{N}_{0}$. Then we have 
\begin{equation}
\int\limits_{\mathbb{Z}_{p}}x_{(n)}x_{(m)}d\mu _{-1}\left( x\right)
=\sum\limits_{k=0}^{m}(-1)^{m+n-k}\binom{m}{k}\binom{n}{k}\frac{k!(m+n-k)!}{
2^{m+n-k}}.  \label{1LaH}
\end{equation}
}
\end{theorem}

\begin{theorem}
{\normalsize Let $n,r\in \mathbb{N}_{0}$. Then we have 
\begin{equation*}
\int\limits_{\mathbb{Z}_{p}}x^{v}\binom{x}{n}^{r}d\mu _{-1}\left( x\right)
=\sum_{k=0}^{nr}\sum_{j=0}^{k}\sum_{l=0}^{k}\left( -1\right) ^{j}\binom{k}{j}
\binom{k-j}{n}^{r}\frac{S_{1}(k,l)E_{v+l}}{k!}.
\end{equation*}
}
\end{theorem}

\begin{proof}
{\normalsize By applying the fermionic $p$-adic integral to (\ref{Id-5}), we
get 
\begin{eqnarray*}
\int\limits_{\mathbb{Z}_{p}}x^{v}\binom{x}{n}^{r}d\mu _{-1}\left( x\right)
&=&\sum_{k=0}^{nr}\sum_{j=0}^{k}\left( -1\right) ^{j}\binom{k}{j}\binom{k-j}{
n}^{r}\int\limits_{\mathbb{Z}_{p}}x^{v}\binom{x}{k}d\mu _{-1}\left( x\right)
\\
&=&\sum_{k=0}^{nr}\sum_{j=0}^{k}\left( -1\right) ^{j}\binom{k}{j}\binom{k-j}{
n}^{r}\int\limits_{\mathbb{Z}_{p}}\frac{x^{v}}{k!}x_{(k)}d\mu _{-1}\left(
x\right) .
\end{eqnarray*}
Combining the above equation with (\ref{S1a}), we obtain {\small {\ 
\begin{eqnarray*}
\int\limits_{\mathbb{Z}_{p}}x^{v}\binom{x}{n}^{r}d\mu _{-1}\left( x\right)
=\sum_{k=0}^{nr}\sum_{j=0}^{k}\left( -1\right) ^{j}\binom{k}{j}\binom{k-j}{ n%
}^{r}\sum_{l=0}^{k}\frac{S_{1}(k,l)}{k!}\int\limits_{\mathbb{Z}
_{p}}x^{v+l}d\mu _{-1}\left( x\right) .
\end{eqnarray*}
} Combining the above equation with (\ref{Mm1}), we arrive at the desired
result. }}
\end{proof}

\begin{theorem}
{\normalsize Let $k,n\in 
\mathbb{N}
_{0}$ with $0\leq k\leq n$. Then we have 
\begin{equation}
\sum\limits_{k=0}^{n}(-1)^{k}\int\limits_{\mathbb{Z}_{p}}B_{k}^{n}(x)d\mu
_{-1}\left( x\right) =\sum\limits_{j=0}^{n}\binom{n}{j} (-2)^{n-j}E_{n-j}.
\label{A.Berns.5}
\end{equation}
}
\end{theorem}

\begin{proof}
{\normalsize By applying the fermionic $p$-integral to (\ref{A.Berns.2}), we
obtain {\small {\ 
\begin{eqnarray}
\sum\limits_{k=0}^{n}(-1)^{k}\int\limits_{\mathbb{Z}_{p}}B_{k}^{n}(x)d\mu
_{-1}\left( x\right) &=&\int\limits_{\mathbb{Z}_{p}}(1-2x)^{n}d\mu
_{-1}\left( x\right)  \label{A.Berns.6} \\
&=&\sum\limits_{j=0}^{n} \binom{n}{j} (-2)^{n-j}\int\limits_{\mathbb{Z}%
_{p}}x^{n-j}d\mu _{-1}\left( x\right) .  \notag
\end{eqnarray}%
} Combining the above equation with (\ref{Mm1}), we arrive at the desired
result. }}
\end{proof}

\begin{theorem}
{\normalsize Let $k,n\in 
\mathbb{N}
_{0}$ with $0\leq k\leq n$. Then we have 
\begin{equation*}
\sum\limits_{k=0}^{n}(-1)^{k}\int\limits_{\mathbb{Z}_{p}}B_{k}^{n}(x)d\mu
_{-1}\left( x\right) =\sum\limits_{j=0}^{n}\sum \limits_{m=0}^{n-j}\binom{n}{%
j} (-1)^{m+n-j}2^{n-j-m}S_{2}(n-j,m)m!.
\end{equation*}
}
\end{theorem}

\begin{proof}
{\normalsize Combining (\ref{A.Berns.6}) with (\ref{S2-1a}), we have 
\begin{eqnarray}
&&\sum\limits_{k=0}^{n}(-1)^{k}\int\limits_{\mathbb{Z}_{p}}B_{k}^{n}(x)d\mu
_{-1}\left( x\right)  \label{A.Berns.7} \\
&&=\sum\limits_{j=0}^{n}\sum\limits_{m=0}^{n-j}\binom{n}{j}%
(-2)^{n-j}S_{2}(n-j,m)\int \limits_{\mathbb{Z}_{p}}x_{(m)}d\mu _{-1}\left(
x\right) .  \notag
\end{eqnarray}
Combining (\ref{A.Berns.7}) with (\ref{ak2}), we get the desired result. }
\end{proof}

{\normalsize Combining (\ref{A.Berns.7}) with (\ref{y1}), we arrive at the
following corollary: }

\begin{corollary}
{\normalsize Let $k,n\in 
\mathbb{N}
_{0}$ with $0\leq k\leq n$. Then we have 
\begin{equation}
\sum\limits_{k=0}^{n}(-1)^{k}\int\limits_{\mathbb{Z}_{p}}B_{k}^{n}(x)d\mu
_{-1}\left( x\right) =\sum\limits_{j=0}^{n}\sum\limits_{m=0}^{n-j}\binom{n}{j%
}(-2)^{n-j}S_{2}(n-j,m)Ch_{m}.  \label{A.Berns.8a}
\end{equation}
}
\end{corollary}

{\normalsize By applying the fermionic $p$-integral to (\ref{ay3}) and (\ref%
{ay1B}), using (\ref{ak2}) and (\ref{y1}), we arrive at the following
theorems, respectively: }

\begin{theorem}
{\normalsize Let $n\in \mathbb{N}_{0}$. Then we have 
\begin{equation}
\int\limits_{\mathbb{Z}_{p}}s_{n}(x;\lambda ,\mu )d\mu _{-1}\left( x\right)
=\sum\limits_{v=0}^{n}\binom{n}{v}s_{v}(\lambda ,\mu )Ch_{n-v}.  \label{AF1s}
\end{equation}
}
\end{theorem}

\begin{theorem}
{\normalsize Let $n\in \mathbb{N}_{0}$. Then we have 
\begin{equation}
\int\limits_{\mathbb{Z}_{p}}s_{n}(x;\lambda ,\mu )d\mu _{-1}\left( x\right)
=\sum\limits_{v=0}^{n}(-1)^{n-v}\binom{n}{v}\frac{s_{v}(\lambda ,\mu )(n-v)! 
}{2^{n-v}}.  \label{AF2}
\end{equation}
}
\end{theorem}

\begin{theorem}
{\normalsize Let $n\in \mathbb{N}_{0}$. Then we have {\small {\ 
\begin{equation}
\int\limits_{\mathbb{Z}_{p}}x_{(n)}d\mu _{-1}\left( x\right)
=\sum\limits_{v=0}^{n}\sum\limits_{j=0}^{\mu }\binom{\mu }{j}\binom{n}{v}
\left( \lambda j\right) _{(v)}\int\limits_{\mathbb{Z}_{p}}s_{n-v}\left(
x;\lambda ,\mu \right) d\mu _{-1}\left( x\right) .  \label{AF3s}
\end{equation}
} }}
\end{theorem}

\begin{theorem}
{\normalsize Let $n,\mu \in \mathbb{N}_{0}$. Then we have 
\begin{equation}
\sum\limits_{v=0}^{n}\sum\limits_{j=0}^{\mu }\binom{\mu }{j}\binom{n}{v}
\left( \lambda j\right) _{(v)}\int\limits_{\mathbb{Z}_{p}}s_{n-v}\left(
x;\lambda ,\mu \right) d\mu _{-1}\left( x\right) =(-1)^{n}\frac{n!}{2^{n}}.
\label{AF4s}
\end{equation}
}
\end{theorem}

{\normalsize By applying the fermionic $p$-integral to (\ref{ay1C}), we have%
\begin{equation*}
\sum\limits_{v=0}^{n}\sum\limits_{k=0}^{v}\binom{n}{v}\lambda ^{k}B\left(
k,\mu \right) s\left( v,k\right) \int\limits_{\mathbb{Z}_{p}}s_{n-v}\left(
x;\lambda ,\mu \right) d\mu _{-1}\left( x\right) =\int\limits_{\mathbb{Z}
_{p}}x_{(n)}d\mu _{-1}\left( x\right).
\end{equation*}%
Combining the above equation with (\ref{ak2}) and (\ref{y1}), we obtain the
following results: }

\begin{theorem}
{\normalsize Let $n,v\in \mathbb{N}_{0}$. Then we have 
\begin{equation}
\sum\limits_{v=0}^{n}\sum\limits_{k=0}^{v}\binom{n}{v}\lambda ^{k}B\left(
k,\mu \right) s\left( v,k\right) \int\limits_{\mathbb{Z}_{p}}s_{n-v}\left(
x;\lambda ,\mu \right) d\mu _{-1}\left( x\right) =Ch_{n}.  \label{1aSs1}
\end{equation}
}
\end{theorem}

\begin{theorem}
{\normalsize Let $n,v\in \mathbb{N}_{0}$. Then we have 
\begin{equation}
\sum\limits_{v=0}^{n}\sum\limits_{k=0}^{v}\binom{n}{v}\lambda ^{k}B\left(
k,\mu \right) s\left( v,k\right) \int\limits_{\mathbb{Z}_{p}}s_{n-v}\left(
x;\lambda ,\mu \right) d\mu _{-1}\left( x\right) =(-1)^{n}\frac{n!}{2^{n}}.
\label{1aSs1a}
\end{equation}
}
\end{theorem}

\begin{theorem}
{\normalsize Let $H_{k}\in \mathbb{H}$, the set of harmonic numbers. Let $%
1\leq n\leq k$. Then we have 
\begin{equation*}
\int\limits_{\mathbb{Z}_{p}}\prod\limits_{j=1}^{k}\left( 1+jx\right) d\mu
_{-1}\left( x\right) =\sum\limits_{n=0}^{k}k!\binom{H_{k}}{k-n}_{\mathbb{H}
}E_{n},
\end{equation*}
where $\binom{H_{k}}{n}_{\mathbb{H}}$ denotes the harmonic binomial
coefficient. }
\end{theorem}

\begin{proof}
{\normalsize By applying the fermionic $p$-adic integral to (\ref{AF5s}), we
get 
\begin{equation*}
\int\limits_{\mathbb{Z}_{p}}\prod\limits_{j=1}^{k}\left( 1+jx\right) d\mu
_{-1}\left( x\right) =\sum\limits_{n=0}^{k}k!\binom{H_{k}}{k-n}_{\mathbb{H}
}\int\limits_{\mathbb{Z}_{p}}x^{n}d\mu _{-1}\left( x\right) .
\end{equation*}
Combining the above equation with (\ref{Mm1}), we arrive at the desired
result. }
\end{proof}

\section{{\protect\normalsize \textbf{Identities and Relations}}}

{\normalsize By using the results obtained in the previous sections, we give
some new formulas and relations in this section. These formulas and
relations are involving the Bernoulli numbers and polynomials, the Euler
numbers and polynomials, the Stirling numbers, the Lah numbers, the Peters
numbers and polynomials, the central factorial numbers, the Daehee numbers
and polynomials, the Changhee numbers and polynomials, the Harmonic numbers,
the Fubini numbers, combinatorial numbers and sums. }

\begin{theorem}
{\normalsize Let $l\in \mathbb{N}$ and $n\in \mathbb{N}_{0}$. Then we have 
\begin{equation}
\sum_{j=0}^{n}\binom{n}{j}\frac{B_{j+l}}{j+l}=\sum_{k=1}^{l}(-1)^{l-k}\binom{
l-1}{l-k}\left( \frac{B_{n+k}(1)-B_{0}}{n+k}\right) .  \label{AF6b}
\end{equation}
}
\end{theorem}

\begin{proof}
{\normalsize By applying the Volkenborn integral to the following well-known
combinatorial series identity: 
\begin{equation}
\sum_{j=0}^{n}\binom{n}{j}\frac{x^{j+l}}{j+l}=\sum_{k=1}^{l}(-1)^{l-k}\binom{
l-1}{l-k}\left( \frac{\left( 1+x\right) ^{n+k}-1}{n+k}\right),  \label{AF7c}
\end{equation}
(\textit{cf}. \cite[Eq. (2.4)]{Cohi}), we obtain 
\begin{eqnarray*}
&&\sum_{j=0}^{n}\binom{n}{j}\frac{1}{j+l}\int\limits_{\mathbb{Z}
_{p}}x^{j+l}d\mu _{1}\left( x\right) \\
&=&\sum_{k=1}^{l}(-1)^{l-k}\binom{l-1}{l-k}\frac{1}{n+k}\int\limits_{\mathbb{%
\ Z}_{p}}\left( \left( 1+x\right) ^{n+k}-1\right) d\mu _{1}\left( x\right) .
\end{eqnarray*}
Combining the above equation with (\ref{ABw}) and (\ref{wb}), we arrive at
the desired result. }
\end{proof}

{\normalsize For $l=1$, (\ref{AF6b}) coincides with the following corollary: 
}

\begin{corollary}
{\normalsize Let $n\in \mathbb{N}_{0}$. Then we have 
\begin{equation*}
\sum_{j=0}^{n}\binom{n}{j}\frac{B_{j+1}}{j+1}=\frac{B_{n+1}(1)-B_0}{n+1}.
\end{equation*}
}
\end{corollary}

\begin{theorem}
{\normalsize Let $l\in \mathbb{N}$ and $n\in \mathbb{N}_{0}$. Then we have 
\begin{equation}
\sum_{j=0}^{n}\binom{n}{j}\frac{E_{j+l}}{j+l}=\sum_{k=1}^{l}(-1)^{l-k}\binom{
l-1}{l-k}\left( \frac{E_{n+k}(1)-E_{0}}{n+k}\right) .  \label{Af8e}
\end{equation}
}
\end{theorem}

\begin{proof}
{\normalsize That is, by applying the fermionic $p$-adic integral (\ref{AF7c}%
), we obtain 
\begin{eqnarray*}
&&\sum_{j=0}^{n}\binom{n}{j}\frac{1}{j+l}\int\limits_{\mathbb{Z}
_{p}}x^{j+l}d\mu _{-1}\left( x\right) \\
&=&\sum_{k=1}^{l}(-1)^{l-k}\binom{l-1}{l-k}\frac{1}{n+k}\int\limits_{\mathbb{%
\ Z}_{p}}\left( \left( 1+x\right) ^{n+k}-1\right) d\mu _{-1}\left( x\right) .
\end{eqnarray*}
Combining the above equation with (\ref{Mm1}) and (\ref{we}), we arrive at
the desired result. }
\end{proof}

{\normalsize For $l=1$, (\ref{Af8e}) coincides with the following corollary: 
}

\begin{corollary}
{\normalsize Let $n\in \mathbb{N}_{0}$. Then we have 
\begin{equation*}
\sum_{j=0}^{n}\binom{n}{j}\frac{E_{j+1}}{j+1}=\frac{E_{n+1}(1)-E_{0}}{n+1}.
\end{equation*}
}
\end{corollary}

\begin{remark}
{\normalsize By using (\ref{AF7c}), Choi and Srivastava \cite[Lemma 1]{Cohi}
gave the following summation formulas involving harmonic numbers and
combinatorial series identity: 
\begin{equation*}
\sum_{j=0}^{n}(-1)^{j}\binom{n}{j}\frac{1}{j+l}=\frac{1}{n+1},
\end{equation*}
where $n\in \mathbb{N}_{0}$\ and 
\begin{equation*}
\sum_{j=0}^{n}(-1)^{j+1}\binom{n}{j}\frac{H_{j}}{j+l}=\frac{H_{n}}{n+1},
\end{equation*}
where $n\in \mathbb{N}_{0}$. }
\end{remark}

\begin{theorem}
{\normalsize Let $n,k\in \mathbb{N}_{0}$. Then we have 
\begin{equation}
B_{n}=\sum_{k=0}^{n}\sum_{j=0}^{k}T(n,k)t(j,k)B_{j}.  \label{acnum1TB}
\end{equation}
}
\end{theorem}

\begin{proof}
{\normalsize By applying the Volkenborn integral to the equation (\ref%
{acnum1T}), we get 
\begin{equation*}
\int\limits_{\mathbb{Z}_{p}}x^{n}d\mu _{1}\left( x\right)
=\sum_{k=0}^{n}T(n,k)\int\limits_{\mathbb{Z}_{p}}x^{\left[ k\right] }d\mu
_{1}\left( x\right) .
\end{equation*}
Combining the above equation with (\ref{ABw}) and (\ref{acnum1Tt}), we
arrive at the desired result. }
\end{proof}

\begin{theorem}
{\normalsize Let $n\in \mathbb{N}_{0}$. Then we have 
\begin{equation}
E_{n}=\sum_{k=0}^{n}\sum_{j=0}^{k}T(n,k)t(j,k)E_{j}.  \label{acnum1TtE}
\end{equation}
}
\end{theorem}

\begin{proof}
{\normalsize By applying the Volkenborn integral to the equation (\ref%
{acnum1T}), we get 
\begin{equation*}
\int\limits_{\mathbb{Z}_{p}}x^{n}d\mu _{-1}\left( x\right)
=\sum_{k=0}^{n}T(n,k)\int\limits_{\mathbb{Z}_{p}}x^{\left[ k\right] }d\mu
_{-1}\left( x\right) .
\end{equation*}
Combining the above equation with (\ref{Mm1}) and (\ref{acnum1Ttt}), we
arrive at the desired result. }
\end{proof}

{\normalsize By using (\ref{A.Berns.5}), we have%
\begin{equation*}
\int\limits_{\mathbb{Z}_{p}}B_{0}^{n}(x)d\mu _{-1}\left( x\right)
+\sum\limits_{k=1}^{n}(-1)^{k}\int\limits_{\mathbb{Z}_{p}}B_{k}^{n}(x)d\mu
_{-1}\left( x\right) =\sum\limits_{j=0}^{n}(-2)^{n-j}E_{n-j}.
\end{equation*}%
Combining the above equation with (\ref{A.Berns.3}) and (\ref{A.Berns.4}),
we arrive at the following theorem: }

\begin{theorem}
{\normalsize Let $n\in 
\mathbb{N}
_{0}$. Then we have 
\begin{equation*}
E_{n}=\sum\limits_{j=0}^{n}(-2)^{n-j}E_{n-j}-\sum\limits_{k=1}^{n}(-1)^{k} 
\binom{n}{k}\sum\limits_{j=0}^{n-k}(-1)^{n-k-j}\binom{n-k}{j}E_{n-j}-2.
\end{equation*}
}
\end{theorem}

{\normalsize Combining (\ref{A.Berns.5}) with (\ref{A.Berns.8a}), we get the
following result: }

\begin{theorem}
{\normalsize Let $n\in 
\mathbb{N}
_{0}$. Then we have 
\begin{equation*}
\sum\limits_{j=0}^{n}\binom{n}{j}(-2)^{n-j}\left(
E_{n-j}-\sum\limits_{m=0}^{n-j}S_{2}(n-j,m)Ch_{m}\right) =0.
\end{equation*}
}
\end{theorem}

{\normalsize Combining (\ref{A.Berns.9}) with (\ref{A.Berns.9a}) and (\ref%
{aii3}), we arrive at the following results: }

\begin{theorem}
{\normalsize Let $n\in 
\mathbb{N}
_{0}$. Then we have 
\begin{eqnarray*}
&&\sum\limits_{j=0}^{n}\sum\limits_{m=0}^{n-j}\binom{n}{j}(-1)^{n+m-j}\frac{
2^{n-j}S_{2}(n-j,m)m!}{m+1} \\
&&=\sum\limits_{j=0}^{n}\sum\limits_{m=0}^{n-j}\sum \limits_{l=0}^{m}\binom{n%
}{j}(-2)^{n-j}S_{2}(n-j,m)S_{1}(m,l)B_{l}.
\end{eqnarray*}
}
\end{theorem}

{\normalsize By applying the Volkenborn integral to (\ref{ay1C}), we have%
\begin{equation*}
\sum\limits_{v=0}^{n}\sum\limits_{k=0}^{v}\binom{n}{v}\lambda ^{k}B\left(
k,\mu \right) s\left( v,k\right) \int\limits_{\mathbb{Z}_{p}}s_{n-v}\left(
x;\lambda ,\mu \right) d\mu _{1}\left( x\right) =\int\limits_{\mathbb{Z}
_{p}} x_{(n)}d\mu _{1}\left( x\right) .
\end{equation*}%
Combining the above equation with (\ref{C7}), (\ref{Y1}), (\ref{ak1}), (\ref%
{ak1d}) and (\ref{aii3}), we arrive at the following theorems, respectively: 
}

\begin{theorem}
{\normalsize Let $n,v\in \mathbb{N}_{0}$. Then we have 
\begin{equation*}
D_{n}=\sum\limits_{v=0}^{n}\sum\limits_{k=0}^{v}\binom{n}{v}\lambda
^{k}B\left( k,\mu \right) s\left( v,k\right) \sum\limits_{m=0}^{n-v}\binom{
n-v}{m}s_{m}(\lambda ,\mu )\sum_{l=0}^{n-v-m}S_{1}(n-v-m,l)B_{l}.
\end{equation*}
}
\end{theorem}

\begin{theorem}
{\normalsize Let $n,v\in \mathbb{N}_{0}$. Then we have 
\begin{eqnarray*}
&&\sum\limits_{v=0}^{n}\sum\limits_{k=0}^{v}\binom{n}{v}\lambda ^{k}B\left(
k,\mu \right) s\left( v,k\right) \sum\limits_{m=0}^{n-v}\binom{n-v}{m}
s_{m}(\lambda ,\mu )\sum_{l=0}^{n-v-m}S_{1}(n-v-m,l)B_{l} \\
&=&(-1)^{n}\frac{n!}{n+1}.
\end{eqnarray*}
}
\end{theorem}

\begin{theorem}
{\normalsize Let $n,v\in \mathbb{N}_{0}$. Then we have 
\begin{eqnarray*}
&&\sum\limits_{v=0}^{n}\sum\limits_{k=0}^{v}\binom{n}{v}\lambda ^{k}B\left(
k,\mu \right) s\left( v,k\right) \sum\limits_{m=0}^{n-v}\binom{n-v}{m}
s_{m}(\lambda ,\mu )\sum_{l=0}^{n-v-m}S_{1}(n-v-m,l)B_{l} \\
&=&\sum\limits_{v=0}^{n}S_{1}(n,v)B_{l}.
\end{eqnarray*}
}
\end{theorem}

{\normalsize By combining (\ref{1aS1}), (\ref{1aS2}), (\ref{1aS3}) and (\ref%
{1aS5}) with (\ref{ak1}), (\ref{ak1d}) and (\ref{aii3}), we get the
following results: }

\begin{theorem}
{\normalsize Let $n,\mu \in \mathbb{N}_{0}$. Then we have 
\begin{equation*}
\sum\limits_{v=0}^{n}\sum\limits_{j=0}^{\mu }\binom{\mu }{j}\binom{n}{v}
\left( \lambda j\right) _{(v)}\sum\limits_{l=0}^{n-v}\binom{n-v}{l}
s_{l}(\lambda ,\mu )D_{n-v-l}=(-1)^{n}\frac{n!}{n+1}.
\end{equation*}
}
\end{theorem}

\begin{theorem}
{\normalsize Let $n,\mu \in \mathbb{N}_{0}$. Then we have 
\begin{eqnarray*}
&&\sum\limits_{v=0}^{n}\sum\limits_{j=0}^{\mu }\binom{\mu }{j}\binom{n}{v}
\left( \lambda j\right) _{(v)}\sum\limits_{l=0}^{n-v}\binom{n-v}{l}
s_{l}(\lambda ,\mu ) \\
&&\times \sum\limits_{m=0}^{n-v-l}\binom{n-v-l}{m}s_{m}(\lambda ,\mu
)\sum_{k=0}^{m}S_{1}(m,k)B_{k} \\
&=&(-1)^{n}\frac{n!}{n+1}.
\end{eqnarray*}
}
\end{theorem}

\begin{theorem}
{\normalsize Let $n,\mu \in \mathbb{N}_{0}$. Then we have 
\begin{eqnarray*}
&&\sum\limits_{v=0}^{n}\sum\limits_{j=0}^{\mu }\binom{\mu }{j}\binom{n}{v}
\left( \lambda j\right) _{(v)}\sum\limits_{l=0}^{n-v}(-1)^{n-v-l}\binom{n-v}{%
l} \frac{s_{l}(\lambda ,\mu )(n-v)!}{n-v-l+1} \\
&=&(-1)^{n}\frac{n!}{n+1}.
\end{eqnarray*}
}
\end{theorem}

{\normalsize Combining (\ref{1aSs1}) and (\ref{1aSs1a}), we arrive at the
following results: }

\begin{theorem}
{\normalsize Let $n,v\in \mathbb{N}_{0}$. Then we have {\small {\ 
\begin{equation}
Ch_{n}=\sum\limits_{v=0}^{n}\sum\limits_{k=0}^{v}\binom{n}{v}\lambda
^{k}B\left( k,\mu \right) s\left( v,k\right) \sum\limits_{m=0}^{n-v}\binom{
n-v}{m}s_{m}(\lambda ,\mu )Ch_{n-v-m}.  \label{yy3}
\end{equation}
} }}
\end{theorem}

{\normalsize Combining (\ref{yy1}) with (\ref{yy3}), we arrive at the
following corollary: }

\begin{corollary}
{\normalsize Let $n,v\in \mathbb{N}_{0}$. Then we have 
\begin{eqnarray*}
&&\sum\limits_{v=0}^{n}\sum\limits_{k=0}^{v}\binom{n}{v}\lambda ^{k}B\left(
k,\mu \right) s\left( v,k\right) \sum\limits_{m=0}^{n-v}(-1)^{n-v-m}\binom{
n-v}{m}\frac{s_{m}(\lambda ,\mu )\left( n-v-m\right) !}{2^{n-v-m}} \\
&=&(-1)^{n}\frac{n!}{2^{n}}.
\end{eqnarray*}
}
\end{corollary}

{\normalsize After comparing and combining the equation (\ref{AF1s}) with
the equation (\ref{AF4s}), and making the necessary algebraic operations, we
obtain the following results, respectively: }

\begin{theorem}
{\normalsize Let $n,\mu \in \mathbb{N}_{0}$. Then we have 
\begin{equation*}
\sum\limits_{v=0}^{n}\sum\limits_{j=0}^{\mu }\binom{\mu }{j}\binom{n}{v}
\left( \lambda j\right) _{(v)}\sum\limits_{l=0}^{n-v}\binom{n-v}{l}
s_{l}(\lambda ,\mu )Ch_{n-v-l}=(-1)^{n}\frac{n!}{2^{n}}.
\end{equation*}
}
\end{theorem}

\begin{theorem}
{\normalsize Let $n,\mu \in \mathbb{N}_{0}$. Then we have 
\begin{equation}
Ch_{n}=\sum\limits_{v=0}^{n}\sum\limits_{j=0}^{\mu }\binom{\mu }{j}\binom{n}{
v}\left( \lambda j\right) _{(v)}\sum\limits_{l=0}^{n-v}(-1)^{n-v-l}\binom{n-v%
}{ l}\frac{s_{l}(\lambda ,\mu )}{2^{n-v-l}}.  \label{yy2}
\end{equation}
}
\end{theorem}

{\normalsize Combining (\ref{yy1}) with (\ref{yy2}), we arrive at the
following corollary: }

\begin{corollary}
{\normalsize Let $n,\mu \in \mathbb{N}_{0}$. Then we have 
\begin{equation*}
\sum\limits_{v=0}^{n}\sum\limits_{j=0}^{\mu }\binom{\mu }{j}\binom{n}{v}
\left( \lambda j\right) _{(v)}\sum\limits_{l=0}^{n-v}(-1)^{n-v-l}\binom{n-v}{%
l} \frac{s_{l}(\lambda ,\mu )}{2^{n-v-l}}=(-1)^{n}\frac{n!}{2^{n}}.
\end{equation*}
}
\end{corollary}

\begin{theorem}
{\normalsize Let $n\in \mathbb{N}_{0}$. Then we have 
\begin{equation*}
\sum\limits_{j=0}^{n}\left( 
\begin{array}{c}
n \\ 
j%
\end{array}
\right) \lambda ^{n-j}Y_{j,2}\left( \lambda \right)
Ch_{n-j}=\sum\limits_{j=0}^{n}\left( -1\right) ^{n}j!(n-j)!\left( 
\begin{array}{c}
n \\ 
j%
\end{array}
\right) \frac{\lambda ^{n+j}}{2^{n}\left( \lambda -1\right) ^{j+1}}.
\end{equation*}
}
\end{theorem}

\begin{proof}
{\normalsize By applying the fermionic $p$-adic integral to (\ref{A1}) and (%
\ref{A3}), we have the following relations, respectively: 
\begin{equation*}
\int_{\mathbb{Z}_{p}}Y_{n,2}\left( x;\lambda \right) d\mu
_{-1}(x)=\sum\limits_{j=0}^{n}\left( 
\begin{array}{c}
n \\ 
j%
\end{array}
\right) \lambda ^{n-j}Y_{j,2}\left( \lambda \right) \int_{\mathbb{Z}
_{p}}x_{(n-j)}d\mu _{-1}(x)
\end{equation*}
and 
\begin{equation*}
\int_{\mathbb{Z}_{p}}Y_{n,2}\left( x;\lambda \right) d\mu
_{-1}(x)=2\sum\limits_{j=0}^{n}\left( -1\right) ^{j}j!\left( 
\begin{array}{c}
n \\ 
j%
\end{array}
\right) \frac{\lambda ^{n+j}}{\left( 2\lambda -2\right) ^{j+1}}\int_{\mathbb{%
\ Z}_{p}}x_{(n-j)}d\mu _{-1}(x).
\end{equation*}
Combining the above equations with (\ref{y1}) and (\ref{ak2}), we get: 
\begin{equation}
\int_{\mathbb{Z}_{p}}Y_{n,2}\left( x;\lambda \right) d\mu
_{1}(x)=\sum\limits_{j=0}^{n}\left( 
\begin{array}{c}
n \\ 
j%
\end{array}
\right) \lambda ^{n-j}Y_{j,2}\left( \lambda \right) Ch_{n-j}  \label{aii3Yc}
\end{equation}
and 
\begin{equation}
\int_{\mathbb{Z}_{p}}Y_{n,2}\left( x;\lambda \right) d\mu
_{1}(x)=\sum\limits_{j=0}^{n}\left( -1\right) ^{n}j!(n-j)!\left( 
\begin{array}{c}
n \\ 
j%
\end{array}
\right) \frac{\lambda ^{n+j}}{2^{n}\left( \lambda -1\right) ^{j+1}}.
\label{aii3Y1c}
\end{equation}
Combining (\ref{aii3Yc}) with (\ref{aii3Y1c}), we arrive at the desired
result. }
\end{proof}

\begin{theorem}
{\normalsize Let $n\in \mathbb{N}_{0}$. Then we have 
\begin{equation*}
\sum\limits_{j=0}^{n}\left( 
\begin{array}{c}
n \\ 
j%
\end{array}
\right) \lambda ^{n-j}Y_{j,2}\left( \lambda \right)
D_{n-j}=2\sum\limits_{j=0}^{n}\sum_{l=0}^{n-j}\left( -1\right) ^{j}j!\left( 
\begin{array}{c}
n \\ 
j%
\end{array}
\right) \frac{\lambda ^{n+j}S_{1}(n-j,l)B_{l}}{\left( 2\lambda -2\right)
^{j+1}}.
\end{equation*}
}
\end{theorem}

\begin{proof}
{\normalsize By applying the Volkenborn integral to (\ref{A1}) and (\ref{A3}%
), we have the following relations, respectively: 
\begin{equation*}
\int_{\mathbb{Z}_{p}}Y_{n,2}\left( x;\lambda \right) d\mu
_{1}(x)=\sum\limits_{j=0}^{n}\left( 
\begin{array}{c}
n \\ 
j%
\end{array}
\right) \lambda ^{n-j}Y_{j,2}\left( \lambda \right) \int_{\mathbb{Z}
_{p}}x_{(n-j)}d\mu _{1}(x)
\end{equation*}
and 
\begin{equation*}
\int_{\mathbb{Z}_{p}}Y_{n,2}\left( x;\lambda \right) d\mu
_{1}(x)=2\sum\limits_{j=0}^{n}\left( -1\right) ^{j}j!\left( 
\begin{array}{c}
n \\ 
j%
\end{array}
\right) \frac{\lambda ^{n+j}}{\left( 2\lambda -2\right) ^{j+1}}\int_{\mathbb{%
\ Z}_{p}}x_{(n-j)}d\mu _{1}(x).
\end{equation*}
Combining the above equations with (\ref{Y1}) and (\ref{aii3}), we get 
\begin{equation}
\int_{\mathbb{Z}_{p}}Y_{n,2}\left( x;\lambda \right) d\mu
_{1}(x)=\sum\limits_{j=0}^{n}\left( 
\begin{array}{c}
n \\ 
j%
\end{array}
\right) \lambda ^{n-j}Y_{j,2}\left( \lambda \right) D_{n-j}  \label{aii3Y}
\end{equation}
and {\small {\ 
\begin{equation}
\int_{\mathbb{Z}_{p}}Y_{n,2}\left( x;\lambda \right) d\mu
_{1}(x)=2\sum\limits_{j=0}^{n}\sum_{l=0}^{n-j}\left( -1\right) ^{j}j!\left( 
\begin{array}{c}
n \\ 
j%
\end{array}
\right) \frac{\lambda ^{n+j}S_{1}(n-j,l)B_{l}}{\left( 2\lambda -2\right)
^{j+1}}.  \label{aii3Y1}
\end{equation}%
} Combining (\ref{aii3Y}) with (\ref{aii3Y1}), we arrive at the desired
result. }}
\end{proof}

\section{{\protect\normalsize \textbf{New Sequences Containing Bernoulli
Numbers and Euler Numbers}}}

{\normalsize In this section, we examine $p$-adic integrals of the function%
\begin{equation*}
J(x)=x_{(n)}x^{(m)}.
\end{equation*}%
Moreover, we give some applications of these integrals. With the help of the
integrals of this special function $J(x)$, we define two new sequences
containing the Bernoulli numbers of the first kind and the Euler numbers of
the first kind, respectively. We give some properties of these two
sequences. We also prove that the general term of these sequences can be
written in terms of the central factorial numbers. We also give some
identities and relations involving the Bernoulli numbers, the Euler numbers,
the stirling numbers, the Lah numbers, and the central factorial numbers. }

{\normalsize Let's start this section with the following questions: }

{\normalsize How can we compute the following integrals: }

{\normalsize Question 1:%
\begin{equation*}
\int\limits_{\mathbb{Z}_{p}}J(x)d\mu _{1}\left( x\right) =?
\end{equation*}
}

{\normalsize Question 2:%
\begin{equation*}
\int\limits_{\mathbb{Z}_{p}}J(x)d\mu _{-1}\left( x\right) =?
\end{equation*}
}

{\normalsize By using (\ref{LahLAH}), we have the following identity:%
\begin{equation}
x_{(n)}x^{(m)}=\sum_{k=1}^{m}\left\vert L(m,k)\right\vert x_{(k)}x_{(n)}.
\label{ALaH}
\end{equation}
}

{\normalsize By applying the Volkenborn integral to (\ref{ALaH}), we get%
\begin{equation*}
\int\limits_{\mathbb{Z}_{p}}x_{(n)}x^{(m)}d\mu _{1}\left( x\right)
=\sum_{k=1}^{m}\left\vert L(m,k)\right\vert \int\limits_{\mathbb{Z}
_{p}}x_{(k)}x_{(n)}d\mu _{1}\left( x\right) .
\end{equation*}%
Combining the above equation with (\ref{1LaH}), we arrive at the following
theorem. The result of the following theorem gives us the solution of the
Question 1. }

\begin{theorem}
{\normalsize Let $m,n\in \mathbb{N}_{0}$. Then we have 
\begin{equation*}
\int\limits_{\mathbb{Z}_{p}}x_{(n)}x^{(m)}d\mu _{1}\left( x\right)
=\sum_{k=1}^{m}\sum\limits_{j=0}^{n}(-1)^{k+n-j}\binom{m}{j}\binom{k}{j} 
\frac{j!(n+k-j)!\left\vert L(m,k)\right\vert }{m+k-j+1}.
\end{equation*}
}
\end{theorem}

{\normalsize By applying the fermionic $p$-adic integral to (\ref{ALaH}), we
get%
\begin{equation*}
\int\limits_{\mathbb{Z}_{p}}x_{(n)}x^{(m)}d\mu _{-1}\left( x\right)
=\sum_{k=1}^{m}\left\vert L(m,k)\right\vert \int\limits_{\mathbb{Z}
_{p}}x_{(k)}x_{(n)}d\mu _{-1}\left( x\right) .
\end{equation*}%
Combining the above equation with (\ref{1LaHv}), we arrive at the following
theorem. The result of the following theorem gives us the solution of the
Question 2. }

\begin{theorem}
{\normalsize Let $m,n\in \mathbb{N}_{0}$. Then we have 
\begin{equation*}
\int\limits_{\mathbb{Z}_{p}}x_{(n)}x^{(m)}d\mu _{-1}\left( x\right)
=\sum_{k=1}^{m}\sum\limits_{j=0}^{n}(-1)^{n+k-j}\binom{n}{j}\binom{k}{j} 
\frac{j!(n+k-j)!\left\vert L(m,k)\right\vert }{2^{n+k-j}}.
\end{equation*}
}
\end{theorem}

{\normalsize Substituting $m=n$ into Question 1 and Question 2, we define
the following sequences containing the Bernoulli numbers of the first kind
and the Euler numbers of the first kind, respectively: }{\small {\ 
\begin{eqnarray}
\mathcal{Y}(n,B)&=&\int\limits_{\mathbb{Z}_{p}}x_{(n)}x^{(n)}d\mu _{1}\left(
x\right)  \label{Ad-1} \\
&=&\int\limits_{\mathbb{Z}_{p}}x^{2}(x^{2}-1)(x-2^{2})(x^{2}-3^{2})\cdots
(x^{2}-(n-1)^{2})d\mu _{1}\left( x\right)  \notag
\end{eqnarray}
} }{\normalsize and }{\small {\ 
\begin{eqnarray}
\mathcal{Y}(n,E)&=&\int\limits_{\mathbb{Z}_{p}}x_{(n)}x^{(n)}d\mu
_{-1}\left( x\right)  \label{Ad-2} \\
&=&\int\limits_{\mathbb{Z}_{p}}x^{2}(x^{2}-1)(x-2^{2})(x^{2}-3^{2}) \cdots
(x^{2}-(n-1)^{2})d\mu _{-1}\left( x\right) .  \notag
\end{eqnarray}%
} }

{\normalsize By using (\ref{ABw}) and (\ref{Mm1}), we compute few values of
the sequences given by (\ref{Ad-1}) and (\ref{Ad-2}), respectively, as
follows: 
\begin{eqnarray*}
\mathcal{Y}(0,B) &=&B_{0}, \\
\mathcal{Y}(1,B) &=&B_{2}, \\
\mathcal{Y}(2,B) &=&B_{4}-B_{2}, \\
\mathcal{Y}(3,B) &=&B_{6}-5B_{4}+4B_{2}, \\
\mathcal{Y}(4,B) &=&B_{8}-14B_{6}+49B_{4}-36B_{2}, \\
\mathcal{Y}(5,B) &=&B_{10}-30B_{8}+273B_{6}-870B_{4}+576B_{2} \\
\mathcal{Y}%
(6,B)&=&B_{12}-55B_{10}+1023B_{8}-7645B_{6}+21076B_{4}-14400B_{2}, \ldots
\end{eqnarray*}%
and%
\begin{eqnarray*}
\mathcal{Y}(0,E) &=&E_{0}, \\
\mathcal{Y}(1,E) &=&E_{2}, \\
\mathcal{Y}(2,E) &=&E_{4}-E_{2}, \\
\mathcal{Y}(3,E) &=&E_{6}-5E_{4}+4E_{2}, \\
\mathcal{Y}(4,E) &=&E_{8}-14E_{6}+49E_{4}-36E_{2}, \\
\mathcal{Y}(5,E) &=&E_{10}-30E_{8}+273E_{6}-870E_{4}+576E_{2} \\
\mathcal{Y}%
(6,E)&=&E_{12}-55E_{10}+1023E_{8}-7645E_{6}+21076E_{4}-14400E_{2}, \ldots
\end{eqnarray*}%
}

{\normalsize When the integrals, given by (\ref{Ad-1}) and (\ref{Ad-2}), are
calculated for the special values of the number $n$, we can observe that the
row numbers given in the matrix representation of the central factorial
numbers $t(i,j)$ in the equation (\ref{AmatCt}) and the coefficients of the
Bernoulli numbers of the first kind and the Euler of the first kind are
equal. Therefore, we arrive at the following theorems: }

\begin{theorem}
{\normalsize Let $n\in \mathbb{N}_{0}$. Then we have 
\begin{equation*}
\mathcal{Y}(n,B)=\sum_{k=1}^{n}t(2n,2k)B_{2k}.
\end{equation*}
}
\end{theorem}

\begin{proof}
{\normalsize By applying the Volkenborn integral to the following well-known
equation 
\begin{equation}
x^{2}(x^{2}-1)(x-2^{2})(x^{2}-3^{2})\cdots
(x^{2}-(n-1)^{2})=\sum\limits_{k=1}^{n}t(2n,2k)x^{2k}  \label{aABUT}
\end{equation}
(\textit{cf}. \cite[p. 430]{Butzer}, \cite{Kim2018PJMS}), we get 
\begin{equation*}
\int\limits_{\mathbb{Z}_{p}}x^{2}(x^{2}-1)(x-2^{2})(x^{2}-3^{2})\cdots
(x^{2}-(n-1)^{2})d\mu _{1}\left( x\right)
=\sum\limits_{k=1}^{n}t(2n,2k)\int\limits_{\mathbb{Z}_{p}}x^{2k}d\mu
_{1}\left( x\right) .
\end{equation*}
Combining the above equation with (\ref{ABw}), we arrive at the desired
result. }
\end{proof}

\begin{theorem}
{\normalsize Let $n\in \mathbb{N}$. Then we have 
\begin{eqnarray*}
\mathcal{Y}(n,E) &=&\sum_{k=1}^{2n}t(2n,2k)E_{2k} \\
&=&0.
\end{eqnarray*}
}
\end{theorem}

\begin{proof}
{\normalsize By applying the fermionic $p$-integral to the equation (\textit{%
\ref{aABUT}} ), we get 
\begin{equation*}
\int\limits_{\mathbb{Z}_{p}}x^{2}(x^{2}-1)(x^{2}-2^{2})(x^{2}-3^{2})\cdots
(x^{2}-(n-1)^{2})d\mu _{-1}\left( x\right)
=\sum\limits_{k=1}^{n}t(2n,2k)\int\limits_{\mathbb{Z}_{p}}x^{2k}d\mu
_{-1}\left( x\right) ,
\end{equation*}
(\textit{cf}. \cite[p. 430]{Butzer}, \cite{Kim2018PJMS}). Combining the
above equation with (\ref{Mm1}), we arrive at the desired result. }
\end{proof}

\begin{remark}
{\normalsize In \cite{simsekRJMP.mtjpam}, we defined two other kinds of
sequences including Bernoulli numbers and polynomials and Euler numbers and
polynomials. Let's briefly give information about two of them: The sequence $%
(Y_{1}(n:B))$ is associated with the Bernoulli numbers. That is, $%
Y_{1}(0:B)=B_{0}=1$, $Y_{1}(1:B)=B_{1}=-\frac{1}{2}$, $Y_{1}(2:B)=B_{2}-B_{1}
$, $Y_{1}(3:B)=B_{3}-3B_{2}+2B_{1}$,\ldots . If we continue to calculate the
terms of the sequence $(Y_{1}(n:B))$ in this way, the general term of this
sequence is given by the following formula including the Daehee numbers: 
\begin{equation}
Y_{1}(n:B)=D_{n}.  \label{IR}
\end{equation}
The sequence $(y_{2}(n:E))\ $is associated with the Euler numbers of the
first kind. That is, $y_{1}(0:E)=y_{2}(0:E)=E_{0}=1$ and $%
y_{1}(1:E)=y_{2}(1:E)=E_{1}=-\frac{1}{2}$, $y_{1}(2:E)=E_{2}-E_{1}$, $%
y_{1}(3:E)=E_{3}-3E_{2}+2E_{1}$,\ldots . Similarly, if we continue to
calculate the terms of the sequence $(y_{2}(n:E))$ in this way, the general
term of this sequence is given by the following formula including the
Changhee numbers: 
\begin{equation*}
y_{1}(n:E)=Ch_{n}.
\end{equation*}
In this paper, we do not consider whether there is any relationship between
the sequences $Y_{1}(n:B)$ and the sequence $y_{1}(n:E)$ and the newly
defined the sequence $\mathcal{Y}(n,B)$ and the sequence $\mathcal{Y}(n,E)$.
Perhaps the sequence $\mathcal{Y}(n,B)$ and the sequence $\mathcal{Y}(n,E)$
may be subsequences of the sequences $Y_{1}(n:B)$ and the sequence $%
y_{1}(n:E)$, respectively. }
\end{remark}

{\normalsize The following theorem gives us that Bernoulli numbers of the
first kind can be computed with the help of the central factorial numbers of
the second kind $T(n,k)$ and the sequence $\mathcal{Y}(k,B)$. }

\begin{theorem}
{\normalsize Let $n\in \mathbb{N}_{0}$. Then we have 
\begin{equation*}
B_{2n}=\sum_{k=0}^{n}T(n,k)\mathcal{Y}(k,B).
\end{equation*}
}
\end{theorem}

\begin{proof}
{\normalsize By applying the Volkenborn integral to the following well-known
equation 
\begin{equation}
x^{n}=\sum_{k=0}^{n}T(n,k)x(x-1)(x-2^{2})(x-3^{2})\cdots (x-(n-1)^{2}),
\label{NN1}
\end{equation}
(\textit{cf}. \cite[p. 430]{Butzer}, \cite{Kim2018PJMS}). By replacing $x$
by $x^{2}$ into (\ref{NN1}), we have 
\begin{equation*}
x^{2n}=\sum_{k=0}^{n}T(n,k)x^{2}(x^{2}-1)(x^{2}-2^{2})(x^{2}-3^{2})\cdots
(x^{2}-(n-1)^{2}).
\end{equation*}
By applying the Volkenborn integral to the above equation, we get {\small {\ 
\begin{equation*}
\int\limits_{\mathbb{Z}_{p}}x^{2n}d\mu _{1}\left( x\right)
=\sum_{k=0}^{n}T(n,k)\int\limits_{\mathbb{Z}%
_{p}}x^{2}(x^{2}-1)(x^{2}-2^{2})(x^{2}-3^{2}) \cdots (x^{2}-(n-1)^{2})d\mu
_{1}\left( x\right) .
\end{equation*}%
} Combining the above equation with (\ref{ABw}) and (\ref{Ad-1}), we arrive
at the desired result. }}
\end{proof}

{\normalsize The following theorem gives us that Euler numbers of the first
kind can be computed with the help of the central factorial numbers of the
second kind $T(n,k)$ and the sequence $\mathcal{Y}(k,E)$. }

\begin{theorem}
{\normalsize Let $n\in \mathbb{N}_{0}$. Then we have 
\begin{equation*}
E_{2n}=\sum_{k=0}^{n}T(n,k)\mathcal{Y}(k,E).
\end{equation*}
}
\end{theorem}

\begin{proof}
{\normalsize By applying the Volkenborn integral to the following well-known
equation 
\begin{equation}
x^{n}=\sum_{k=0}^{n}T(n,k)x(x-1)(x-2^{2})(x-3^{2})\cdots (x-(n-1)^{2}),
\label{acnum}
\end{equation}
}

{\normalsize By replacing $x$ by $x^{2}$ into (\ref{acnum}), we have 
\begin{equation*}
x^{2n}=\sum_{k=0}^{n}T(n,k)x^{2}(x^{2}-1)(x^{2}-2^{2})(x^{2}-3^{2})\cdots
(x^{2}-(n-1)^{2}).
\end{equation*}
By applying the fermionic $p$-adic integral to the above equation, we get 
{\small {\ 
\begin{equation*}
\int\limits_{\mathbb{Z}_{p}}x^{2n}d\mu _{-1}\left( x\right)
=\sum_{k=0}^{n}T(n,k)\int\limits_{\mathbb{Z}%
_{p}}x^{2}(x^{2}-1)(x^{2}-2^{2})(x^{2}-3^{2}) \cdots (x^{2}-(n-1)^{2})d\mu
_{-1}\left( x\right) .
\end{equation*}%
} Combining the above equation with (\ref{Mm1}) and (\ref{Ad-2}), we arrive
at the desired result. }}
\end{proof}

{\normalsize We now give another solutions of Question 1 and Question 2 in
theorems stated below. }

{\normalsize Combining (\ref{Ad-1}) with (\ref{S1a}) and (\ref{LahLAH}), we
get%
\begin{equation*}
\mathcal{Y}(n,B)=\sum_{j=0}^{n}\sum_{k=1}^{n}S_{1}(n,j)\left\vert
L(n,k)\right\vert \int\limits_{\mathbb{Z}_{p}}x^{j}x_{(k)}d\mu _{1}\left(
x\right).
\end{equation*}%
Combining the above equation with (\ref{aS1B}), we arrive at the following
theorem: }

\begin{theorem}
{\normalsize Let $n\in \mathbb{N}_{0}$. Then we have 
\begin{equation*}
\mathcal{Y}(n,B)=\sum_{j=0}^{n}\sum_{k=1}^{n}\sum
\limits_{m=0}^{k}S_{1}(n,j)S_{1}(k,m)B_{j+m}\left\vert L(n,k)\right\vert.
\end{equation*}
}
\end{theorem}

{\normalsize Combining (\ref{Ad-2}) with (\ref{S1a}) and (\ref{LahLAH}), we
get%
\begin{equation*}
\mathcal{Y}(n,E)=\sum_{j=0}^{n}\sum_{k=1}^{n}S_{1}(n,j)\left\vert
L(n,k)\right\vert \int\limits_{\mathbb{Z}_{p}}x^{j}x_{(k)}d\mu _{-1}\left(
x\right).
\end{equation*}%
Combining the above equation with (\ref{aS11a}), we arrive at the following
theorem: }

\begin{theorem}
{\normalsize Let $n\in \mathbb{N}_{0}$. Then we have 
\begin{equation*}
\mathcal{Y}(n,E)=\sum_{j=0}^{n}\sum_{k=1}^{n}\sum
\limits_{m=0}^{k}S_{1}(n,j)S_{1}(k,m)E_{j+m}\left\vert L(n,k)\right\vert.
\end{equation*}
}
\end{theorem}

\end{document}